\documentclass[journal]{amsart}
\usepackage{amsmath,amsfonts,amsthm,bm}
\usepackage{fullpage}
\usepackage{algorithmic}
\usepackage{algorithm}
\usepackage{array}
\usepackage[caption=false,font=normalsize,labelfont=sf,textfont=sf]{subfig}
\usepackage{textcomp}
\usepackage{stfloats}
\usepackage{url}
\usepackage{verbatim}
\usepackage{graphicx}
\usepackage{cite}
\usepackage{epstopdf}
\usepackage{subcaption}
\usepackage{subfig}
\usepackage{pgfplots}
\usepackage{tikz}
\usetikzlibrary{shapes,calc,positioning,arrows}
\usepackage{hyperref}

\pgfarrowsdeclarecombine{twotriang}{twotriang}
{stealth'}{stealth'}{stealth'}{stealth'} 

\newcommand{\mybinom}[2]{\Bigl(\begin{array}{@{}c@{}}#1\\#2\end{array}\Bigr)}
\newcommand{\pr}{\mbox{\sf P}}
\newcommand{\ex}{{\mathbb{E}}}               

\newcommand{\cov}{\mbox{\sf Cov}}

\newcommand{\Real}{\mathbb R}

\newcommand{\bx}{{\bm x}}               

\newcommand{\g}{\lambda}                
   
\newcommand{\ra}{\rightarrow}           

\newtheorem{thm}{Theorem}
\newtheorem{lem}{Lemma}
\newtheorem{pro}{Proposition}
\newtheorem{cor}{Corollary}

\newtheorem{rem}{Remark}

\usepackage{color}

\begin{document}
	
	\title[Power Iteration Algorithm]{Analysis of Power Iteration Algorithm with Partially Observed Matrix-vector Products}
	
	\author[S. Ghosh {\it et al}]{Soumyadip Ghosh, Lior Horesh, Vassilis Kalantzis, Yingdong Lu,\\ Tomasz Nowicki, and Shashanka Ubaru\\\\IBM Research, Thomas J. Watson Research Center, Yorktown Heights, NY 10598, U.S.A.
	}

	\maketitle
	
	\begin{abstract}
		We consider the problem of computing the dominant eigenvector of a symmetric matrix via the power iteration algorithm subject to constraints in the computation of matrix-vector pr    ucts. In particular, we focus on scenarios where the entries of matrix-vector products with the input matrix are only partially observed. Such constraints frequently arise on cloud architectures implemented via the controller-worker model where the matrix-vector products are distributed across workers on remote servers. 
		Instead of a prolonged delay incurred by waiting for the slowest workers to return their output to the controller, a phenomenon known as straggling, a set of pre-determined values can replace the values of the delayed workers and allow the power iteration  to proceed to the next iteration. In this paper, we develop two algorithms whose expected approximation converges to the true dominant eigenvector. 
		The first algorithm relies on a probabilistic switch 
		between two different approaches to set the omitted entries: either set them to zero or to their previous recorded value. The second algorithm relies on averaging previously generated partial power iteration approximations obtained by ignoring a set of columns of the iteration matrix. several theoretical details are discussed while numerical experiments verify the effectiveness of the two proposed schemes and demonstrate their comparative performance advantage over current state-of-the-art. 
	\end{abstract}

	\section{Introduction}
	
	Power iteration is a numerical algorithm that computes the eigenvector and associated (simple) eigenvalue of 
	largest modulus of a $N\times N$ matrix $\mathbf{A}$ by applying a sequence of matrix-vector multiplications 
	between the matrix $\mathbf{A}$ and an initial approximation $\bm{x}_0 \in \mathbb{R}^N$ of the dominant eigenvector~\cite{saad2011numerical,parlett1998symmetric}. The numerical approximation at iteration $k\in 
	\mathbb{N}$ is expressed as $\bm{x}_k = \mathbf{A}\bm{x}_{k-1}/\|\mathbf{A}\bm{x}_{k-1}\|$ and each iteration enhances the alignment of $\bm{x}_k$ with the dominant eigenvector of the matrix $\mathbf{A}$. 
	The power iteration algorithm has maintained its appeal as a workhorse of numerical linear algebra over several decades due to its 
	simplicity, and has several practical applications, e.g., graph signal processing \cite{sandryhaila2013discrete,teke2018asynchronous}, subspace tracking \cite{badeau2005fast}, independent component analysis \cite{basiri2017alternative}, determining the eigenvector
	centrality of graphs~\cite{bonacich2007some}, computing the stationary distributions of Markov chains~\cite{nesterov2015finding}, trust assignment in P2P networks~\cite{jelasity2007asynchronous}, and importance ranking for hyperlinked web pages~\cite{langville2004deeper}. 
	Recent advances in randomized numerical linear algebra, e.g., \cite{martinsson2020randomized,halko2011finding}, have also revived the interest in the power iteration algorithm and its extensions~\cite{bose2019terapca,musco2015randomized,gu2015subspace,rokhlin2010randomized}. 
	
	In this paper, we focus on the application of the power iteration algorithm in scenarios where the
	matrix-vector products $\bm{x}_k =\mathbf{A}\bm{x}_{k-1}$ are only partially observed. In particular, 
	we consider that at any given iteration $k$ we can observe either: $a$) a subset of $1\leq T_k\leq N$ entries of the matrix-vector product $\mathbf{A}\bm{x}_{k-1}$ or $b$) an approximation of $\mathbf{A}\bm{x}_{k-1}$ formed by ignoring all but $T_k$ columns of the matrix $\mathbf{A}$. The value of $T_k$ as well as the corresponding index set of 
	observed entries/columns ${\mathcal T}_k\subseteq \{1,2,\ldots,N\}$ are modeled as random variables. 
	Setting “$a$)"represents an instance of inexact power iteration and has been analyzed in great length from an asynchronous perspective \cite{teke2018asynchronous,teke2019random,teke2020signals}. 
	Other types of inexact power iteration include partial matrix-vector products in quantum many-body problems \cite{lu2020full}, subspampling of the entries of $\mathbf{A}$ in adaptive power method \cite{shin2023adaptive}, and stochastic matrix-vector products \cite{xu2018accelerated}. The work 
		in \cite{11196482} focuses on non-uniform probability scenarios where each index of the approximate dominant eigenvector is updated according to an independent (and, possibly, time-varying) probability, and at any given iteration, the update decision is guided by a Bernoulli trial with success probability equal to the update probability.
	Finally, a detailed study on the behavior of inexact power iteration (subspace iteration) can be found in 
	\cite{saad2016analysis}.
	
	The study of power iteration with partially observed matrix-vector products (also referred to as 
	\emph{partial power iteration}) is motivated by the phenomenon of \emph{straggling} associated 
	with cloud computing infrastructures 
	when the matrix-vector products with the matrix $\mathbf{A}$ are computed in parallel on remote processing elements (workers) \cite{hussain2019sla,kalantzis2025straggler,11196482}. Straggling workers refer to those processing elements that 
	communicate their results significantly slower than their peers and thus delay the general flow of 
	computations \cite{wang2015using}. Common causes of straggling are the heterogeneous nature of cloud servers 
	as well as the oversubscription model enforced on non-dedicated servers \cite{jacquet2024scroogevm}. 
	Aside straggling, another motivating application for the work presented in this paper is out-of-core execution for 
	matrices that are too large to fit in the system memory \cite{zhou2012out,kalantzis2023rayleigh,kalantzis2025single}. In such cases, performing a single step of power iteration requires multiple passes over secondary 
	memory, with each pass fetching a different block of rows (columns) of the matrix $\mathbf{A}$. To mitigate excessive memory transfer costs, an alternative is to perform an inexact matrix-vector product by loading only $T\ll N$ rows (columns) of $\mathbf{A}$ with the value of $T$ depending on the amount of available system memory \cite{kalantzis2021projection}. 
	
	When the matrix $\mathbf{A}$ is stored by rows, partial power iteration requires a mechanism to set a value for the 
	$N-T_k$ non-observed entries of $\bm{x}_k =\mathbf{A}\bm{x}_{k-1}$ so that it 
	proceeds to the next iteration without delays. In this work  
	we distinguish among two different options. The first option sets the $N-T_k$ 
	non-observed entries of $\bm{x}_k$ equal to zero, \emph{i.e.}, $[\bm{x}_k]_i=0,\ i\notin 
	{\mathcal T}_k$. The latter implies that $\bm{x}_k$ has at most $T_k$ non-zero entries. 
	While special\footnote{For example, adjacency matrices of star graphs or random sparse 
		matrices with heavy tails.} matrices that exhibit approximately sparse dominant 
	eigenvectors can be constructed, general symmetric matrices have dense\footnote{For 
		random matrices, the work in \cite{tao2011random} suggests that the coefficients of 
		eigenvectors tend to be spread out (delocalized) rather than concentrated in a few 
		entries.} eigenvectors, and thus setting the entries of $\bm{x}_k$ to zero limits its 
	approximation accuracy. Nonetheless, as we demonstrate in this paper, this option is 
	equivalent -in expectation- to the approximation $\bm{x}_k$ resulting after $k$ steps 
	of the classical power iteration algorithm with exact matrix-vector products. The second 
	option sets the $i$-th non-observed entries of $\bm{x}_k,\ [\bm{x}_k]_i$, equal to the 
	corresponding entry in the previous approximation $\bm{x}_{k-1}$, \emph{i.e.}, 
	$[\bm{x}_k]_i = [\bm{x}_{k-1}]_i,\ i\notin {\mathcal T}_k$. In this case, the resulting 
	inexact power iteration algorithm is mathematically equivalent to asynchronous power 
	iteration without normalization and is known to converge to the dominant eigenvector~\cite{teke2018asynchronous}. As we demonstrate in Section~\ref{sec:varr}, setting $[\bm{x}_k]_i=0,\ i\notin 
		{\mathcal T}_k$ results in larger variance than setting $[\bm{x}_k]_i = [\bm{x}_{k-1}]_i,\ i\notin {\mathcal T}_k$, which makes the former more attractive during the early stage of an algorithm when large improvements are desired but less so in the later phase when more accurate updates are necessary.

	\subsection{Contributions}
	We present two new algorithms that converge to the dominant eigenvector of a symmetric matrix $\mathbf{A}$ with a simple eigenvalue of largest modulus when the row subset ${\mathcal T}_k$ 
	is sampled with equal probability from the space of all $T_k$-size row subsets of $\{1,2\ldots,N\}$. The contributions made in this paper can be summarized as follows:
	\begin{itemize}
		\item  The first algorithm consists of partial power iteration with a probabilistic mechanism to choose between the options $[\bm{x}_k]_i=0,\ i\notin {\mathcal T}_k$ and 
		$[\bm{x}_k]_i=[\bm{x}_{k-1}]_i,\ i\notin {\mathcal T}_k$. The intuition behind this approach is to as follows: leverage the first option to eliminate the contribution of non-dominant eigenvectors during the early stages followed by the second option to maintain nearly accurate eigencomponents as the iterative process progresses.
		\item The second algorithm approximates the dominant eigenvector of the matrix $\mathbf{A}$ via a weighted linear combination of all previous approximations. In contrast to the first algorithm, the matrix $\mathbf{A}$ is now accessed by columns. We show that the sequence of approximate eigenvectors converges to the dominant eigenvector direction of the matrix $\mathbf{A}$ and this algorithm provides higher accuracy than any alternative during the early phases of partial power iteration.
		\item We analyze the first and second orders statistics introduced by the options $[\bm{x}_k]_i=0,\ i\notin {\mathcal T}_k$ and $[\bm{x}_k]_i=[\bm{x}_{k-1}]_i,\ i\notin {\mathcal T}_k$. Our analysis on the convergence and quality of asynchronous variants, presented in Proposition~\ref{pro:mean-squre-convergence}, refines and improves previous results, see e.g.~\cite{teke2018asynchronous}. In addition, we conduct a detailed comparative analysis on the variances of different probabilistic mechanisms.
	\end{itemize}
	Although not explored in this paper, we note that the proposed algorithms can be used to compute more than one dominant eigenvector via deflation, \emph{i.e.}, by operating on the orthogonal complement of computed eigenvectors \cite{chapman1997deflated}.
	
	\subsection{Organization and notation} 
	
	Section \ref{sec:back} 
	presents background information on the power iteration algorithm and an asynchronous variant. Section \ref{sec:newm} presents and analyzes partial power iteration 
	when the non-observed entries of the matrix-vector products with the matrix $\mathbf{A}$ 
	are set equal to zero. Section \ref{sec:varr} presents results on the variance of two 
	different models to update partial power iteration. Section~\ref{sec:switch_algo} describes an algorithm that probabilistically switches between the two models. Section \ref{sec:alg2} presents 
	an averaging algorithm to approximate the dominant eigenvector of the matrix $\mathbf{A}$ by 
	a linear combination of all previous approximations. Section \ref{sec:experiments} presents numerical experiments on dense model problems.
	Finally, Section \ref{sec:conclusions} presents our concluding remarks.
	
	Lowercase bold letters, $\bm{a}$, denote vectors, and uppercase bold letters, $\mathbf{A}$, denote matrices. 
	The $ij$-th entry of the matrix $\mathbf{A}$ is denoted by $\mathbf{A}_{ij}$, while the $ij$-th entry of the 
	matrix power $\mathbf{A}^k,\ k\in \mathbb{N}$, is denoted by $[\mathbf{A}^k]_{ij}$. Similarly, the $i$-th entry of a vector $\bm{x}$ will be denoted as $[\bm{x}]_i$. The norm $\|\cdot\|$ refers to the Euclidean norm, \emph{i.e.}, 
	$\|\bm{x}\|=(\sum_{i=1}^N [\bm{x}]_i^2)^\frac12$. The term $\mathbb{E}[\cdot]$ denotes the expectation operator. 
	The symbol $\mathbb{N}$ denotes the set of positive integers.

	\section{Background: the power iteration algorithm} \label{sec:back}
	
	In this section we provide background information on the power iteration algorithm and 
	its
	asynchronous variants. Unless mentioned otherwise (\emph{i.e.}, as in Section \ref{sec:alg2}), we assume that at any given iteration $k$ we can observe only a subset of $1\leq T_k\leq N$ entries of the matrix-vector product $\mathbf{A}\bm{x}_{k-1}$. 
	
	\subsection{The classical power iteration algorithm}
	
	In this section, we provide a summary of the classical power iteration algorithm for the symmetric matrix case. For a detailed discussion on classical power iteration, we refer readers to 
	\cite{saad2011numerical}. Let $\mathbf{A}\in \mathbb{R}^{N\times N}$ be symmetric with $(\lambda_j,\bm{v}_j)$, $||v_j||=1$,  
	denoting its corresponding $j$-th eigenpair such that $\lambda_1 > |\lambda_2| \geq \ldots \geq |\lambda_N|$, \emph{i.e.}, the dominant eigenvalue $\lambda_1$ is simple. Without loss of generality, let us assume that  $\g_1=1$. The matrix $\mathbf{A}$ can be written in a compact form as $\mathbf{A}=\mathbf{V}\mathbf{\Lambda} \mathbf{V}^\top$ where 
	$\mathbf{\Lambda} = \mathrm{diag}(\lambda_1,\ldots,\lambda_N)$ and $\mathbf{V}=[\bm{v}_1,\ldots,\bm{v}_N]$ 
	with orthonormal $\bm{v}_i$'s.
	
	Consider now the $k$-th matrix power of $\mathbf{A}$ expressed as the sum of $N$ outer products:
	\begin{align*}
		\mathbf{A}^k = \mathbf{V}\mathbf{\Lambda}^k \mathbf{V}^\top = \sum_{j=1}^N \lambda_j^k \bm{v}_j \bm{v}_j^\top.
	\end{align*}
	Let $\bm{x}_0\in \mathbb{R}^N$ be an initial approximation of $\bm{v}_1$ chosen at random such that\footnote{For example, any distribution with continuous support, such as a Gaussian random vector, satisfies this condition with probability one \cite{kalantzis2021fast}.}  
	$\bm{v}_1^{\top}\bm{x}_0 \neq 0$. The matrix-vector product $\mathbf{A}^k\bm{x}_0$ can be then written as
	\begin{align*}
		\bm{x}_{k} & = \mathbf{A}^k\bm{x}_0 = \sum_{j=1}^N \lambda_j^k \bm{v}_j (\bm{v}_j^\top \bm{x}_0)
		= \lambda_1^k \sum_{j=1}^N \frac{\lambda_j^k}{\lambda_1^k} \bm{v}_j (\bm{v}_j^\top \bm{x}_0).
	\end{align*}
	Recalling that $|\lambda_1| > |\lambda_j|,\ j=2,\ldots,N$, all the ratios $\frac{\lambda_j^k}{\lambda_1^k}$
	converge to zero as $k\rightarrow \infty$, and thus $A^k\bm{x}_0$ becomes increasingly parallel to
	the dominant eigendirection $\bm{v}_1$. More specifically, the rate
	in which
	the eigendirection $\bm{v}_j$ is annihilated from $\bm{x}_{k}$ is determined by the ratio $\frac{\lambda_j^k}{\lambda_1^k}$, with larger
	gaps between $|\lambda_j|$ and $|\lambda_1|$ leading to faster annihilation of the eigendirection
	$\bm{v}_j$. Since the eigendirection that is annihilated with the slowest rate is $v_2$, the rate of convergence of power iteration is 
	dictated by the ratio $\frac{|\lambda_2|}{|\lambda_1|}$. 
	
	A practical implementation of the power iteration algorithm starts with $k=1$ and expresses the approximation of $\bm{v}_1$ at the $k$-th iteration as $\bm{x}_{k} = A\bm{x}_{k-1}/\|A\bm{x}_{k-1}\|$. The procedure
	is typically terminated once two successive approximations $\bm{x}_{k-1}$ and $\bm{x}_{k}$ become parallel
	in floating-point arithmetic. The power iteration algorithm can be extended to the simultaneous 
	computation of more than one eigenpairs, an algorithmic extension known as subspace iteration \cite{saad2011numerical}.

	\subsection{Asynchronous power iteration} \label{sec:async}
	
	In this section, we provide a brief summary on the asynchronous power iteration presented
	in~\cite{teke2019random,teke2018asynchronous}. For a detailed discussion, we refer the
	reader to~\cite[Section 2]{teke2020signals}, while background information on asynchronous 
	computing and numerical methods can be found in \cite{bertsekas2015parallel}. 
	The asynchronous power iteration does not perform any scaling of the vector iterates and 
	thus requires $|\lambda_1|=1$ which we assume throughout this section. 

	Let $\bm{\widehat{x}}_{k-1}\in \mathbb{R}^N,\ k>1$, denote the current approximation of 
	$\bm{v}_1$ by asynchronous power iteration where $\bm{\widehat{x}}_0\in \mathbb{R}^N$ is chosen 
	such that $\bm{v}_1^\top\bm{\widehat{x}}_0 \neq 0$. Moreover, let $T_{k}\in [1,N]$ and denote 
	by ${\mathcal T}_{k}$ a subset of $T_{k}$ indices drawn from the set $\{1,2,\ldots,N\}$ 
	without replacement \cite{teke2018asynchronous}. Asynchronous power iteration computes $\bm{\widehat{x}}_{k}$ via by the formula\footnote{This formula can be seen as an instance of the model described in \cite{lubachevsky1986chaotic}.}
	\begin{equation}\label{updzero2}
		[\bm{\widehat{x}}_{k}]_i = 
		\begin{cases}
			\sum_{j=1}^{j=N} \mathbf{A}_{ij}[\bm{\widehat{x}}_{k-1}]_j  &\ \mathrm{if}\ i\in {\mathcal T}_{k}\\
			[\bm{\widehat{x}}_{k-1}]_i &\ \mathrm{if}\ i\notin {\mathcal T}_{k}
		\end{cases}.
	\end{equation} 
	In words, the $i$-th component of $\bm{\widehat{x}}_{k}$ is updated if and only if 
	$i\in {\mathcal T}_{k}$, otherwise $[\bm{\widehat{x}}_{k}]_i=[\bm{\widehat{x}}_{k-1}]_i$. 
	The deterministic number of updated indices $T_{k}$ can be seen as the $k$-th 
	sample of a random integer $T\in \mathbb{N}$ admitting values from the closed 
	interval $[\gamma_1,\gamma_2] \subseteq [1,N]$ where $\gamma_1,\gamma_2 \in 
	\mathbb{N}$ and $\gamma_1\leq \gamma_2$. Likewise, the deterministic sample of 
	updated indices ${\mathcal T}_{k}$ is a sample of the random variable ${\mathcal T}$ that 
	denotes the space of all possible size-$T$ subsets of integers from $\{1,2,\ldots,N\}$ 
	without replacement. In other words $T_k$ is the size of the set $\mathcal{T}_k$.

	\begin{rem}
		While our focus in this paper does not lie in asynchronous computing, the asynchronous 
		power iteration algorithm is mathematically equivalent to (non-normalized) power 
		iteration with partial matrix-vector products under the condition that the omitted 
		entries of the matrix-vector product $\mathbf{A}\bm{\widehat{x}}_{k-1}$ are set equal 
		to the corresponding entries of the vector $\bm{\widehat{x}}_{k-1}$. 
		For this reason, 
		asynchronous power iteration will serve as our baseline algorithm throughout the rest 
		of this paper.
	\end{rem}
	
	Let now $\bm{e}_i\in \{0,1\}^N$ denote the $i$-th column of the identity matrix  
	of size $N$. The matrix 
	\begin{equation*}
		\mathbf{D}_{\mathcal{T}_{k}}=\sum_{i\in {\mathcal T}_{k}} \bm{e}_i\bm{e}_i^\top    
	\end{equation*}
	is diagonal with ones on every row $i\in {\mathcal T}_{k}$ and zero otherwise. Then, the 
	update in (\ref{updzero2}) can be compactly written through the matrix-vector product:
	\begin{align*}
		\bm{\widehat{x}}_{k}=[\mathbf{I}+\mathbf{D}_{\mathcal{T}_{k}}(\mathbf{A}-\mathbf{I})] \bm{\widehat{x}}_{k-1}.
	\end{align*} 
	The iterative process defined by (\ref{updzero2}) converges -in expectation- to $\bm{v}_1$  
	when classical power iteration does, \emph{i.e.}, when $\lambda_2,\ldots,\lambda_N \in (-1,1)$. 
	Moreover, in contrast to classical power iteration without normalization, $\bm{\widehat{x}}_{k}$ 
	might converge to $\bm{v}_1$ even when a non-unit eigenvalue lies on or outside the unit circle. 
	
	Assume now that both $T$ and ${\mathcal T}$ are sampled so that the samples $T_{k}$ and ${\mathcal T}_{k}$ 
	appear with equal probability, \emph{i.e.}, $T_{k}$ can take any value in $[\gamma_1,\gamma_2]\subseteq [1,N]$ with probability $1/(\gamma_2-\gamma_1+1)$, and, likewise, each of the ${\mybinom{N}{T_{k}}}$ possible row sets of $\mathbf{A}$ is chosen with probability equal to the reciprocal of ${\mybinom{N}{T_{k}}}$. In \cite{teke2019random}, it was shown that the expected value of the iterate $\bm{\widehat{x}}_{k}$ can be then compactly written as 
	\begin{align*}
		\mathbb{E}[\bm{\widehat{x}}_{k}] & = \mathbb{E}[[\mathbf{I}+\mathbf{D}_{\mathcal{T}_{k}}(\mathbf{A}-\mathbf{I})]\cdots [\mathbf{I}+\mathbf{D}_{\mathcal{T}_{1}}(\mathbf{A}-\mathbf{I})]]\bm{\widehat{x}}_0 \\ 
		& = \left[\frac{\mathbb{E}[T]}{N}\mathbf{A} + \frac{N-\mathbb{E}[T]}{N}\mathbf{I}\right]^k \bm{\widehat{x}}_0.
	\end{align*}
	Therefore, applying $k$ steps of asynchronous power iteration can be seen as a process that applies -in expectation- the $k$-th power of the  matrix $\frac{\mathbb{E}[T]}{N}\mathbf{A} + \frac{N-\mathbb{E}[T]}{N}\mathbf{I}$ onto the 
	initial vector $\bm{\widehat{x}}_0$. The eigenvalues of this matrix are 
	equal to $1+\frac{\mathbb{E}[T]}{N}(\lambda_j-1),\ j=1,\ldots,N$, while the expectation of the inner product between $\bm{\widehat{x}}_{k}$ and $\bm{v}_j,\ j=1,2,\ldots,N$, is equal to (e.g., see Theorem 1 in~\cite[Thm. 1]{teke2018asynchronous}) 
	\begin{equation}
		\label{eqn:vjxk}
		\mathbb{E}[\bm{v}_j^\top\bm{\widehat{x}}_{k}] = \left[1+\frac{\mathbb{E}[T]}{N}(\lambda_j-1)\right]^k \bm{v}_j^\top\bm{\widehat{x}}_{0}.
	\end{equation}
	Furthermore, the convergence of the algorithm can be characterized as follows. 
	\begin{pro}
		\label{pro:mean-squre-convergence} 
		For $\bm{{\widehat x}}_{k}$ defined by~\eqref{updzero2}, we have, for $k \in \mathbb{N}$:  
		\begin{equation*}
			\varrho_*\ex\|\bm{{\widehat x}}_{k-1}-(\bm{{\widehat x}}^\top_0 \bm{v}_1)\bm{v}_1\|^2\nonumber \le \ex\|\bm{{\widehat x}}_{k}-(\bm{{\widehat x}}^\top_0 \bm{v}_1)\bm{v}_1\|^2 \le \varrho^*\ex\|\bm{{\widehat x}}_{k-1}-(\bm{{\widehat x}}^\top_0 \bm{v}_1)\bm{v}_1\|^2\ , \label{eqn:decay_statement}
		\end{equation*}
				with 
				\begin{align*}
					\varrho_*:=&\min\limits_{i\ge 2}\left[\left(1-\frac{\ex[T]}{N}\right)+\frac{\ex[T]}{N}\lambda_i^2 \right]
					=\left[\left(1-\frac{\ex[T]}{N}\right)+\frac{\ex[T]}{N}\lambda_N^2 \right],
					\\ \varrho^*:=&\max\limits_{i\ge 2}\left[\left(1-\frac{\ex[T]}{N}\right)+\frac{\ex[T]}{N}\lambda_i^2 \right]=\left[\left(1-\frac{\ex[T]}{N}\right)+\frac{\ex[T]}{N}\lambda_2^2 \right].
				\end{align*}
			\end{pro}
			\begin{proof}
				Following~\eqref{eqn:vjxk}, we have 
				\begin{equation*}
					\ex\|\bm{{\widehat x}}_{k}-(\bm{{\widehat x}}^\top_0 \bm{v}_1)\bm{v}_1\|^2=\ex[\|\bm{{\widehat x}}_{k}\|^2-(\bm{{\widehat x}}^\top_0 \bm{v}_1)^2].
				\end{equation*}
				Thus, it suffices to show that
				\begin{align}
					\varrho_* [\ex[\|\bm{{\widehat x}}_{k-1}\|^2-(\bm{{\widehat x}}^\top_0 \bm{v}_1)^2] \le & \ex[\|\bm{{\widehat x}}_{k}\|^2]-(\bm{{\widehat x}}^\top_0 \bm{v}_1)^2 \le \varrho^* [\ex[\|\bm{{\widehat x}}_{k-1}\|^2-(\bm{{\widehat x}}^\top_0 \bm{v}_1)^2]\,. \label{eqn:decay}
				\end{align}
				Recall now that $\mathbf{A}$ is a symmetric matrix with the eigenpairs $(\lambda_i,\bm{v}_i)$, $i=1,\ldots, N$,  where the eigenvectors $\bm{v}_1,\ldots,\bm{v}_N$, form an orthonormal basis with $\mathbf{A}\bm{v}_i=\lambda_i \bm{v}_i$, $\bm{v}_i^\top \bm{v}_j=0$ for $i\not=j$, $\bm{v}_i^\top \bm{v}_i=1$ and $\lambda_i$ satisfy and  $|\lambda_i|<\lambda_1=1$. Thus, we can write $\bm{{\widehat x}}_{k-1}=\sum_{i=1}^N b_i \bm{v}_i$, hence $\|\bm{{\widehat x}}_{k-1}\|^2 =\sum_{i=1}^N b_i^2$. Then, with $\ex_{k}$ denotes the expectation taken with respect to the random matrix at step $k$, we have,  
				\begin{align*}
					\ex[\bm{{\widehat x}}_{k}^\top \bm{{\widehat x}}_{k}] & =  \ex[ \bm{{\widehat x}}_{k-1}^\top \ex_{k}[I+\mathbf{D}_{\mathcal{T}}(A-I)]^2 \bm{{\widehat x}}_{k-1}]
					\\ &= \ex\left[ \left(\sum_{i=1}^N b_i \bm{v}_i\right)^\top \left[\left(1-\frac{\ex[T]}{N}\right)\mathbf{I}+\frac{\ex[T]}{N} \mathbf{A}^2\right]\left(\sum_{i=1}^N b_i \bm{v}_i\right)\right]
					\\ &=
					\ex\left[b^2_1+\sum_{i=2}^N b_i^2 \left[\left(1-\frac{\ex[T]}{N}\right)+\frac{\ex[T]}{N}\lambda_i^2\right]\right],
				\end{align*}
				where the last equality is the result of orthogonality. It is easy to see that $b_1= \bm{{\widehat x}}^\top_0 \bm{v}_1$. Inequality~\eqref{eqn:decay} then follows directly. 
			\end{proof}
			\begin{rem}\leavevmode
				\begin{itemize}
					\item Proposition~\ref{pro:mean-squre-convergence} concludes that $\bm{{\widehat x}}_{k}$ converges exponentially to $\left(\bm{{\widehat x}}^\top_0 \bm{v}_1\right)\bm{v}_1$ in mean-square norm. This in turn also implies all weaker convergences, that is convergence in expectation, probability and distribution, but almost sure convergence does not necessarily hold. Furthermore, the lower and upper bounds in Proposition~\ref{pro:mean-squre-convergence} can be achieved if $\bm{{\widehat x}}_k$ takes the values of $\bm{v}_N$ and $\bm{v}_2$, respectively. 
					\item
					In~\cite{teke2018asynchronous}, similar but more involved bounds are obtained (see Theorem 2 therein), 
					while 
					the associated exponential convergence proof requires an additional condition  on the non-unit eigenvalues of the matrix $\mathbf{A}$.
					\item
					Proposition~\ref{pro:mean-squre-convergence} can be generalized with the specific distributional assumption on ${\mathcal T}_{k}$ relaxed to be just independent and $T_k$ bounded by $\gamma_1$ and $\gamma_2$. 
				\end{itemize}
			\end{rem}

			\subsection{A limitation of the asynchronous matrix-vector model}
			
			Following the discussion in the previous section, applying $k$ steps of the asynchronous 
			power iteration transforms -in expectation- the spectrum $\lambda_1,\ldots,\lambda_N$ of 
			the matrix $\mathbf{A}$ to the spectrum $1+\frac{\mathbb{E}[T]}{N}\left(\lambda_j-1\right)$ of 
			the matrix $\frac{\mathbb{E}[T]}{N}\mathbf{A} + \frac{N-\mathbb{E}[T]}{N}\mathbf{I}$. 
			The linear transformation $1+\frac{\mathbb{E}[T]}{N}\left(\lambda_j-1\right)$ can be rewritten as 
			$\frac{\mathbb{E}[T]}{N}\lambda_j + \frac{N-\mathbb{E}[T]}{N}$ from which it becomes evident 
			that each eigenvalue $\lambda_j$ is multiplied by the positive ratio $\frac{\mathbb{E}[T]}{N}$ 
			followed by a shift with the positive scalar $\frac{N-\mathbb{E}[T]}{N}$. More specifically, 
			the multiplication with the factor $\frac{\mathbb{E}[T]}{N}$ reduces the modulus of $\lambda_j$ 
			while, at the same time, the addition of $\frac{N-\mathbb{E}[T]}{N}$ shifts the eigenvalues 
			rightwards. Thus, with the exception of $\lambda_1$, when $\lambda_j \in (0,1)$, small values 
			of $\mathbb{E}[T]$ cluster $\frac{\mathbb{E}[T]}{N}\lambda_j + \frac{N-\mathbb{E}[T]}{N}$ 
			closer to one, thus reducing the spectral gap between the dominant and trailing eigenvalues.
			
			
			This drawback can be understood better by considering the extreme case where 
			$\lambda_j=0,\ j=2,\ldots,N$. The rank of the matrix $\mathbf{A}$ is equal to one, and 
			just one step of classical power iteration will compute the dominant eigenvector 
			$\bm{v}_1$. On the other hand, asynchronous power iteration will -in expectation- 
			apply power iteration to the matrix $\frac{\mathbb{E}[T]}{N}\mathbf{A} + \frac{N-\mathbb{E}[T]}{N}\mathbf{I}$ 
			that has $N$ non-zero eigenvalues for any choice $\mathbb{E}[T] < N$. Moreover, 
			all $N-1$ zero eigenvalues of $\mathbf{A}$ are identically mapped to an eigenvalue of 
			modulus $\frac{N-\mathbb{E}[T]}{N}$. The spectral gap of the matrix $\frac{\mathbb{E}[T]}{N}\mathbf{A} + \frac{N-\mathbb{E}[T]}{N}\mathbf{I}$ now becomes equal to $\frac{N-\mathbb{E}[T]}{N-\mathbb{E}[T] + \mathbb{E}[T]}$, which approaches 1  as $N$ grows and $\mathbb{E}[T]$ is bounded or grows much slower than $N$. 
			
			\begin{figure*}[ht]
				\centering
				\includegraphics[width=0.48\linewidth]{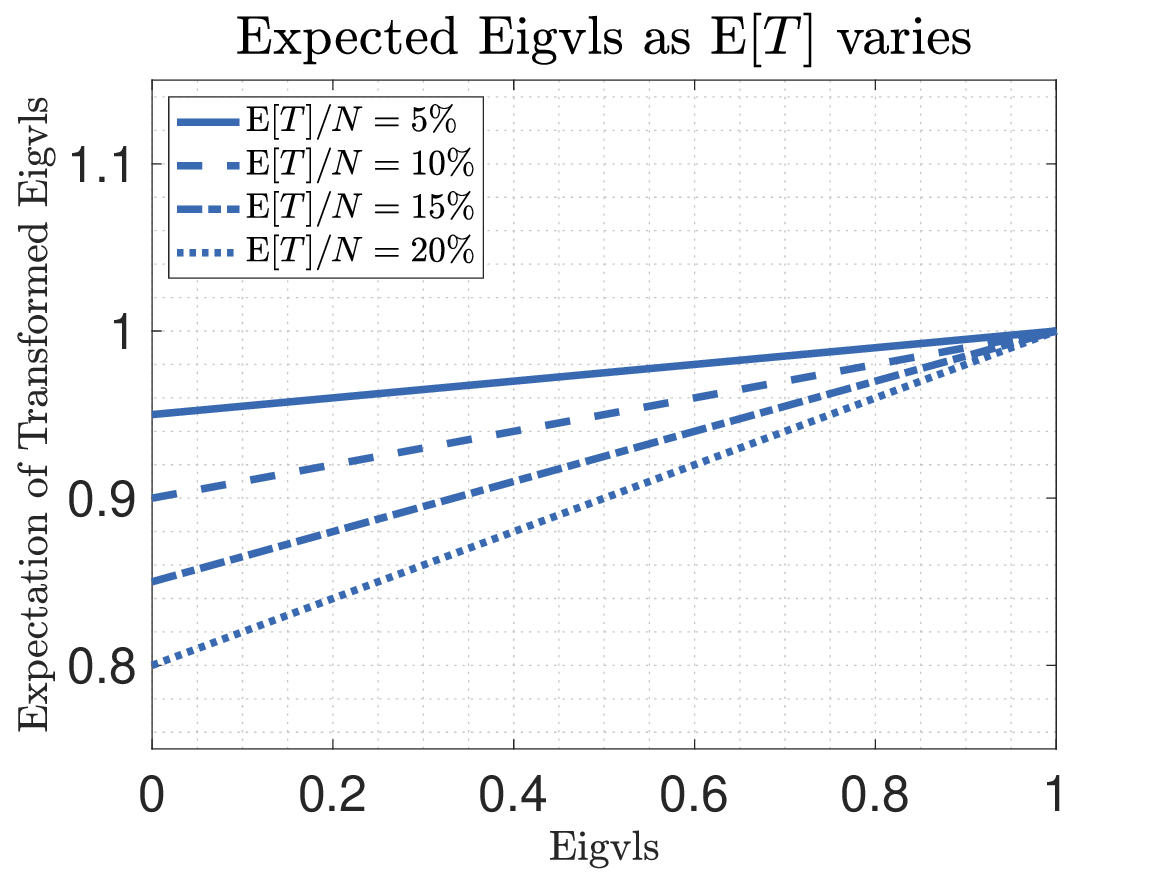}   
				\includegraphics[width=0.48\linewidth]{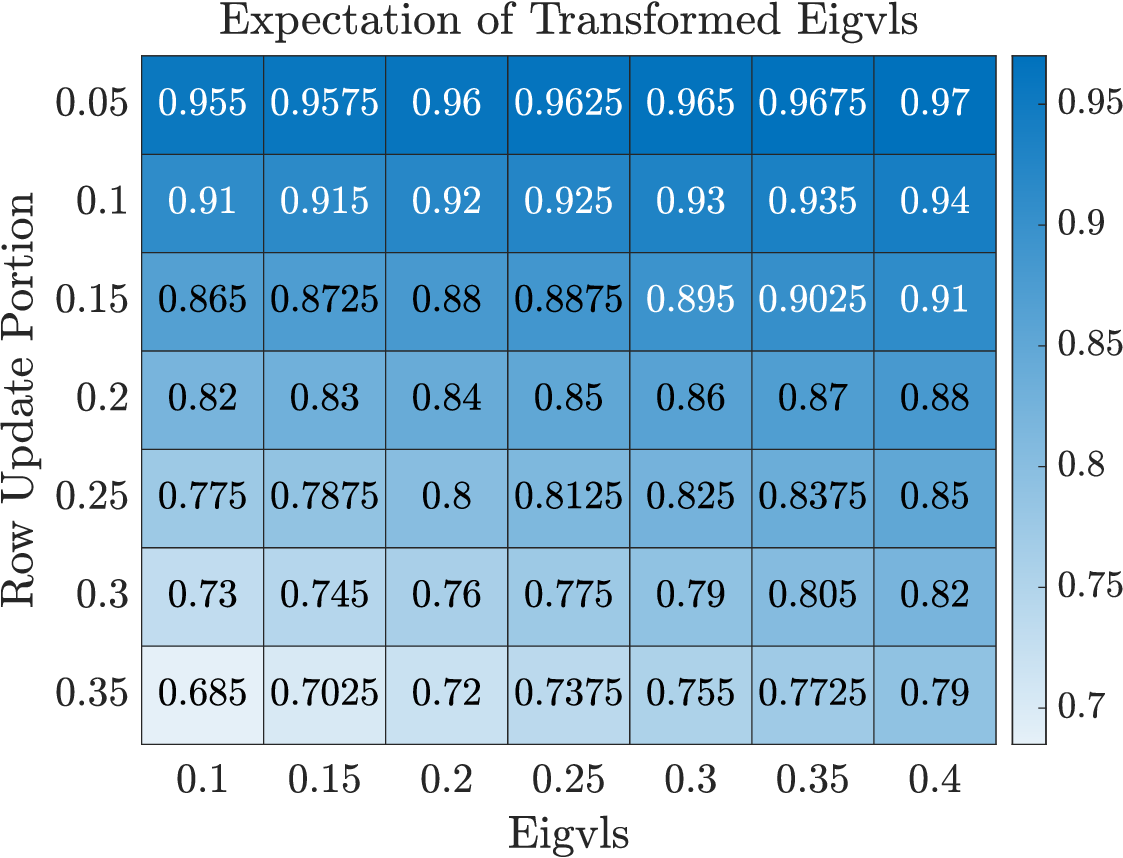}  
				\caption{{\it Left: mapping of $\lambda$ to $1+\frac{\mathbb{E}[T]}{N}(\lambda-1)$ for different values of $\frac{\mathbb{E}[T]}{N}$. Right: zoomed-in mapping of $\lambda$ to $1+\frac{\mathbb{E}[T]}{N}(\lambda-1)$ for $\frac{\mathbb{E}[T]}{N}=\{0.05:0.05:0.35\}$ and $\lambda=\{0.1:0.05:0.4\}$.}}\label{fig:2}
			\end{figure*}
			
			Figure \ref{fig:2} plots the mapping of $\lambda_j \in [0,1]$ ($x$-coordinate) to 
			$1+\frac{\mathbb{E}[T]}{N}(\lambda_j-1)$ ($y$-coordinate) for various values of the ratio 
			$\frac{\mathbb{E}[T]}{N}$. In particular, as $\mathbb{E}[T]\rightarrow 1$, the mapping 
			resembles a line that becomes more parallel to the real axis, indicating that the spread 
			of the range to which each $\lambda_j$ is mapped to becomes narrower around one. On the 
			other hand, as ${\mathbb{E}[T]}\rightarrow N$, the modulus of 
			$1+\frac{\mathbb{E}[T]}{N}\left(\lambda_j-1\right)$ is mapped closer to the original modulus 
			of $\lambda_j$. As expected, setting $\mathbb{E}[T]=N$ maps each $\lambda_j$ to itself, 
			since $\frac{\mathbb{E}[T]}{N}A + \frac{N-\mathbb{E}[T]}{N}I$ now becomes equal to $\mathbf{A}$.

			\section{An alternative matrix-vector model}\label{sec:newm}
			
			In this section, we discuss a different model to update the approximation 
			$\bm{\widehat{x}}_{k-1}$ than the one presented in (\ref{updzero2}). For the 
			sake of simplicity, we assume the same setting as in the asynchronous power 
			iteration where $|\lambda_1|=1$. The matrix-vector model suggested in this section is devised so that the approximate eigenvector $\bm{\widehat{x}}_k$ is -in expectation- equal to $\mathbf{A}^k\bm{\widehat{x}}_0$
			where $\bm{\widehat{x}}_0\in 
			\mathbb{R}^N$ is any vector such that $\bm{\widehat{x}}_0^\top \bm{v}_1 \neq 0$. More specifically, we 
			consider the update induced by the following formula for $k\geq 1$:
			\begin{equation} \label{updzero0}
				\bm{\widehat{x}}_{k} = \frac{N}{\mathbb{E}[T]}\mathbf{D}_{\mathcal{T}_k}\mathbf{A}\,\bm{\widehat{x}}_{k-1}.
			\end{equation}
			The update in (\ref{updzero0}) is similar to that induced by (\ref{updzero2}), except that 
			entries whose corresponding indices are not in ${\mathcal T}_k$ 
			are set to zero instead. In particular, 
			the $i$th component of $\bm{\widehat{x}}_{k}$ can be written as follows:
			\begin{equation}\label{updzero1}
				[\bm{\widehat{x}}_{k}]_i = 
				\begin{cases}
					\frac{N}{\mathbb{E}[T]}\sum_{j=1}^{j=N} \mathbf{A}_{ij}[\bm{\widehat{x}}_{k-1}]_j  &\ \mathrm{if}\ i\in {\mathcal T}_k\\
					0 &\ \mathrm{if}\ i\notin {\mathcal T}_k
				\end{cases}.
			\end{equation} 
			
			\begin{pro}\label{pro1}
				Let $\bm{\widehat{x}}_{k} = \frac{N}{\mathbb{E}[T]}\mathbf{D}_{\mathcal{T}_k}\mathbf{A}\bm{\widehat{x}}_{k-1}$ where all outcomes of the random variables $T$ and ${\mathcal T}$ appear with equal probability. Then, 
				\begin{align*}
					\mathbb{E}[\bm{\widehat{x}}_{k}] & = \mathbf{A}^k\bm{\widehat{x}}_0.
				\end{align*}
			\end{pro}
			\begin{proof}
				We can write $\bm{\widehat{x}}_{k}$ as 
				\begin{equation*}
					\bm{\widehat{x}}_{k} = \left(\frac{N}{\mathbb{E}[T]}\right)^k\mathbf{D}_{\mathcal{T}_k}\mathbf{A}\mathbf{D}_{\mathcal{T}_{k-1}}\mathbf{A} \cdots \mathbf{D}_{\mathcal{T}_1}\mathbf{A}\bm{\widehat{x}}_0.
				\end{equation*}
				Taking expectations on both sides leads to
				\begin{align*}
					\mathbb{E}[\bm{\widehat{x}}_{k}] &= \left(\frac{N}{\mathbb{E}[T]}\right)^k \mathbb{E}\left[\mathbf{D}_{\mathcal{T}_k}\mathbf{A}\mathbf{D}_{\mathcal{T}_{k-1}}\mathbf{A} \cdots \mathbf{D}_{\mathcal{T}_1}\mathbf{A}\bm{\widehat{x}}_0\right]\\
					& = \left(\frac{N}{\mathbb{E}[T]}\right)^k
					\mathbb{E}[\mathbf{D}_{\mathcal{T}_k}\mathbf{A}]\mathbb{E}[\mathbf{D}_{\mathcal{T}_{k-1}}\mathbf{A}] \cdots 
					\mathbb{E}[\mathbf{D}_{\mathcal{T}_1}\mathbf{A}]\bm{\widehat{x}}_0,
				\end{align*}
				where we took advantage of the fact that the $k$ samples of the random matrix 
				$\mathbf{D}_T\mathbf{A}$ are identically and independently distributed. Since 
				all outcomes of the random variables $T$ and ${\mathcal T}$ appear with equal 
				probability, it follows that 
				\begin{equation*}
					\mathbb{E}[\mathbf{D}_{\mathcal{T}_k}\mathbf{A}] = \dfrac{\mathbb{E}[T]}{N} \mathbf{A},\ k\geq 1,
				\end{equation*}
				which concludes the proof.
			\end{proof}
			Therefore, applying a cascade of $k$ samples of the random matrix 
			$\frac{N}{\mathbb{E}[T]}\mathbf{D}_T\mathbf{A}$ to an initial eigenvector approximation 
			$\bm{\widehat{x}}_0$ is -in expectation- equivalent to applying the matrix 
			power $\mathbf{A}^k,\ k\geq 1$. In turn, this is, up to a scaling factor, 
			precisely the algorithm of classical power iteration. In other words, 
			updating the eigenvector approximation as dictated by (\ref{updzero1}) 
			allows -in expectation- to annihilate the non-dominant eigendirections 
			of $\mathbf{A}$ according to the original ratios $|\lambda_j|/|\lambda_1|,\ j=2,\ldots,N$. 
			
			\begin{lem}
				For any $\bm{\widehat{x}}_0\in \mathbb{R}^N$ such that $\bm{\widehat{x}}_0^\top \bm{v}_1 \neq 0$, the 
				iterative process induced by the update in (\ref{updzero1}) results in
				\begin{align*}
					\mathbb{E}[\bm{v}_j^\top\bm{\widehat{x}}_{k}] = \lambda_j^{k}  \bm{v}_j^\top \bm{\widehat{x}}_0.
				\end{align*}
			\end{lem}
			\begin{proof}
				Following Proposition \ref{pro1} we have $\mathbb{E}[\bm{\widehat{x}}_{k}] = \mathbf{A}^k\bm{\widehat{x}}_0$ 
				and thus $\mathbb{E}[\bm{v}_j^\top\bm{\widehat{x}}_{k}] = \bm{v}_j^\top \mathbf{A}^k\bm{\widehat{x}}_0=\bm{v}_j^\top \mathbf{V}\mathbf{\Lambda}^k\mathbf{V}^\top\bm{\widehat{x}}_0$. The proof follows by noticing that $\bm{v}_j^\top \mathbf{V}$ is 
				identically zero except the $j$-th entry which is equal to one.
			\end{proof}
			
			\begin{thm}
				\label{thm:average}
				Let $\bm{\widehat{x}}_{k} = \frac{N}{\mathbb{E}[T]}\mathbf{D}_{\mathcal{T}_k}\mathbf{A}\bm{\widehat{x}}_{k-1}$, $\bm{\widehat{x}}_0\in \mathbb{R}^N$, such that $\bm{\widehat{x}}_0^\top \bm{v}_1 \neq 0$, where all outcomes of the random variables $T$ and ${\mathcal T}$ appear with equal probability. Then, for sufficiently large $k \in \mathbb{N}$ and $\lambda_j \in (-1,1),\ j=2,\ldots,N$, the vector $\mathbb{E}[\bm{\widehat{x}}_k]$ 
				and the eigenvector $\bm{v}_1$ are parallel.
			\end{thm}
			\begin{proof}
				Since $\lambda_j \in (-1,1),\ j=2,\ldots,N$, we know that classical power iteration is guaranteed 
				to converge to $\bm{v}_1$. Let now $k$ denote the smallest integer for which $\mathbf{A}^k\bm{\widehat{x}}_0$ becomes parallel to $\bm{v}_1$ in a floating-point arithmetic environment, i.e, $\mathbf{A}^k\bm{\widehat{x}}_0 = \alpha \bm{v}_1,\ \alpha \neq 0$. Since $\mathbb{E}[\bm{\widehat{x}}_{k}] = \mathbf{A}^k\bm{\widehat{x}}_0$, it follows that for any $j=2,\ldots,N$:
				\begin{align*}
					\mathbb{E}[\bm{v}_j^\top\bm{\widehat{x}}_{k}] & = \bm{v}_j^\top\mathbb{E}[\bm{\widehat{x}}_{k}] =\bm{v}_j^\top \mathbf{A}^k\bm{\widehat{x}}_0= \alpha \bm{v}_j^\top \bm{v}_1=0.
				\end{align*}
				Thus, $\mathbb{E}[\bm{\widehat{x}}_{k}]$ is perpendicular to all non-dominant eigenvectors $\bm{v}_2,\ldots,\bm{v}_N$.
			\end{proof}
			
			Although $\mathbb{E}[\bm{v}_j^\top\bm{\widehat{x}}_{k}]$ now depends on the original modulus of $\lambda_2,\ldots,\lambda_N$, the approximation $\bm{\widehat{x}}_{k}$ of $\bm{v}_1$ has 
			-in expectation- $N$-$\mathbb{E}[T]$ entries of zero modulus. As a result, the approximation $\bm{\widehat{x}}_k$ produced by a single run of partial power iteration equipped with the 
			matrix-vector model in (\ref{updzero1}) can not generally converge to $\bm{v}_1$ regardless 
			of the value of $k$.

			\section{Variance of the two partial matrix-vector models} \label{sec:varr}
			
			The quality of the approximation of the expectation $\mathbb{E}[\bm{\widehat{x}}_{k}]$ by 
			the partial power iteration variants defined by the partial matrix-vector models 
			in (\ref{updzero2}) and (\ref{updzero1}) is closely related to the higher order statistics 
			of the random vector $\bm{\widehat{x}}_k$. In this section, we present some basic calculations 
			related to the covariance associated with the update models in (\ref{updzero2}) and 
			(\ref{updzero1}). 
			
			\subsection{General Approach}
			
			For any $k\in \mathbb{N}$, calculating the covariance matrix $\cov[\bx]$ for an $N$-variate distribution $\bx$ is reduced to the calculations of the matrix $\ex[\bx \bx^\top]$ since
			\begin{equation*}
				\cov[\bm{x}]=\ex[(\bx-\ex[\bx])(\bx-\ex[\bx])^\top]=\ex[\bx \bx^\top]- \ex[\bx]\ex[\bx]^\top.
			\end{equation*}
			For this purpose, let $\mathbf{Q}_k$ denote the random matrix at iteration $k$ in~\eqref{updzero2} 
			and~\eqref{updzero1}, \emph{i.e.}, $\mathbf{Q}_k=\mathbf{I}-\mathbf{D}_{\mathcal{T}_k}(\mathbf{A}-\mathbf{I})$ for~\eqref{updzero2} and $\mathbf{Q}_k= \frac{N}{\ex[T]}\mathbf{D}_{\mathcal{T}_k}\mathbf{A}$ for~\eqref{updzero1}. Throughout this section we generically write the former as 
			\begin{equation*}
				\bm{\widehat{x}}_{k}=\mathbf{Q}_k \bm{\widehat{x}}_{k-1}.
			\end{equation*}
			\begin{lem}\label{lem:cov_general}
				Following the above definition of $\bm{\widehat{x}}_{k}$, we have
				\begin{align}
					\label{eqn:variance_calc}\ex[\bm{\widehat{x}}_{k}\bm{\widehat{x}}_{k}^\top]=\left(\prod_{q=1}^k\ex[\mathbf{S}_q]\right) (\bm{\widehat{x}}_{0}\bm{\widehat{x}}_{0}^\top),
				\end{align}
				where $\quad \mathbf{S}_q=\mathbf{Q}_q\otimes \mathbf{Q}_q$ and $\otimes$ denotes the tensor product (Kronecker product). In the case that $\mathbf{Q}_k$ are i.i.d, we have, $\ex[\bm{\widehat{x}}_{k}\bm{\widehat{x}}_{k}^\top]=\ex[S]^k (\bm{\widehat{x}}_{0}\bm{\widehat{x}}_{0}^\top)$, with $S=\mathbf{Q}_1\otimes \mathbf{Q}_1$.
			\end{lem}
			\begin{proof}
				\begin{align*}
					\ex[\bm{\widehat{x}}_{k}\bm{\widehat{x}}_{k}^\top] & = \ex\left[\left(\prod_{q=1}^k\mathbf{Q}_q\right) \bm{\widehat{x}}_{0}\,\, \left(\left(\prod_{q=1}^k\mathbf{Q}_q\right)\bm{\widehat{x}}_{0}\right)^\top\right]
					\\ &= \ex\left[\left(\prod_{q=1}^k\mathbf{Q}_q\right) \bm{\widehat{x}}_{0}\bm{\widehat{x}}_{0}^\top \left(\prod_{q=1}^k\mathbf{Q}_q\right)^\top \right]
					\\&= \ex[\mathbf{Q}_k\cdots \mathbf{Q}_2 \mathbf{Q}_1\bm{\widehat{x}}_{0}\bm{\widehat{x}}_{0}^\top  \mathbf{Q}_1^\top \mathbf{Q}_2^\top \cdots \mathbf{Q}_k^\top]
					\\&= \ex_k[ \mathbf{Q}_k\ex_{k-1}[\cdots  \ex_2[\mathbf{Q}_2\ex_1 [\mathbf{Q}_1\bm{\widehat{x}}_{0}\bm{\widehat{x}}_{0}^\top \mathbf{Q}_1^\top] \mathbf{Q}_2^\top] \cdots ]\mathbf{Q}_k^\top],
				\end{align*}
				where $\ex_q$, $q=1, 2,\ldots, k$ denote the expectations taken with respect to $\mathbf{Q}_q $, $q=1, 2,\ldots, k$.
				
				For any $N\times N$ matrix $\mathbf{B}$, $\mathbf{A}\mapsto \mathbf{B}\mathbf{A}\mathbf{B}^\top$ is a linear transformation with respect to $\mathbf{A}$, which can be viewed as vectors in $\Real^{N^2}$. A linear transformation on $\Real^{N^2}$ is then represented by a $N^2\times N^2$ matrix. Straightforward calculations lead to that 
				$\mathbf{B}\mathbf{A}\mathbf{B}^\top= \mathbf{SA}$, with $\mathbf{S}=\mathbf{B}\otimes \mathbf{B}$ being the Kronecker product of $\mathbf{B}$ and itself, an $N^2\times N^2$ matrix with $\mathbf{S}_{(i,j),(k, \ell)} =\mathbf{B}_{ik}\mathbf{B}_{\ell j}$ for $i,j,k, \ell=1,\ldots,N$, and $\mathbf{A}$ being treated as a $N^2$ dimensional vector. In conjunction with the independence assumption, we obtain~\eqref{eqn:variance_calc}.
			\end{proof}
			
			Lemma~\ref{lem:cov_general} enables us to calculate the variances for different update algorithms, including those that have random selection that changes over time (depending on time, not the state), which might be called “update according to a schedule'', similar to analogous procedures in optimization and learning theory.

			In the following we shall use:
			\begin{align}\label{eqn:Sigma Ip}
				\mathbf{\Sigma}:=
				\left(\begin{array}{c} \mathbf{A}_1\otimes \mathbf{A}_1\\\bm{e}_1\otimes \mathbf{A}_2+\mathbf{A}_1\otimes \bm{e}_2\\ \vdots \\ \bm{e}_1\otimes \mathbf{A}_N+\mathbf{A}_1\otimes \bm{e}_N\\ \bm{e}_2\otimes \mathbf{A}_1+\mathbf{A}_2\otimes \bm{e}_1\\\mathbf{A}_2\otimes \mathbf{A}_2\\ \vdots \\ \bm{e}_2\otimes \mathbf{A}_N+\mathbf{A}_2\otimes \bm{e}_N \\ \vdots \\ \bm{e}_N\otimes \mathbf{A}_1+\mathbf{A}_N\otimes \bm{e}_1\\ \vdots \\ \mathbf{A}_N\otimes \mathbf{A}_N\end{array}\right)
				\quad I':=\left(\begin{array}{c} \bm{A}_1\otimes \bm{A}_1\\0\\ \vdots \\ 0\\ 0\\\bm{A}_2\otimes \bm{A}_2\\ \vdots \\ 0 \\ \vdots \\ 0\\0\\ \vdots \\ \bm{A}_N\otimes \bm{A}_N\end{array}\right).
				\,,
			\end{align}
			
			\subsection{Covariance of the update (\ref{updzero2})}\label{subsec: CovUpdate1}
			
			Applying Lemma~\ref{lem:cov_general} to update~\eqref{updzero2}, we can get the the matrix multiplier for the covariance matrix calculations. For ease of exposition, write
			\begin{align*}
				\mathbf{A}=\left(\begin{array}{c} \bm{A}_1\\\bm{A}_2\\ \vdots \\ \bm{A}_N \end{array}\right), \quad \hbox{$\bm{A}_i=\bm{A}_{i:}$ the $i$-th row of $\mathbf{A}$.}
			\end{align*}
			Therefore, for the covariance of a single update we have 
				\begin{equation}\label{equ:cov1}
					\ex[\mathbf{S}]=\frac{\ex[T]}{N}\left(1-\frac{\ex[T]}{N}\right)
					\left(-\frac{1}{N-1} \mathbf{A}\otimes \mathbf{A}+ \mathbf{I}\otimes \mathbf{I}+\frac{N}{N-1}\mathbf{\Sigma}- (\mathbf{I}\otimes \mathbf{A}+\mathbf{A}\otimes \mathbf{I})\right)\,.
				\end{equation}
			
			Detailed calculations can be found in the supplement~\ref{appendix_b}. 
			
			\subsection{Covariance of the update (\ref{updzero1})}
			
			Following the same technique as above we can repeat the same computations for the update 
			defined by equation~\eqref{updzero1}. In this case, the covariance is 
				\begin{equation}\label{equ:cov2}
					\ex[\mathbf{S}] =\frac{N}{\ex[T]}\frac{N}{N-1}\left(\left(\frac{\ex[T]}{N}-\frac{1}{N}\right)\mathbf{A}\otimes \mathbf{A} + \left(1-\frac{\ex[T]}{N}\right) I'\right).
				\end{equation}
			As pointed out in the statement of Lemma~\ref{lem:cov_general}, the second moment matrix of the iterates $\bm{\widehat{x}}_k$ itself develops as power iterations of the $N^2\times N^2$ matrix $\ex[\mathbf{S}]$. We observe from the covariance of update~\eqref{updzero1} that while the coefficient for the first term is roughly constant (around one) as the ratio of update $\ex[T]/N$ varies, the coefficient for the second term is of the order $N/\ex[T]$. The variance thus scales inversely with the proportion of the rows updated, and so our preferred regime of a small proportion of updates will lead to large variance. Given Theorem~\ref{thm:average}, the variance can be seen as a surrogate of the deviation from the limit.\footnote{This is verified by the numerical results reported in Section \ref{sec:experiments} where we demonstrates this effect when $\ex[T]/N$ is small.} In comparison, when the above analysis is applied to update~\eqref{updzero2}, the corresponding coefficients are roughly $ \ex[T]/N(1-\ex[T]/N)$, which is at most $1/4$. Furthermore, if we vary the value of $\ex[T]/N$, then the Frobenius norm of $\ex[\mathbf{S}]$ of the update (\ref{updzero2}) remains bounded since each term of the summand will have a bounded Frobenius norm, while its counterpart in (\ref{updzero1}) tends to infinity as $\mathbb{E}[T]/N$ tends to zero. This is depicted in Figure~\ref{fig:Cpv12}, where we report the normalized (by $N^2$) Frobenius norms of the covariance matrix in~\eqref{equ:cov1} and~\eqref{equ:cov1}, respectively, for a randomly generated matrix $\mathbf{A}$ and $\mathbb{E}[T]/N$ values ranging from $1/N$ to $1$.
			\begin{figure}
				\includegraphics[scale=0.5]{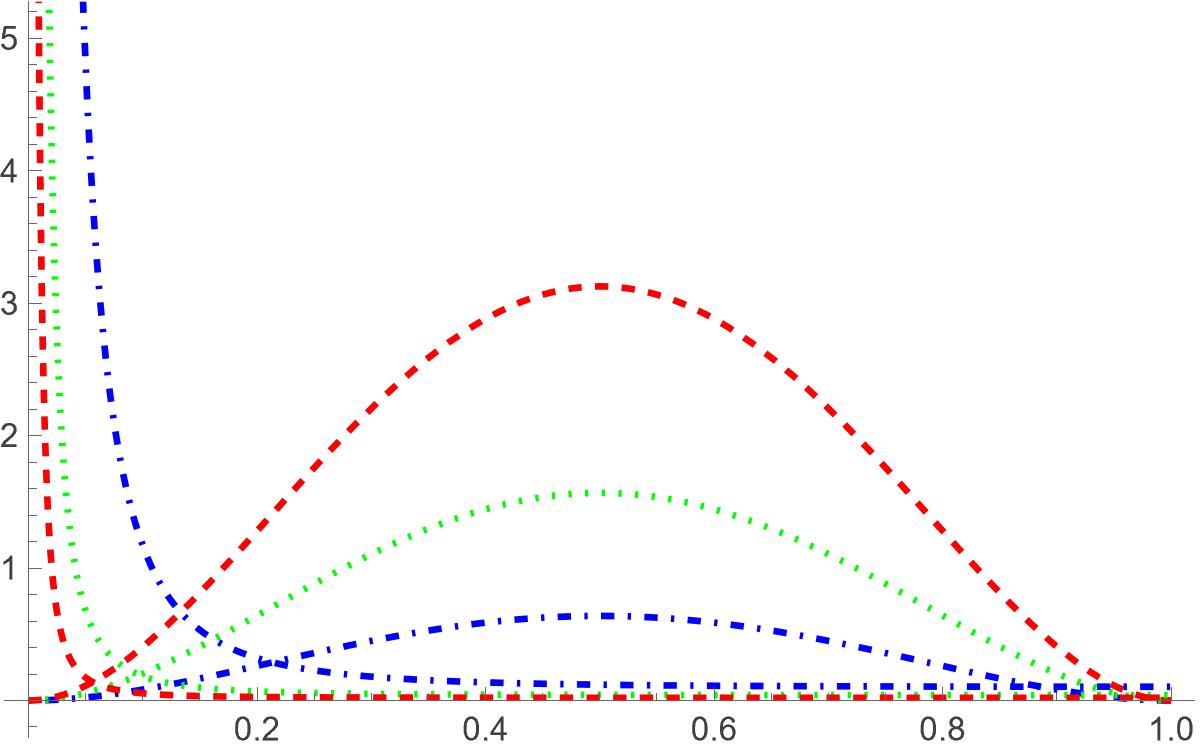}
				\caption{Dependence of the Frobenius norm (Squared, normalized by $N^2$) of the  covariance matrices.  Dashed red for $N=50$, Dotted green for $N=25$ and Dot-dashed blue for $N=10$. The symmetric, unimodal graph corresponds to \eqref{equ:cov1} while the one with asymptote at 0 corresponds to \eqref{equ:cov2}.  The $x$-axis runs over $\mathbb{E}[T]/N$ for $1\leq \mathbb{E}[T]\leq N$.}
				\label{fig:Cpv12}
			\end{figure}

			\section{Probabilistic switching between the two matrix-vector models}\label{sec:switch_algo}
			
			The previous section demonstrated that the update model (\ref{updzero1}) introduces larger 
			variance than the update model (\ref{updzero2}) as a direct consequence of setting to zero 
			$N-T_k$ entries of $\bm{\widehat{x}}_k$ at iteration $k$. On the other hand, Section \ref{sec:newm} 
			showed that, in contrast to (\ref{updzero2}), updating $\bm{\widehat{x}}_k$ as in (\ref{updzero1}) 
			leads to an expected inner product $\mathbb{E}[\bm{v}_j^\top\bm{\widehat{x}}_{k}]$ that depends on the 
			original spectral gap of the matrix $\mathbf{A}$ instead the (smaller) spectral gap of the expectation 
			matrix $\frac{\mathbb{E}[T]}{N}\mathbf{A} + \frac{N-\mathbb{E}[T]}{N}\mathbf{I}$.

			\subsection{A mixture of models}

			In this section we consider the combination of the update models (\ref{updzero2}) and 
			(\ref{updzero1}).
			
			\begin{pro}\label{pro2}
				Let $k=k_1+k_2,\ k_1,k_2\in \mathbb{N}$, and for an appropriate initial approximation $\bm{\widehat{x}}_{0}\in \mathbb{R}^N$ 
				consider $k_1$ iterations of the model (\ref{updzero1}), followed by $k_2$ iterations of the model (\ref{updzero2}). Then, 
				\begin{equation*}
					\mathbb{E}[\bm{v}_j^\top \bm{\widehat{x}}_{k}] = 
					\lambda_j^{k_1}\left[1+\frac{\mathbb{E}[T]}{N}(\lambda_j-1)\right]^{k_2}  \bm{v}_j^\top \bm{\widehat{x}}_{0}.
				\end{equation*}
			\end{pro}
			\begin{proof}
				The approximation of $\bm{v}_1$ after $k$ iterations can be written as 
				\begin{equation*}
					\bm{\widehat{x}}_{k} = \prod_{j=1}^{k_2}\left[\mathbf{I}+\mathbf{D}_{\mathcal{T}_{j}}(\mathbf{A}-\mathbf{I})\right] \prod_{i=1}^{k_1}\frac{N}{\mathbb{E}[T]}\mathbf{D}_{\mathcal{T}_{i}}\mathbf{A}\bm{\widehat{x}}_{0}.
				\end{equation*}
				Due to the independence and identical distribution of each matrix sample, applying the expectation operator to both sides results to
				\begin{align*}
					\mathbb{E}[\bm{\widehat{x}}_{k}] &= \mathbb{E}\left[\prod_{j=1}^{k_2}\left[\mathbf{I}+\mathbf{D}_{\mathcal{T}_{j}}(\mathbf{A}-\mathbf{I})\right]\right] \mathbb{E}\left[\prod_{i=1}^{k_1}\frac{N}{\mathbb{E}[T]}\mathbf{D}_{\mathcal{T}_{i}}\mathbf{A}\right]\bm{\widehat{x}}_{0}\\
					&=    \left[\frac{\mathbb{E}[T]}{N}\mathbf{A} + \frac{N-\mathbb{E}[T]}{N}\mathbf{I}\right]^{k_2} \mathbf{A}^{k_1}
					\bm{\widehat{x}}_{0}.
				\end{align*}
				The proof concludes by multiplying both sides from the left by $\bm{v}_j^\top$ and recalling that $\bm{v}_j^\top V$ is identically zero except for the $j$-th entry which is equal to one.
			\end{proof}
			Proposition \ref{pro2} shows that by combining the two update models we can mitigate the 
			drawback of a large modulus $1+\frac{\mathbb{E}[T]}{N}(\lambda_j-1)$. Indeed, consider the 
			case where $\lambda_j = 0$ and thus $1+\frac{\mathbb{E}[T]}{N}(\lambda_j-1) = \frac{N-\mathbb{E}[T]}{N}$. For very small values of $\mathbb{E}[T]$, the modulus of $\frac{N-\mathbb{E}[T]}{N}$ 
			is approximately one and $k_2$ needs to be quite large to reduce the modulus of 
			$\left[1+\frac{\mathbb{E}[T]}{N}(\lambda_j-1)\right]^{k_2}$. On the other hand, using the  
			model (\ref{updzero1}) just once will -in expectation- nullify $\mathbb{E}[\bm{v}_j^\top \bm{\widehat{x}}_{k}]$. In the more general case where $\lambda_j \approx 0$, applying just a very 
			small number of $k_1$ iterations of the update model (\ref{updzero1}) can reduce $\mathbb{E}[\bm{v}_j^\top \bm{\widehat{x}}_{k}]$ much faster than the model (\ref{updzero2}) alone. 
			
			\begin{cor}
				Let $|\lambda_1|=1$ and $k=k_1+k_2,\ k_1,k_2\in \mathbb{N}$. For any $\lambda_j \in (-1,1)$, 
				\begin{equation*}
					\lim_{k_1 \rightarrow \infty || k_2 \rightarrow \infty}\lambda_j^{k_1} 
					\left[1+\frac{\mathbb{E}[T]}{N}(\lambda_j-1)\right]^{k_2} = 0.
				\end{equation*}
			\end{cor}
			The above corollary indicates that the product $\lambda_j^{k_1} \left[1+\frac{\mathbb{E}[T]}{N}(\lambda_j-1)\right]^{k_2}$ converges to zero as long as either $k_1$ or $k_2$ increase. Indeed, since the modulus of both $\lambda_j$ and $1+\frac{\mathbb{E}[T]}{N}(\lambda_j-1)$ is strictly 
			upper-bounded by one, at least one part of the product is guaranteed to converge to zero as $k$ increases.
			
			While combining (\ref{updzero2}) and (\ref{updzero1}) can reduce $\mathbb{E}[\bm{v}_j^\top \bm{\widehat{x}}_{k}]$ faster, it introduces two parameters, \emph{i.e.}, $k_1$ and $k_2$. Moreover, while in Proposition \ref{pro2} we assumed that the $k_1$ iterations of the model (\ref{updzero1}) precede the $k_2$ iterations of the model (\ref{updzero2}), the presented theoretical result remain valid 
			regardless of the order in which these $k_1+k_2$ iterations\footnote{The variance, on the other hand, varies as this order changes.} occur. Since we have the liberty of choosing this order, we 
			aim to bias the selection of the update model (\ref{updzero1}) in the beginning of the process where the approximate eigenvector $\bm{\widehat{x}}_{k}$ is not yet a highly accurate approximation of $\bm{v}_1$ and thus its bias towards zero is less harmful. To avoid explicitly setting the parameters $k_1$ and $k_2$, we consider a probabilistic framework where the choice of the update model at iteration $k$ depends on biased probability sampling.

			\subsection{Summary of the proposed algorithm} \label{algg}
			
			Algorithm \ref{alg1} summarizes the approach proposed in this section. Given a random initial approximation $\bm{\widehat{x}}_0 \in \mathbb{R}^N$ of the eigenvector $\bm{v}_1$ of the matrix $\mathbf{A}$, 
			the algorithm updates $\bm{\widehat{x}}_{k-1}$ to $\bm{\widehat{x}}_k$ by choosing between the two update formulas (\ref{updzero2}) and (\ref{updzero1}). Following the previous section, as well as the discussion in the last paragraph in Section \ref{sec:varr}, we want to leverage the update in (\ref{updzero1}) during the initial phase of partial power iteration followed by a switch to the update formula in (\ref{updzero2}). 
			
			The choice of either of these two update formulas is realized by sampling a scalar $\rho \in \mathbb{R}$ from the class of uniformly distributed random numbers between zero and one, and comparing the latter against $[\bm{\pi}]_k \in \mathbb{R}$ where the vector $\bm{\pi} \in \mathbb{R}^K$ is predetermined and passed to Algorithm \ref{alg1}. Herein, the integer $K\in \mathbb{N}$ denotes an upper bound on the number of iterations performed by Algorithm \ref{alg1}. Let the entries of $\bm{\pi}$ set such that $[\bm{\pi}]_k > [\bm{\pi}]_{k+1}$. We pick the update in (\ref{updzero1}) if $\rho < [\bm{\pi}]_k$ and the update (\ref{updzero2}) otherwise. Naturally, if $[\bm{\pi}]_k$ is close to one, then it is more likely to pick the update in (\ref{updzero1}) at iteration $k$. On the other hand, if $[\bm{\pi}]_k$ is close to zero, then it is more likely to pick the update in (\ref{updzero2}) at the same iteration. Therefore, devising a vector $\bm{\pi}$ with larger values in its first entries yields the desired outcome of choosing the update in (\ref{updzero1}) with higher probability during the early stages of power iteration.\footnote{The switching mechanism introduces negligible additional computational overhead. At each iteration, our implementation only evaluates a scalar probability, draws one Bernoulli random variable, and then applies one of the two already-defined partial matrix-vector update rules. Therefore, the dominant cost remains the computation and communication associated with the observed entries of the partial matrix-vector product.}
			
			We consider two probabilistic options to choose between the matrix-vector models  (\ref{updzero2}) and (\ref{updzero1}), both acting as a simulated Bernoulli trial. The first option sets $[\bm{\pi}]_k = e^{-\zeta(k-1)}$ where $\zeta$ controls the rate at which the probability of selecting the zero-filling update decays. The entries of the vector $\bm{\pi}$ decay exponentially as $k$ increases, starting from $[\bm{\pi}]_1=1$. 
			Our default choice is $\zeta = \frac{N-\mathbb{E}[T]}{2N}$ so that the decay rate is proportional to the expected fraction of non-updated entries. When only a small fraction of the matrix-vector product is updated, the algorithm moves away from (\ref{updzero1}), whose variance can be large. On the other hand, when more entries are updated, the decay is slower (\ref{updzero1}) can be used for more iterations.
			The second option sets $[\bm{\pi}]_k=\frac{1}{1+e^{\alpha(k-\beta)}}$ for some real scalars $\alpha$ and $\beta$. This option corresponds to a (logistic) sigmoid function and maps $[\bm{\pi}]_k\in [0,1]$ for any $k$. Here, $\beta$ determines the transition point along the real axis, i.e., $[\bm{\pi}]_\beta=1/2$, while $\alpha$ controls the sharpness (slope) of the transition. As a practical default, one can choose $\beta=\left\lceil \frac{2N}{N-\mathbb{E}[T]} \right\rceil$ 
				so that the transition from (\ref{updzero1}) to (\ref{updzero2}) occurs earlier when a large fraction of entries is non-updated and later when most entries are updated. In our experimental framework, the range $\alpha\in[0.3,1]$ has been found to produce the best empirical results, with $\alpha=0.65$ as a default.

			\begin{figure*}[ht]
				\centering
				\includegraphics[width=0.47\linewidth]{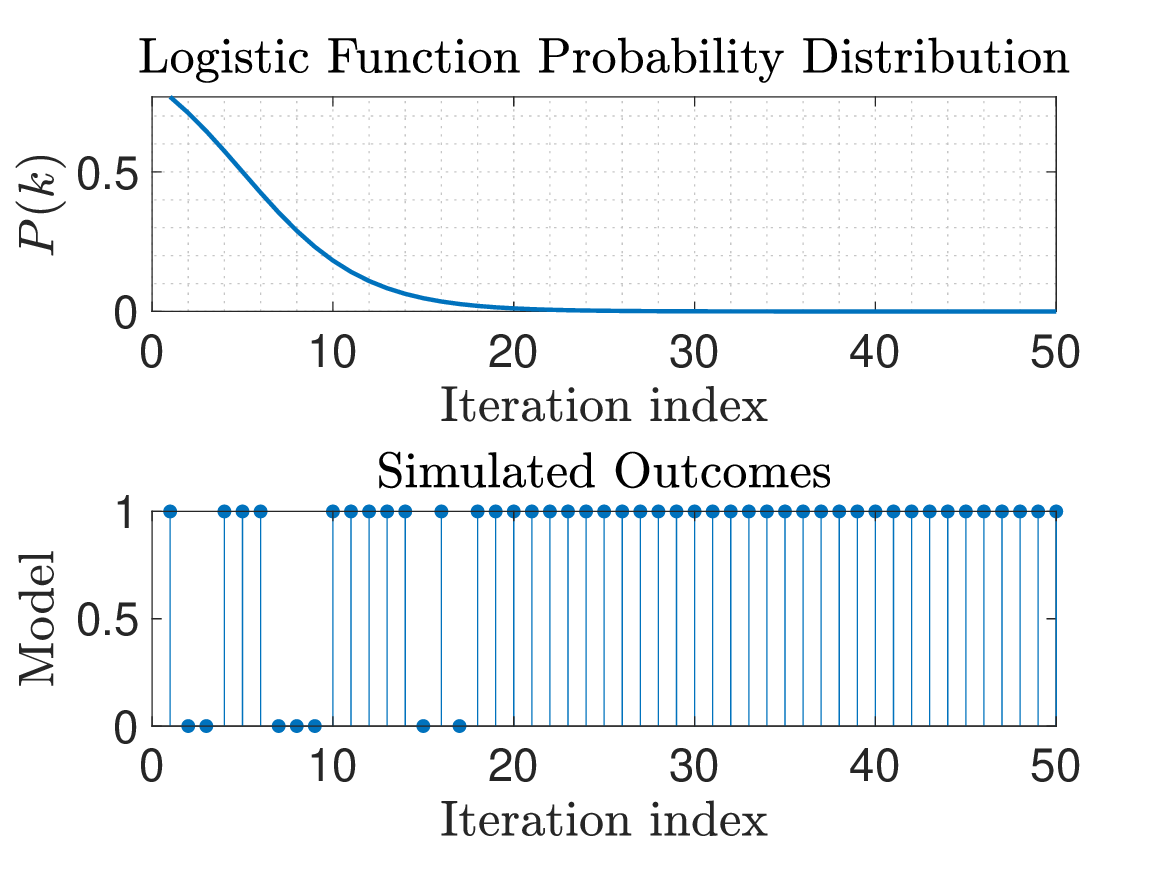}
				\includegraphics[width=0.47\linewidth]{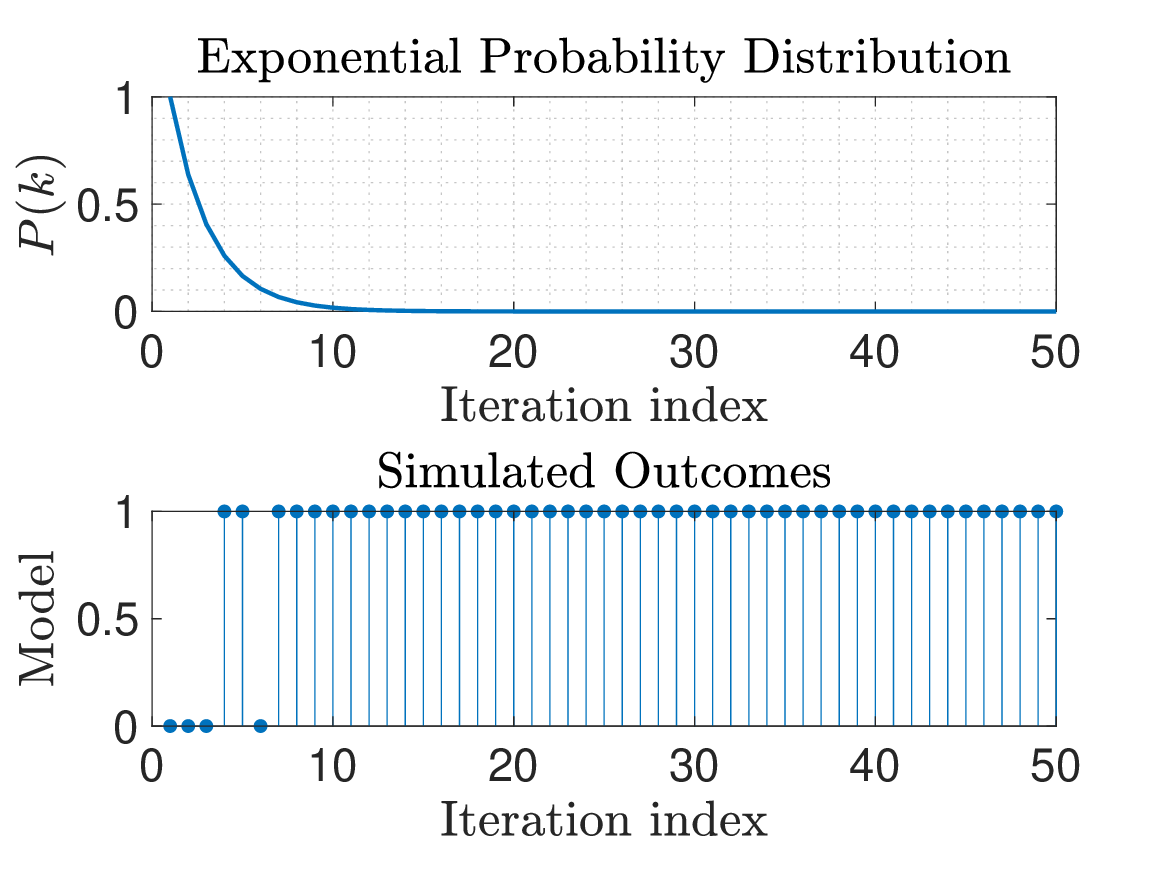}  
				\caption{{\it Top, left: model probability $P(k)\equiv [\bm{\pi}]_k=\frac{1}{1+e^{\alpha(k-\beta)}}$ where $\alpha=.3,\ \beta=5$. Top, right: model probability $P(k)\equiv [\bm{\pi}]_k=e^{-\zeta(k-1)}$ where $\zeta=\frac{N-\mathbb{E}[T]}{2N}$. Bottom: simulated model selection in the range $k\in [1,50]$ where '0' implies (\ref{updzero1}) and '1' implies (\ref{updzero2}).}}\label{fig1}
			\end{figure*}
			
			Figure \ref{fig1} (top) plots the modulus of the $k$-th entry of $\bm{\pi},\ k\in [1,50]$, 
			for the option $[\bm{\pi}]_k=\frac{1}{1+e^{\alpha(k-\beta)}}$ with $\alpha=.3,\ \beta=5$, 
			and the option $[\bm{\pi}]_k=e^{-\zeta(k-1)}$ with $\zeta=\frac{N-\mathbb{E}[T]}{2N}$. The 
			bottom part of the plots shows the actual model chosen at iteration $k$ after sampling 
			$\rho \in (0,1)$ and comparing against $[\bm{\pi}]_k$. Note that it is possible to switch back and forth updates 
			(\ref{updzero1}) and (\ref{updzero2}), nonetheless eventually the entries of $\bm{\pi}$ will 
			decay enough to make the switch to (\ref{updzero2}) permanent in all likelihood.

			\begin{algorithm}
				\caption{Power iteration with probabilistic matrix-vector model selection} \label{alg1}
				\begin{algorithmic}[1]
					\renewcommand{\algorithmicrequire}{\textbf{Input:}}
					\renewcommand{\algorithmicensure}{\textbf{Output:}}
					\REQUIRE $\mathbf{A} \in \mathbb{R}^{N\times N},\ \bm{\widehat{x}}_0 \in \mathbb{R}^N,\ K\in \mathbb{N},\ \mathbb{E}[T] \in \mathbb{N},\ \bm{\pi} \in \mathbb{R}^K$
					\ENSURE $\bm{\widehat{x}}_K\in \mathbb{R}^N$
					\FOR {$k = 1$ to $K$}
					\STATE Sample i.i.d. $\rho \in (0,1)$
					\STATE Determine $T_{k}$ and ${\mathcal T}_{k}$
					\IF {$\rho < [\bm{\pi}]_k$}
					\STATE Compute $\bm{\widehat{x}}_{k}$ as in (\ref{updzero1})
					\ELSE
					\STATE Compute $\bm{\widehat{x}}_{k}$ as in (\ref{updzero2})
					\ENDIF
					\STATE Normalize $\bm{\widehat{x}}_{k}$
					\STATE If $\bm{\widehat{x}}_{k}^\top \bm{\widehat{x}}_{k-1}=1$ in floating-point arithmetic, {\bf exit}
					\ENDFOR
					\RETURN $\bm{x}_k$ 
				\end{algorithmic} 
			\end{algorithm}
			
			The main computational cost in each iteration of Algorithm \ref{alg1} is the progression of 
			$\bm{\widehat{x}}_{k-1}$ to $\bm{\widehat{x}}_{k}$ achieved by utilizing either (\ref{updzero1}) or 
			(\ref{updzero2}). The update in (\ref{updzero2}) can be written as 
			\begin{equation*}
				\bm{\widehat{x}}_{k}=\bm{\widehat{x}}_{k-1} + \mathbf{D}_{\mathcal{T}_k}\mathbf{A}\bm{\widehat{x}}_{k-1} -\mathbf{D}_{\mathcal{T}_k}\bm{\widehat{x}}_{k-1},
			\end{equation*}
			from which we infer that for both (\ref{updzero1}) and (\ref{updzero2}) the main computational 
			cost stems from computing the matrix-vector product 
			$\bm{y}_k=\mathbf{D}_{\mathcal{T}_{k}}\mathbf{A}\bm{\widehat{x}}_{k-1}$. The 
			cost of this computation is upper-bounded by $2T_k(N-1)$ and is achieved when each one of 
			the $T_k$ chosen rows of $\mathbf{A}$ is dense.

			\section{Averaging Power Iterations over~\eqref{updzero1}} \label{sec:alg2}

			Recall that in Section \ref{sec:newm}, it was shown that $\mathbb{E}[\bm{\widehat{x}}_k]=\mathbf{A}^k\bm{\widehat{x}}_0$ 
			and thus for a sufficiently large value of $k$, approximating $\mathbb{E}[\bm{\widehat{x}}_k]$ 
			is equivalent to approximating $\bm{v}_1$. One idea is to approximate the expectation 
			$\mathbb{E}[\bm{\widehat{x}}_k]$ via Monte-Carlo, \emph{i.e.}, by the empirical mean 
			$\frac{1}{L}\sum_{l=1}^L\bm{\widehat{x}}_k^{(l)},\ L\in \mathbb{N}$, where $\bm{\widehat{x}}_k^{(l)}$ 
			denotes the $l$-th replication of $k$ steps of inexact power iteration approximation with 
			the same initial approximation $\bm{\widehat{x}}_0$. While such an approach can lead to a good 
			approximation\footnote{A numerical experiment can be found in our supplement.} of $\bm{v}_1$ for modest values of $L$, it requires $L$ $k$-length 
			independent applications of inexact power iteration, thus increasing the overall computational 
			cost. Moreover, the value of $k$ is not known a priori.
			
			In this section we present an algorithm that converges to $\bm{v}_1$ without requiring more than 
			one trial. A natural way to develop 
			such an algorithm is to keep a running average of the approximations produced by inexact 
			power iteration. In contrast to the approaches outlined so far in this paper, we now consider 
			partial matrix-vector products where the sampling occurs along the rows of $\bm{\widehat{x}}_{k-1}$ 
			instead of matrix~$\mathbf{A}$, \emph{i.e.}, the matrix operator $\mathbf{D}_{\mathcal{T}_k}$ is applied to the 
			columns of $\mathbf{A}$. This is especially appealing when the matrix $\mathbf{A}$ is stored according to Compressed Column Storage, 
				also known as Harwell-Boeing format, where the entries of $\mathbf{A}$ are stored sequentially 
				by column.
			
			More specifically, we define the vector
			\begin{align*}
				\bm{\widehat{m}}_{k}=\frac{k-1}{k}\bm{\widehat{m}}_{k-1} + \frac{1}{k}\bm{\widehat v}_{k}, \quad \bm{\widehat{m}}_0=0,
			\end{align*}
			where for some random $\bm{\widehat v}_0 \in \mathbb{R}^N$, and $\mathbf{A}_{:j}$ representing the 
			$j$-th column of the matrix $\mathbf{A}$, we define
			\begin{align*}
				\bm{\widehat v}_k &= \sum_{j\in {\mathcal T}_k} \mathbf{A}_{:j}[\bm{\widehat{v}}_{k-1}]_j \\
				& = \frac{N}{\mathbb{E}[T]} \mathbf{A} \mathbf{D}_{\mathcal{T}_k}\bm{\widehat{v}}_{k-1}.
			\end{align*}
			\begin{lem}\label{lem:averaging_converge}
				Let $\mathbf{A}^k \bm{\widehat v}_0$ converge to $\bm{v}_1$ as $k \rightarrow \infty$. Then, 
				\begin{equation*}
					\lim_{k\ra \infty } \ex[\bm{\widehat{m}}_{k}]= \phi \bm{v}_1,\ \phi \in \mathbb{R}^*,
				\end{equation*}
				where $\mathbb{R}^*$ denotes the set of all real numbers excluding the origin.
			\end{lem}
			\begin{proof}
				For any $k\ge 1$, we have, 
				\begin{align*}
					\ex[\bm{\widehat{m}}_{k}]=& \frac{1}{k}\sum_{i=0}^{k-1} \ex\left[\prod_{j=1}^i\frac{N}{\mathbb{E}[T]}\mathbf{A}\mathbf{D}_{\mathcal{T}_j}\right] \bm{\widehat v}_0
					= \frac{1}{k}\sum_{i=0}^{k-1} \mathbf{A}^i \bm{\widehat v}_0.
				\end{align*}
				By definition, $\mathbf{A}^k \bm{\widehat v}_0$ converges to $(\bm{\widehat v}_0^\top \bm{v}_1 )\bm{v}_1$ as $k\ra \infty$ and thus 
				$\ex[\bm{\widehat{m}}_{k}]$ converges to a vector that is parallel to $\bm{v}_1$.
			\end{proof}
			
			Algorithm \ref{alg2} presents the averaging procedure that computes an approximation 
			of $\bm{v}_1$ while leveraging only the matrix-vector model (\ref{updzero1}). In addition 
			to the computation of $\frac{N}{\mathbb{E}[T]} \mathbf{D}_{\mathcal{T}_k}\mathbf{A}\bm{\widehat v}_{k-1}$, the $k$-th 
			iteration of Algorithm \ref{alg2} requires $5N$ floating-point operations to normalize 
			$\bm{\widehat v}_k$ and scale/add the vectors $\frac{k-1}{k}\bm{\widehat{m}}_{k-1}$ and $\frac{1}{k}\bm{\widehat v}_k$. When contrasted to Algorithm \ref{alg1}, both algorithms require $2N$ floating-point operations to normalize the most recent iteration vector but Algorithm \ref{alg2} performs $3N$ additional floating-point operations per iteration during the computation of $\bm{\widehat{m}}_k$.
			\begin{algorithm}
				\caption{Power iteration with averaging} \label{alg2}
				\begin{algorithmic}[1]
					\renewcommand{\algorithmicrequire}{\textbf{Input:}}
					\renewcommand{\algorithmicensure}{\textbf{Output:}}
					\REQUIRE $\mathbf{A} \in \mathbb{R}^{N\times N},\ \bm{\widehat{m}}_0 \in \{0\}^N,\ \bm{\widehat v}_0 \in \mathbb{R}^N,\ K,\ \mathbb{E}[T] \in \mathbb{N}$
					\ENSURE $\bm{\widehat{m}}_K\in \mathbb{R}^N$
					\FOR {$k = 1$ to $K$}
					\STATE Determine $T_{k}$ and ${\mathcal T}_{k}$
					\STATE Set $\bm{\widehat v}_k = \frac{N}{\mathbb{E}[T]}  \mathbf{A}\mathbf{D}_{\mathcal{T}_k}\bm{\widehat v}_{k-1}$
					\STATE Normalize $\bm{\widehat v}_k$
					\STATE Set $\bm{\widehat{m}}_k = \frac{k-1}{k}\bm{\widehat{m}}_{k-1} + \frac{1}{k}\bm{\widehat v}_k$
					\ENDFOR
					\RETURN $\bm{\widehat x}_k = \bm{\widehat{m}}_K/\|\bm{\widehat{m}}_K\|$ 
				\end{algorithmic} 
			\end{algorithm}

			\section{Numerical Experiments} \label{sec:experiments}
			
			Our numerical experiments are conducted in a Matlab environment (version R2023b), 
			using 64-bit arithmetic, on a single core of a computing system equipped with an 
			Apple M1 Max processor and 64 GB of system memory. 
			Throughout the rest of our 
			experiments we consider $k=1,2,\ldots,50$. Due to the symmetry of the matrix $\mathbf{A}$, 
			the eigenvector $\bm{v}_1$ is perpendicular to all $N-1$ eigenvectors $\bm{v}_2,\ldots,\bm{v}_N$, 
			of the matrix $\mathbf{A}$, and thus a modulus of $\bm{v}_1^\top \bm{\widehat{x}}_{k}$ equal to one 
			implies that $\bm{\widehat{x}}_{k}$ is perpendicular to all $N$-1 non-sought eigendirections 
			of $\mathbf{A}$. 
			Therefore, monitoring how fast the modulus of (normalized) $\bm{\widehat{x}}_{k}$ approaches the unity 
			allows to determine how fast each algorithm converges to $\bm{v}_1$. In addition to varying 
			the total number of iterations $k$, we also consider different ratios $\mathbb{E}[T]/N$ 
			where for simplicity we set the variance of $T$ equal to zero, \emph{i.e.}, the value of $T$ 
			is constant across all iterations. Lower values of $\mathbb{E}[T]/N$ imply a smaller 
			cardinality of the row update subsets ${\mathcal T}_k$. 
			
			We consider six different algorithms, outlined as follows:
			\begin{enumerate}
				\item Inexact power iteration with each matrix-vector product modeled as 
				in (\ref{updzero2}). The legend associated with the plots 
				referring to this method is set as 
				$[\bm{\widehat{x}}_k]_i=[\bm{\widehat{x}}_{k-1}]_i,\ i\notin {\mathcal T}_k$.
				We normalize $\bm{\widehat{x}}_k$ at the end of the $k$-th iteration.
				\item Inexact power iteration with each matrix-vector product modeled as 
				in (\ref{updzero1}). The legend associated with the plots 
				referring to this method is set as 
				$[\bm{\widehat{x}}_k]_i=0,\ i\notin {\mathcal T}_k$.
				\item Algorithm \ref{alg1} with $[\bm{\pi}]_k=e^{-\zeta(k-1)}$ and $\zeta=\frac{N-\mathbb{E}[T]}{2N}$. The legend associated with the plots 
				referring to this method is set as “Alg. 1 (Expo)". 
				\item Algorithm \ref{alg1} with $[\bm{\pi}]_k=\frac{1}{1+e^{\alpha(k-\beta)}}$ where $\alpha=.3,\ \beta=5$. The legend associated with the plots referring to this method is set as “Alg. 1 (Log)". 
				\item Algorithm \ref{alg1} with $[\bm{\pi}]_k=1$ if $k \leq \lceil\mathbb{E}[T]/N\rceil$, and $[\bm{\pi}]_k=0$ otherwise. The legend associated with the plots referring to this method is set as “Alg. 1 (Step)". 
				\item Algorithm \ref{alg2} with a corresponding legend “Alg. 2".
			\end{enumerate}

			\subsection{Dense model problems}
			
			Our experimental dataset primarily consists of dense, random matrices with a 
			pre-determined spectrum and spectral gap. In particular, we generate the 
			symmetric matrix $\mathbf{A}$ by first creating a random $N\times N$ orthonormal matrix 
			$Q$ and then multiplying $\mathbf{A}=\mathbf{Q}\mathbf{\Lambda} \mathbf{Q}^\top$ where 
			$\mathbf{\Lambda}=\mathrm{diag}(\lambda_1,\ldots,\lambda_N)$ is a diagonal 
			matrix whose prescribed values are the eigenvalues of the matrix $\mathbf{A}$. 
			We consider four different sets of eigenvalues $\lambda_j,\ j=1,\ldots,N$: 
			\begin{enumerate}
				\item $\lambda_1=1,\ \lambda_j \in (0.0,0.25),\ j=2,\ldots,N$
				\item $\lambda_1=1,\ \lambda_j \in (0.0,0.5),\ j=2,\ldots,N$
				\item $\lambda_1=1,\ \lambda_j \in (-0.25,0.25),\ j=2,\ldots,N$
				\item $\lambda_1=1,\ \lambda_j \in (-0.5,0.5),\ j=2,\ldots,N$
			\end{enumerate}
			The first two spectrums have only positive eigenvalues and the eigenvalues 
			$\lambda_2,\ldots,\lambda_N$, are clustered in the intervals $(0.0,0.25)$ 
			and $(0.0,0.5)$, respectively. The last two spectrums have both positive 
			and negative eigenvalues and the eigenvalues 
			$\lambda_2,\ldots,\lambda_N$, are clustered in the intervals $(-0.25,0.25)$ 
			and $(-0.5,0.5)$, respectively. Unless mentioned otherwise, all figures 
			plotted in the following represent empirical means obtained via twenty 
			independent runs with the same initial guesses for the six algorithms.


			\begin{figure*}[ht]
				\centering
				\includegraphics[width=0.32\linewidth]{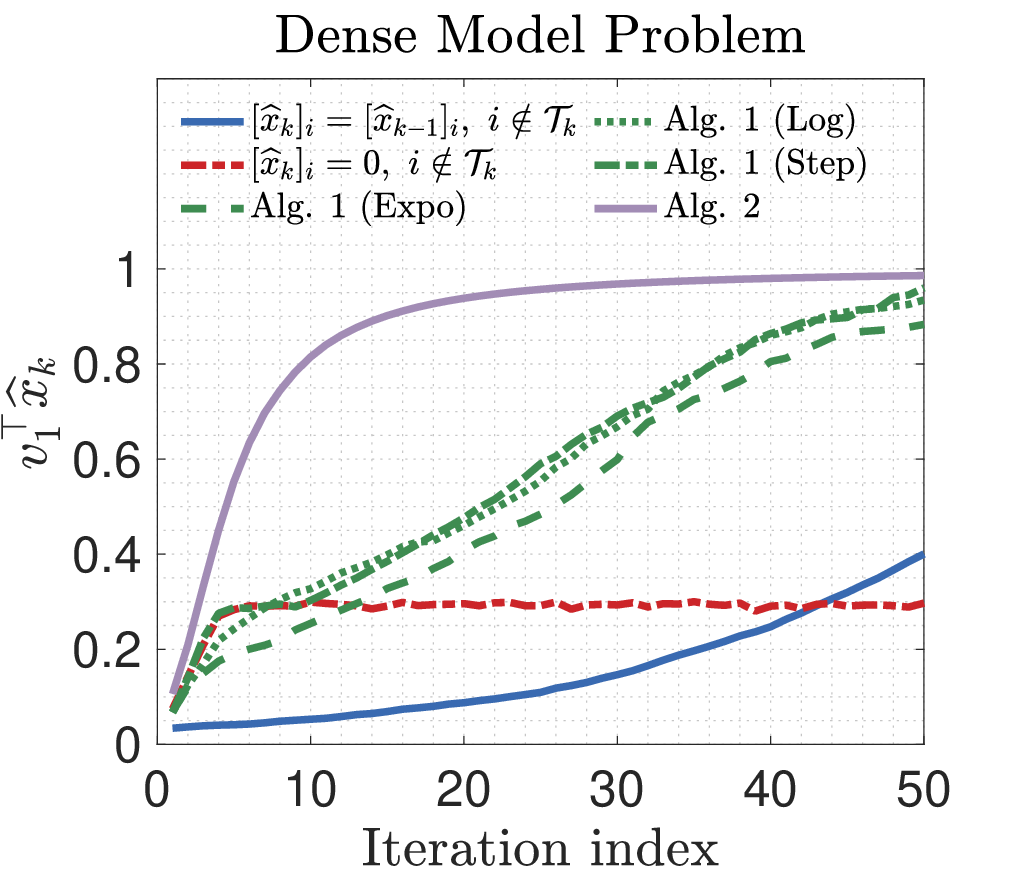}   
				\includegraphics[width=0.32\linewidth]{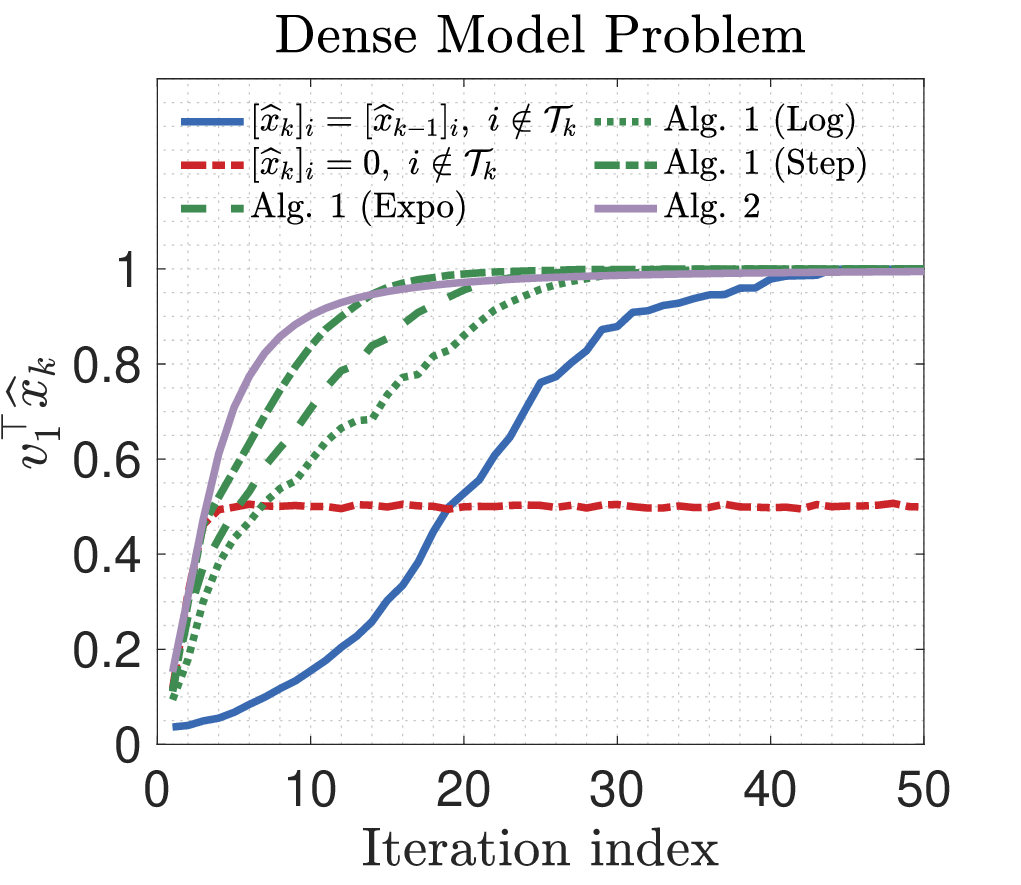}    
				\includegraphics[width=0.32\linewidth]{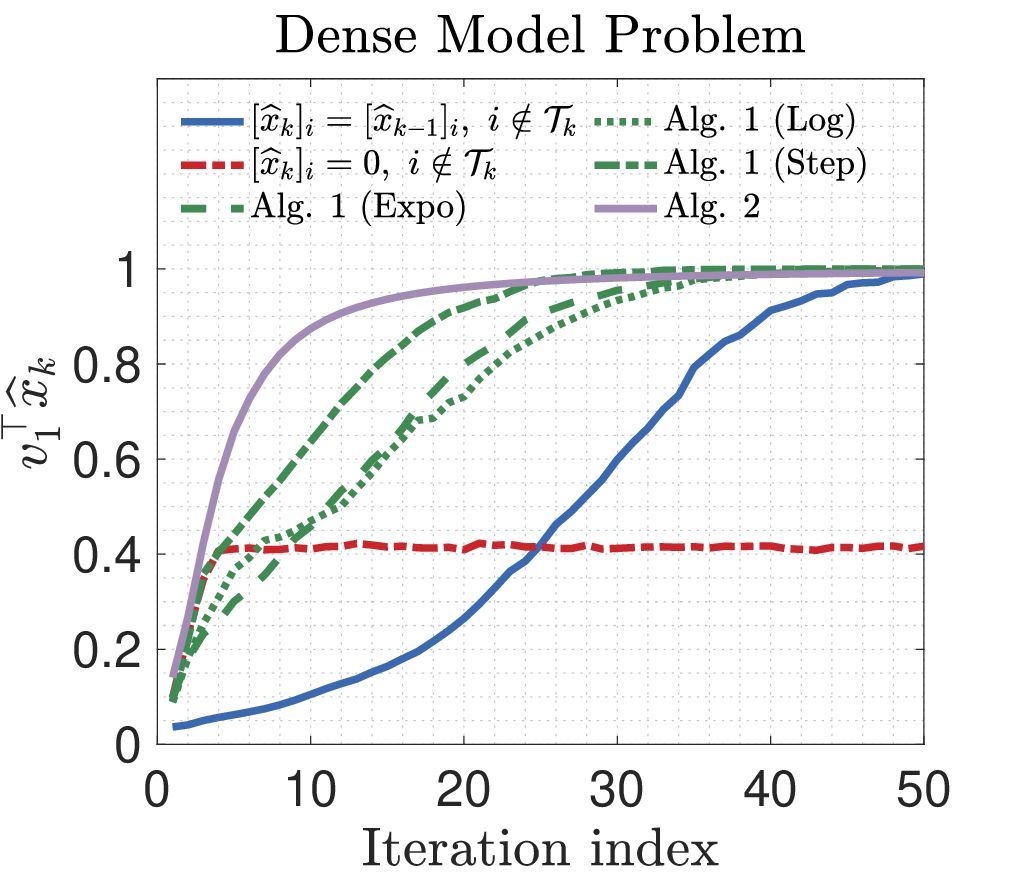} 
				\caption{{\it Convergence plots for the dense model matrix $\mathbf{A}$ with spectrum designed as 
						$\lambda_1=1,\ \lambda_j \in (0.0,0.25),\ j=2,\ldots,N$. Left: $\mathbb{E}[T]/N=0.1$. Center: $\mathbb{E}[T]/N=0.2$. Right: $\mathbb{E}[T]/N=0.3$.}}\label{fig3}
			\end{figure*}
			
			\begin{figure*}[ht]
				\centering
				\includegraphics[width=0.32\linewidth]{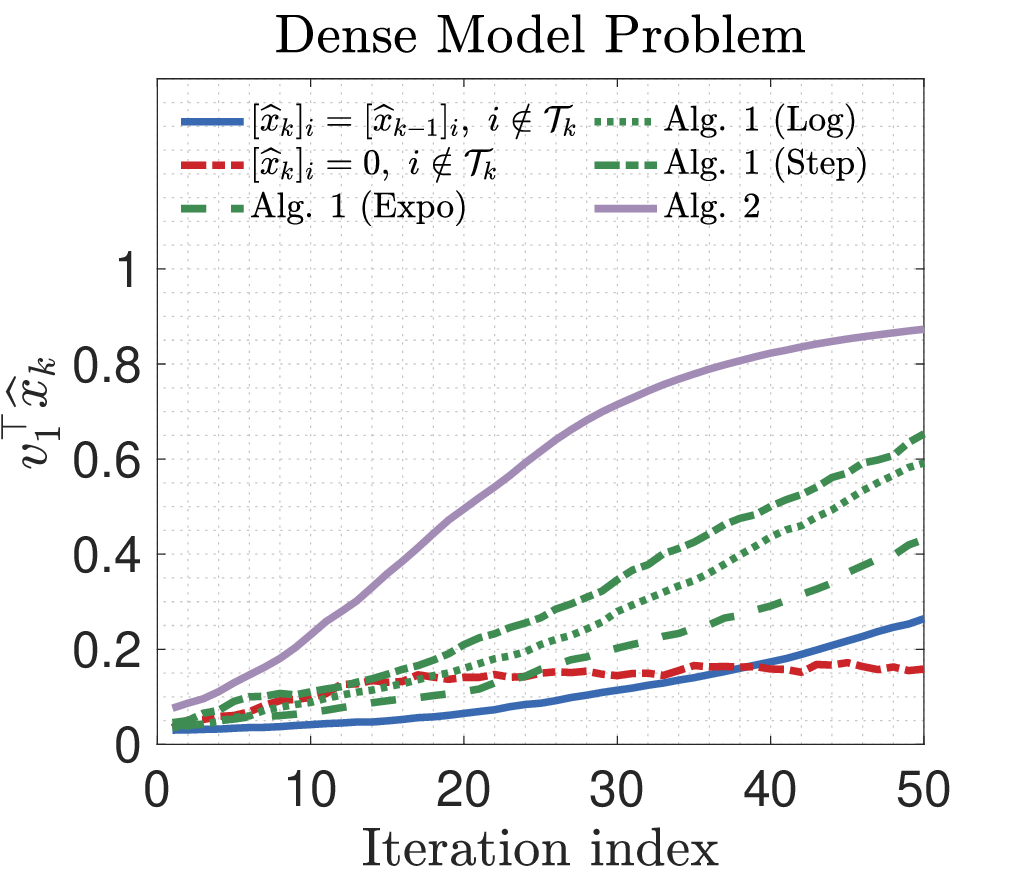}   
				\includegraphics[width=0.32\linewidth]{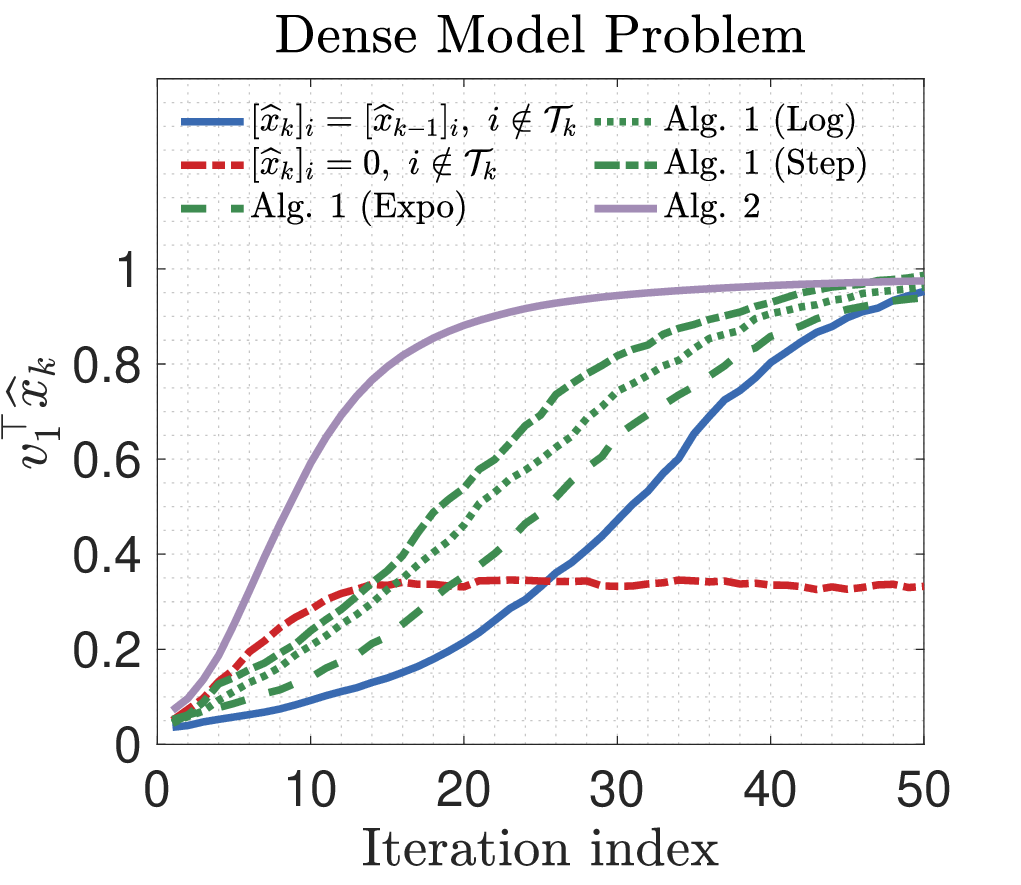}    
				\includegraphics[width=0.32\linewidth]{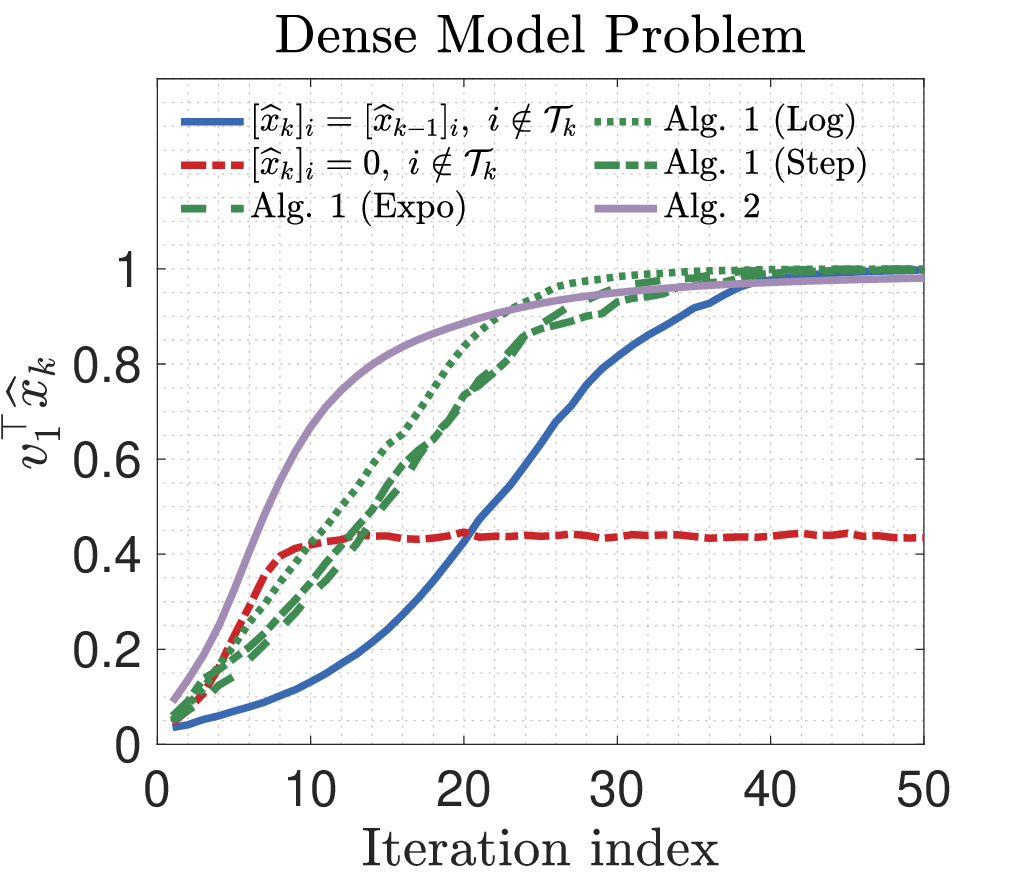} 
				\caption{{\it Convergence plots for the dense model matrix $\mathbf{A}$ with spectrum designed as 
						$\lambda_1=1,\ \lambda_j \in (0.0,0.5),\ j=2,\ldots,N$. Left: $\mathbb{E}[T]/N=0.1$. Center: $\mathbb{E}[T]/N=0.2$. Right: $\mathbb{E}[T]/N=0.3$.}}\label{fig4}
			\end{figure*}
			
			\begin{figure*}[ht]
				\centering
				\includegraphics[width=0.32\linewidth]{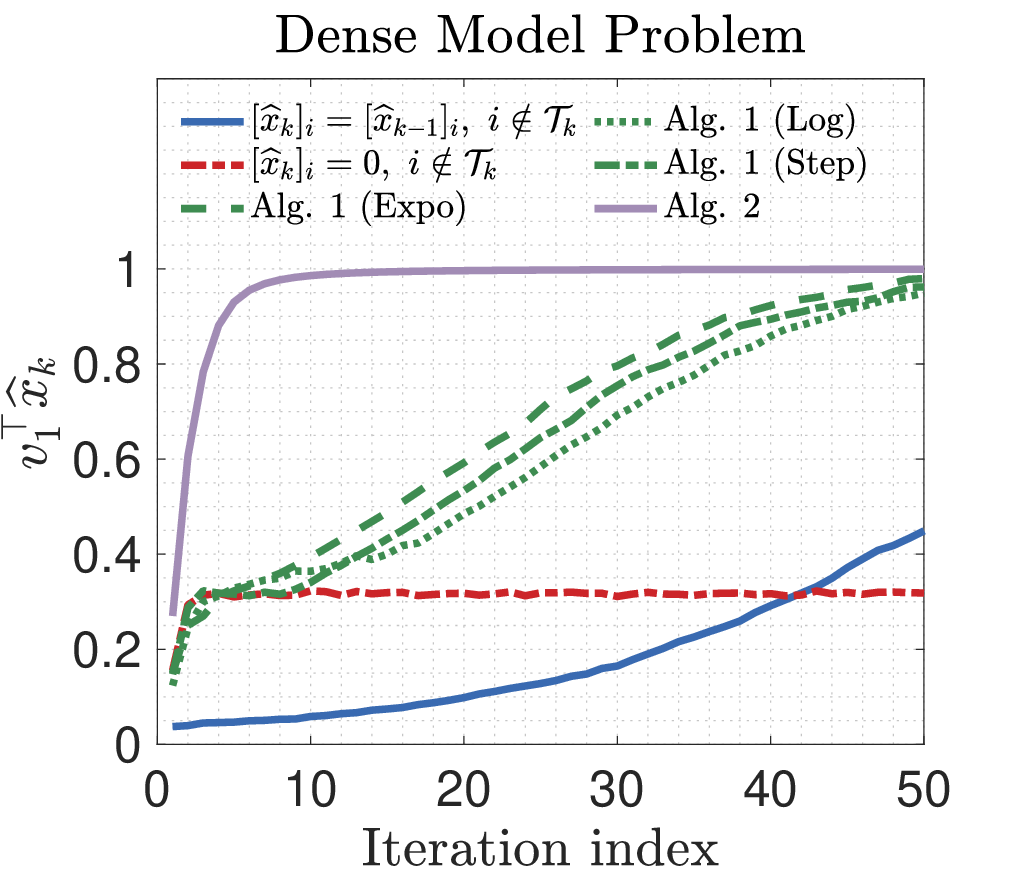}   
				\includegraphics[width=0.32\linewidth]{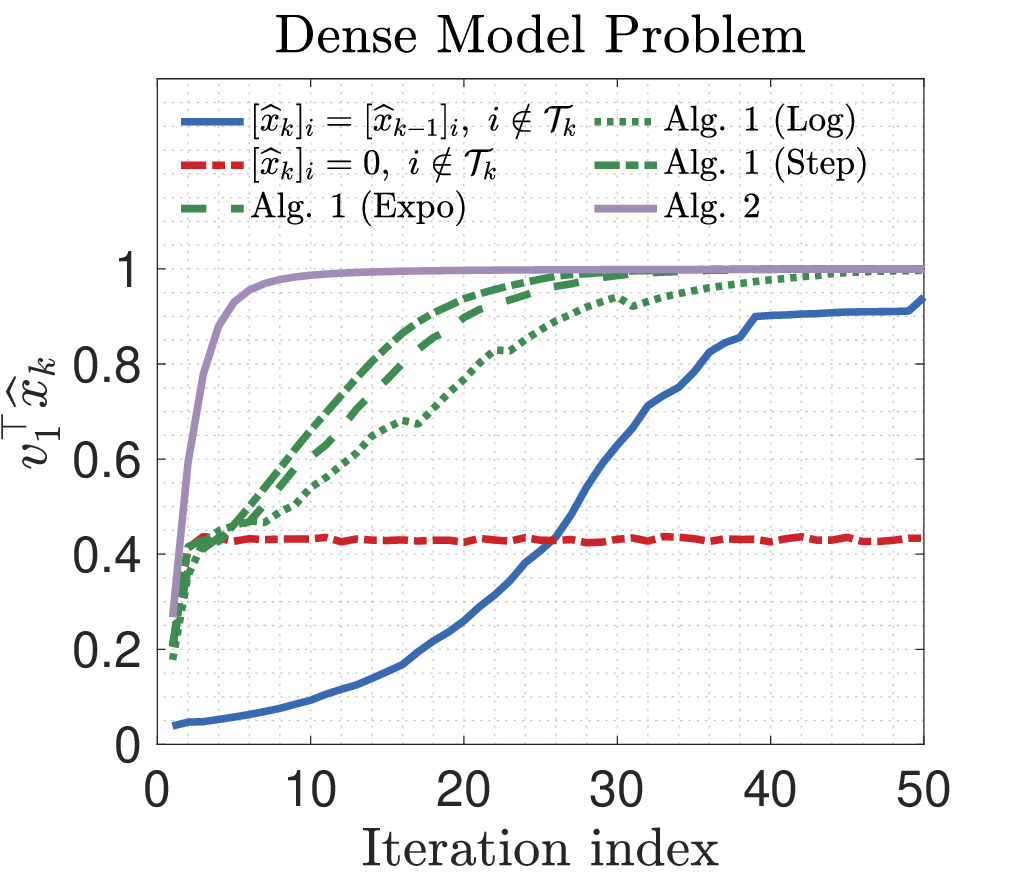}    
				\includegraphics[width=0.32\linewidth]{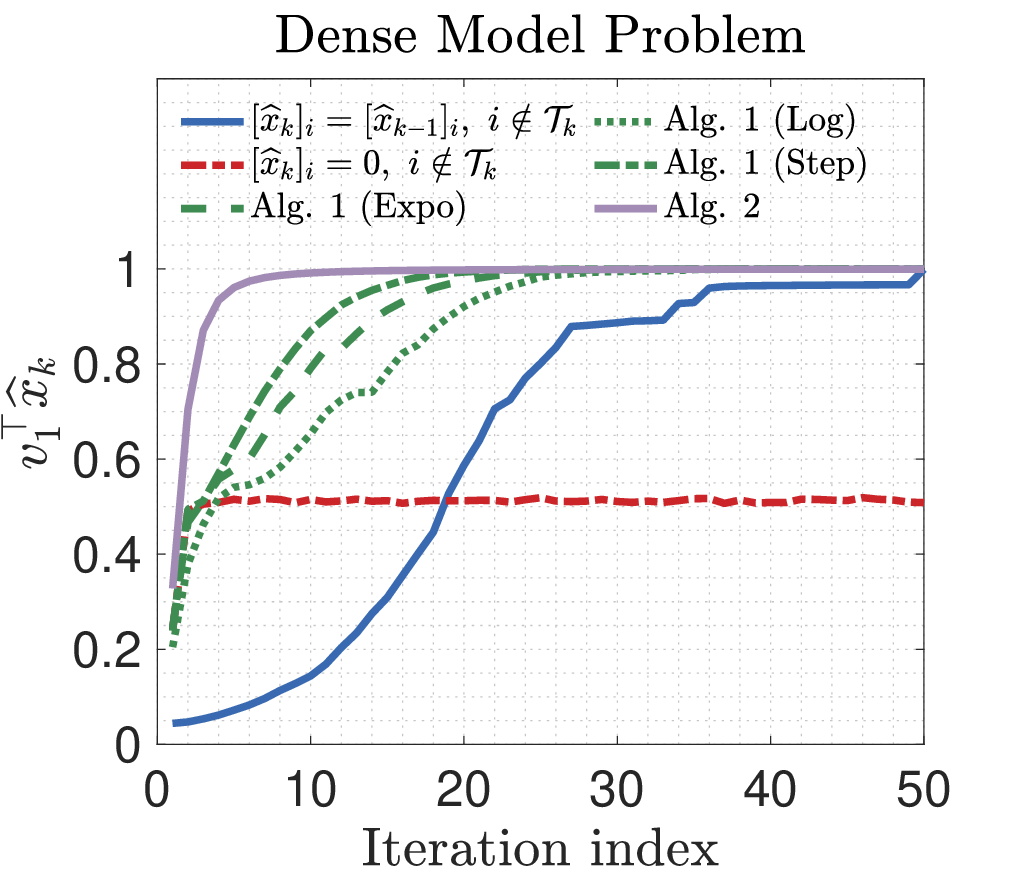} 
				\caption{{\it Convergence plots for the dense model matrix $\mathbf{A}$ with spectrum designed as 
						$\lambda_1=1,\ \lambda_j \in (-0.25,0.25),\ j=2,\ldots,N$. Left: $\mathbb{E}[T]/N=0.1$. Center: $\mathbb{E}[T]/N=0.2$. Right: $\mathbb{E}[T]/N=0.3$.}}\label{fig5}
			\end{figure*}
			
			\begin{figure*}[ht]
				\centering
				\includegraphics[width=0.32\linewidth]{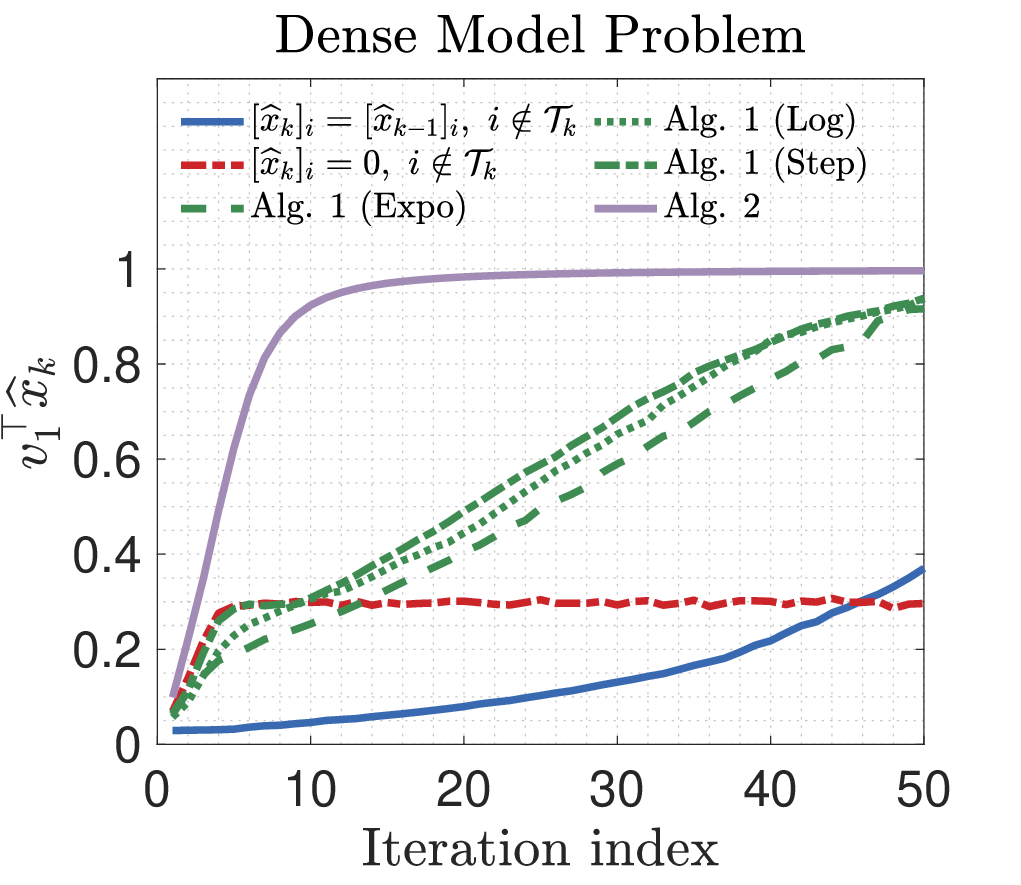}   
				\includegraphics[width=0.32\linewidth]{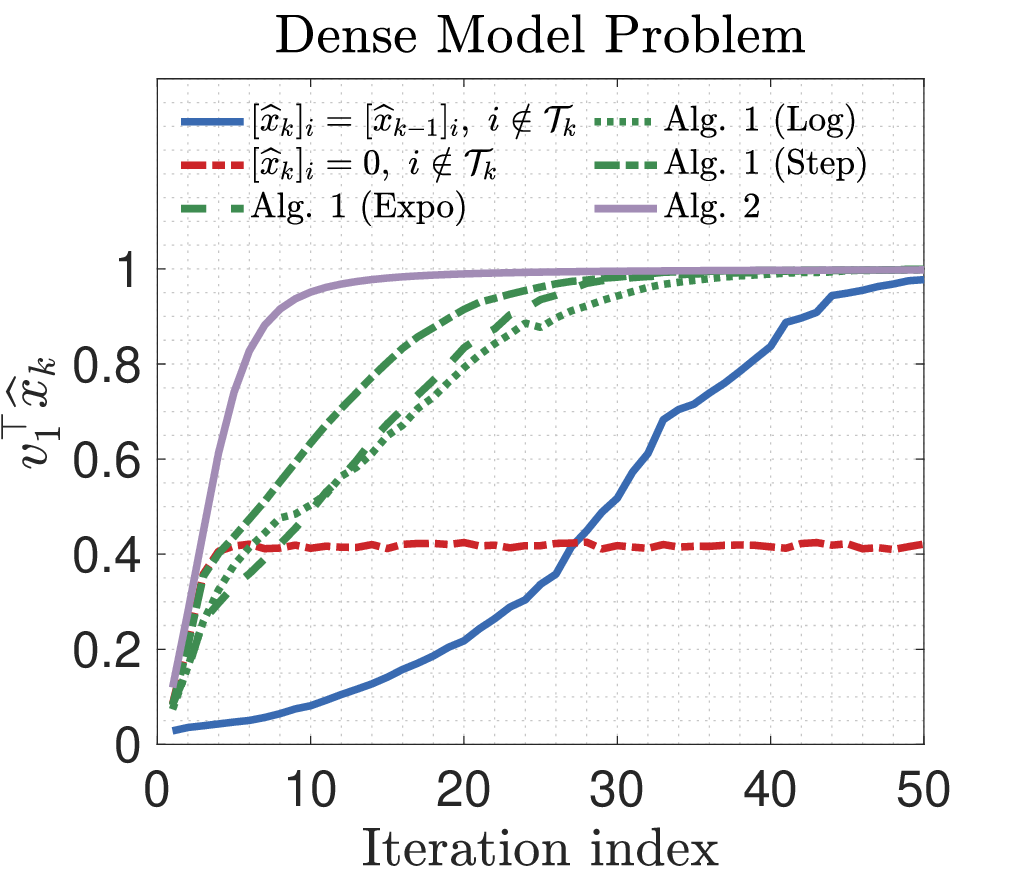}    
				\includegraphics[width=0.32\linewidth]{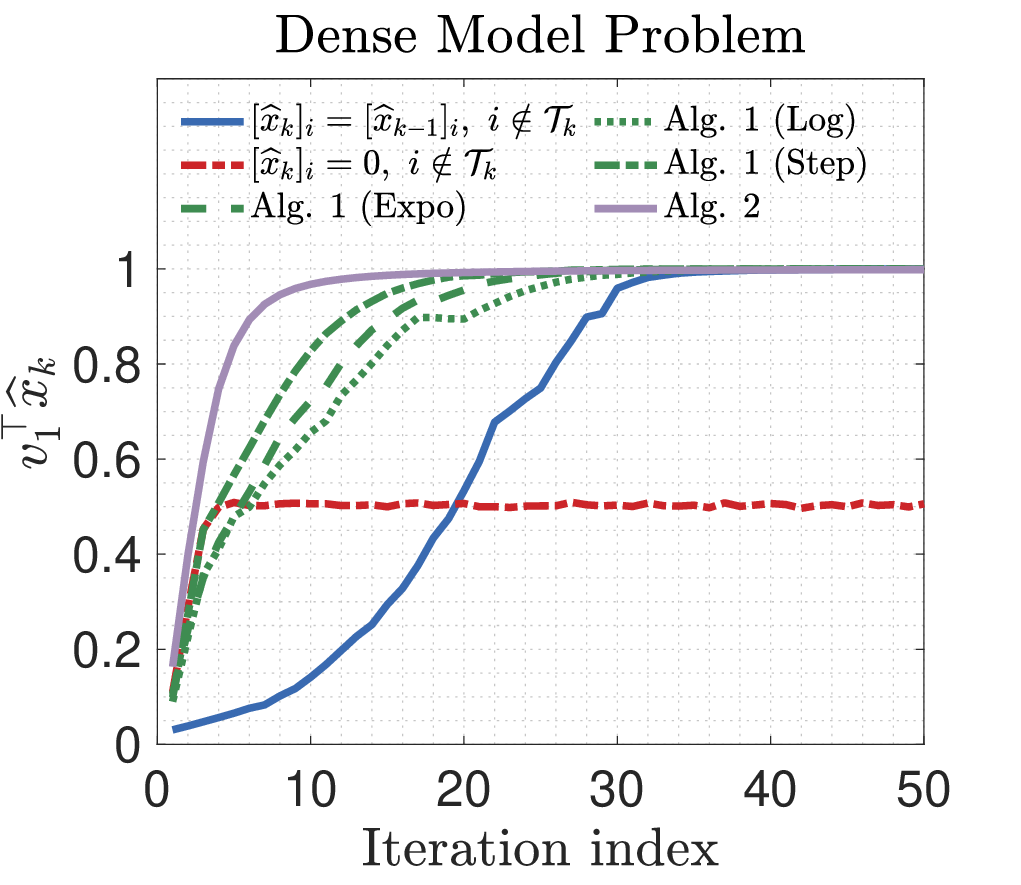} 
				\caption{{\it Convergence plots for the dense model matrix $\mathbf{A}$ with spectrum designed as 
						$\lambda_1=1,\ \lambda_j \in (-0.5,0.5),\ j=2,\ldots,N$. Left: $\mathbb{E}[T]/N=0.1$. Center: $\mathbb{E}[T]/N=0.2$. Right: $\mathbb{E}[T]/N=0.3$.}}\label{fig6}
			\end{figure*}
			
			Figures \ref{fig3}-\ref{fig6} plot the modulus of $\bm{v}_1^\top \bm{\widehat{x}}_{k}$ as a function of 
			the iteration index $k=1,\ldots,50$. The expected number $\mathbb{E}[T]$ of updated entries 
			per iteration for each one of the six algorithms we compare is set such that $\mathbb{E}[T]/N=0.1$ (left), $\mathbb{E}[T]/N=0.2$ (middle), and $\mathbb{E}[T]/N=0.3$ (right). Among the six 
			algorithms we compare, inexact power iteration with the matrix-vector product modeled 
			as in (\ref{updzero1}) performs the worst due to the stagnation in the accuracy resulting from 
			setting $N-T_k$ entries of $\bm{\widehat{x}}_k$ equal to zero. Naturally, as $T_k$ increases, \emph{i.e.}, as 
			$\mathbb{E}[T]/N$ becomes larger, fewer entries are set equal to zero in each iteration and the 
			stagnation accuracy level increases. Nonetheless, what is interesting about this algorithm is 
			that initially is more accurate than inexact power iteration with the matrix-vector product modeled 
			as in (\ref{updzero2}). The reason for this phenomenon is that the former algorithm is -in expectation- equivalent to classical power iteration and and variance is smaller for small values of $k$. On the other hand, the latter algorithm is equivalent to asynchronous power iteration with normalization, and is guaranteed to converge to the dominant eigenvector $\bm{v}_1$ for all matrices $\mathbf{A}$ we consider in this section. The best performing approaches for all problems considered in this section are the two algorithms proposed in this paper, namely Algorithm \ref{alg1} and Algorithm \ref{alg2}. In particular, 
			Algorithm \ref{alg2} achieved the best overall accuracy with a steep increase during the early iterations. On the other hand, the performance accuracy of Algorithm \ref{alg1} lies between that 
			of Algorithm \ref{alg2} and inexact power iteration with a matrix-vector product dictated by  
			(\ref{updzero2}). Among the three different options of Algorithm \ref{alg1}, the worst performing option was the one where the probability vector $\bm{\pi}$ is set as $[\bm{\pi}]_k=\frac{1}{1+e^{\alpha(k-\beta)}}$ (logistic function). The reason for this is that the curve determined by the entries of $\bm{\pi}$ is 
			less steep and thus makes it more likely that Algorithm \ref{alg1} might choose the update model (\ref{updzero1}) at a later iteration, thus zeroing many accurate components of the approximation $\bm{\widehat{x}}_k$. On the other hand, performing a deterministic switch as in “Alg. 1 (Step)" and 
			setting the probability vector as $[\bm{\pi}]_k=e^{-\zeta(k-1)}$ as in “Alg. 1 (Expo)" performed similarly.
			
			\begin{table}[t]
				\centering
				\caption{Iteration count required to reach the tolerance
						$1-|\bm{v}_1^\top \bm{\widehat{x}}_k|\le \phi$ for the dense model problem with
						$\lambda_1=1$ and $\lambda_j\in(-0.5,0.5)$, $j=2,\ldots,N$. The symbol
						``$F$'' indicates that the tolerance was not reached within fifty iterations.}
				\label{tab:iteration-counts-dense}
				\begin{tabular}{c c c c c c}
					\hline
					$E[T]/N$ & $\phi$ &
					Async. $(\bm{\widehat{x}}_k)_i=(\bm{\widehat{x}}_{k-1})_i$ &
					Alg.~1 Exp. &
					Alg.~1 Log. &
					Alg.~2 \\
					\hline
					$0.1$ & $0.3$ & $F$ & $36$& $32$& $7$ \\
					$0.1$ & $0.15$ & $F$ & $45$ & $40$& $11$ \\
					$0.1$ & $0.05$ & $F$ & $F$ & $50$& $18$ \\
					\hline
					$0.2$ & $0.3$ & $33$ & $17$ & $18$ & $6$\\
					$0.2$ & $0.15$ & $40$ & $20$ & $21$ & $9$ \\
					$0.2$ & $0.05$ & $45$ & $28$ & $30$ & $14$ \\
					\hline
					$0.3$ & $0.3$ & $25$ & $10$& $12$ & $5$ \\
					$0.3$ & $0.15$ & $28$ & $14$ & $16$ & $8$ \\
					$0.3$ & $0.05$ & $30$ & $20$ & $23$ & $12$ \\
					\hline
				\end{tabular}
			\end{table}
			Table \ref{tab:iteration-counts-dense} lists the iteration count required to reach the tolerance
				$1-|\bm{v}_1^\top \bm{\widehat{x}}_k|\le \phi,\ \phi\in \{0.3,\ 0.15,\ 0.05\}$ for the dense model problem with $\lambda_1=1$ and $\lambda_j\in(-0.5,0.5)$, $j=2,\ldots,N$. Across the tested options, 
				the proposed algorithms typically attain a given tolerance in fewer iterations than the asynchronous baseline, especially for small and moderate update ratios $E[T]/N$. The dependence on $E[T]/N$ is also consistent with the theoretical results presented so far in this paper, i.e., smaller values of $E[T]/N$ reduce the per-iteration cost and communication volume but increase the variance of the zero-filling model and slow down the asynchronous model by shrinking the effective spectral gap of the expected iteration matrix. Conversely, increasing $E[T]/N$ improves stability and reduces the number of iterations needed to reach a fixed tolerance, but at the cost of more observed rows or columns per iteration.

			\begin{figure*}[ht]
				\centering
				\includegraphics[width=0.32\linewidth]{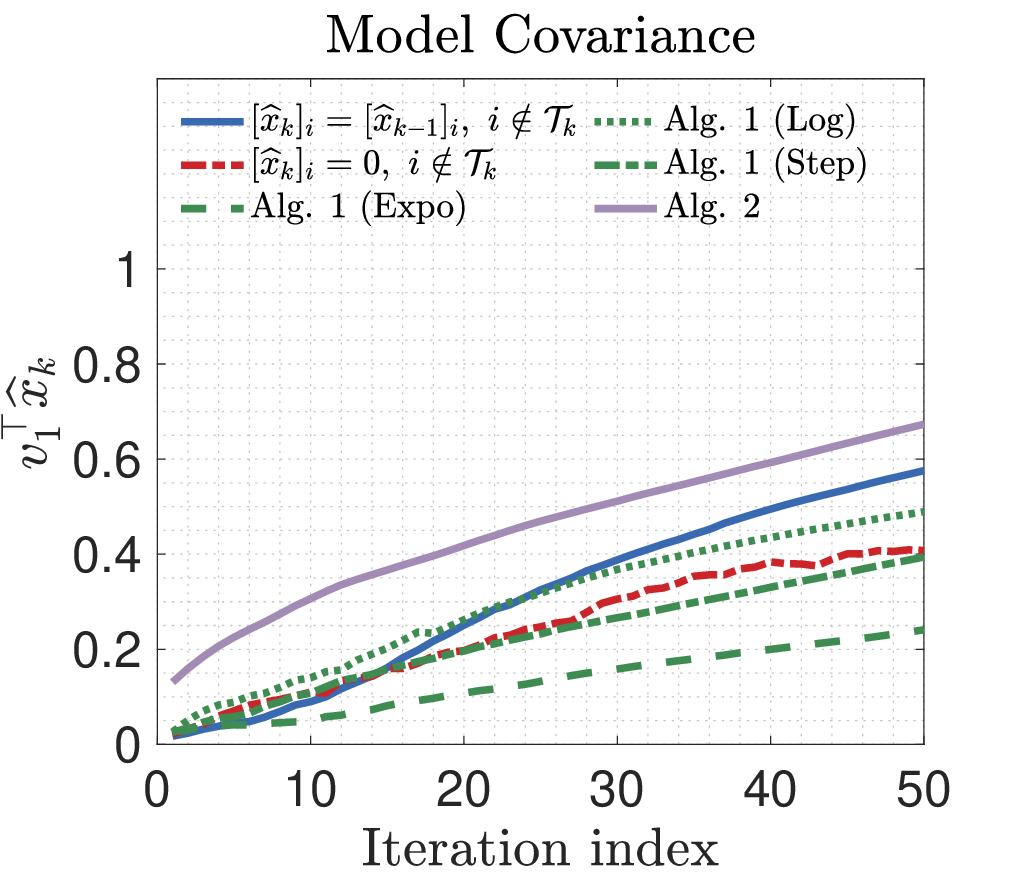}   
				\includegraphics[width=0.32\linewidth]{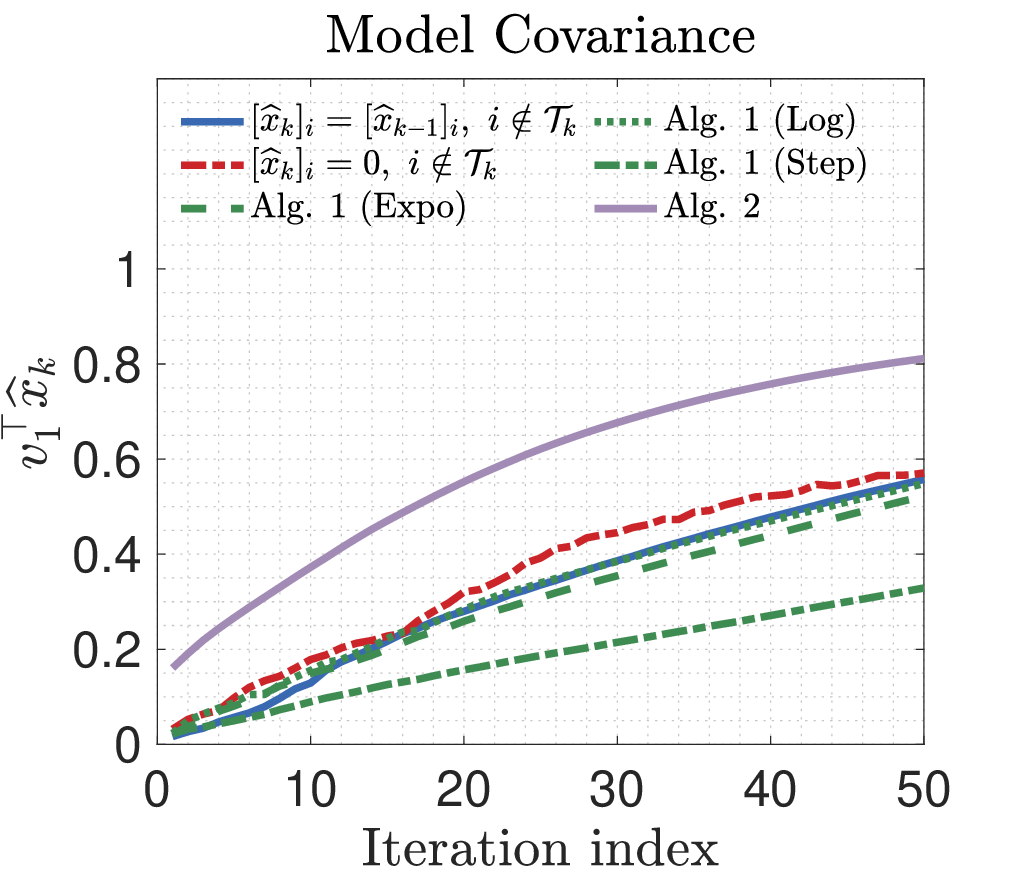}    
				\includegraphics[width=0.32\linewidth]{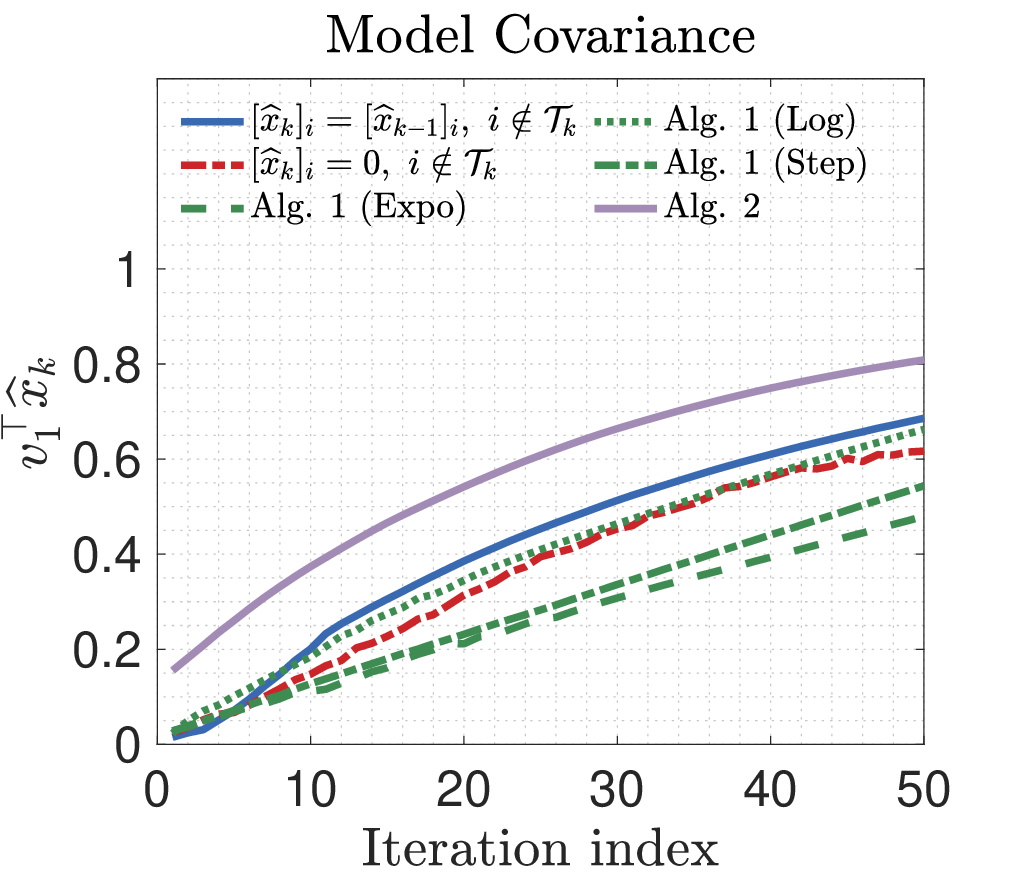} 
				\caption{{\it Convergence plots for the matrix $\mathbf{A}=\mathbf{XX}^\top$ where $\mathbf{X}$ is a random matrix of size $1024\times 256$ with zero column mean. Left: $\mathbb{E}[T]/N=0.3$. Center: $\mathbb{E}[T]/N=0.5$. Right: $\mathbb{E}[T]/N=0.7$.}}\label{fig8}
			\end{figure*}
			Finally, we consider the computation of the dominant eigenvector of sample covariance 
			matrices $\mathbf{A}=\frac{1}{K-1}\mathbf{XX}^{\top}$ where $\mathbf{X}\in \mathbb{R}^{N\times M},\ M\ll N$, is a 
			thin-and-tall matrix with zero column mean and rank equal to $M$. Thus, the rank of the matrix 
			$\mathbf{A}$ is equal to $M$ as well, \emph{i.e.}, $\mathbf{A}$ has $N-M$ eigenvalues of zero modulus. We scale the 
			non-zero eigenvalues of the matrix $\mathbf{A}$ such that $\lambda_1=1$ and apply all five variants of 
			power iteration. The matrix $\mathbf{A}$ is created by setting $N=1024,\ M=256$, and sampling each 
			entry of the matrix $\mathbf{X}$ prior to subtracting its column mean from the standard normal distribution.
			Figure \ref{fig8} plot the modulus of $\bm{v}_1^\top \bm{\widehat{x}}_k$ as a function of the iteration 
			index $k\in [1,50]$ where the expected number of updated entries per iteration 
			is sampled so that $\mathbb{E}[T]/N=\{0.7,0.8,0.9\}$ for each one of the six variants 
			of power iteration. The results are in line with the ones obtained for the full-rank covariance matrices.
			
			\subsection{Adjacency matrices of graphs}
			
			Our final set of experiments considers the application of asynchronous and partial power iteration on sparse matrices representing adjacency matrices of graphs. We consider three different networks. The first network, \emph{Arenas/email}, represents a network of e-mail interchanges between $N=1,133$ members of the University Rovira i Virgili (Tarragona) \cite{guimera2003self,davis2011university}. The second network, \emph{socfb-MIT}, represents an undirected, unweighted, represents a social friendship network between $N=6,403$ members of the Massachusetts Institute of Technology \cite{nr-aaai15}. Finally, our third network, \emph{Protein interaction}, represents observed physical interactions between $N=4,388$ proteins. Each edge represents a protein and an edge exists between two proteins if they have been observed to interact. Interactions between proteins are bidirectional so the set of interactions between proteins in an organism forms an undirected unweighted network  \cite{uetz2000comprehensive}.
			
			\begin{figure*}[ht]
				\centering
				\includegraphics[width=0.32\linewidth]{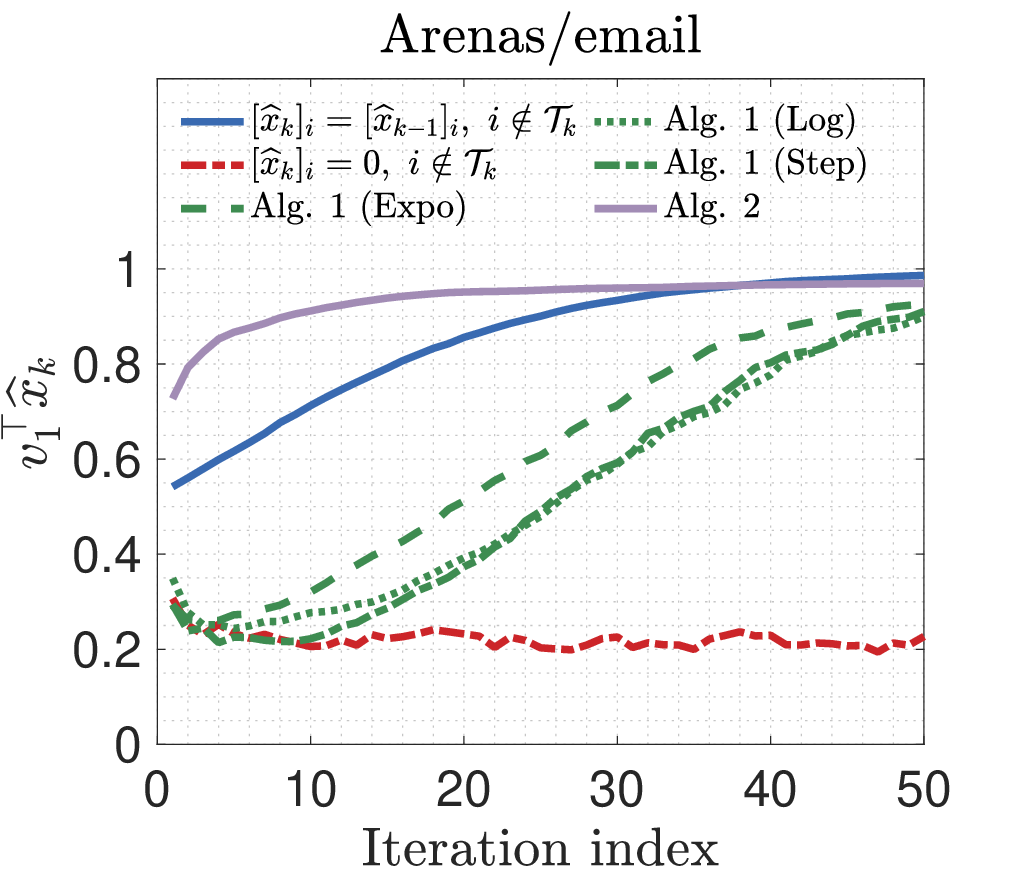}   
				\includegraphics[width=0.32\linewidth]{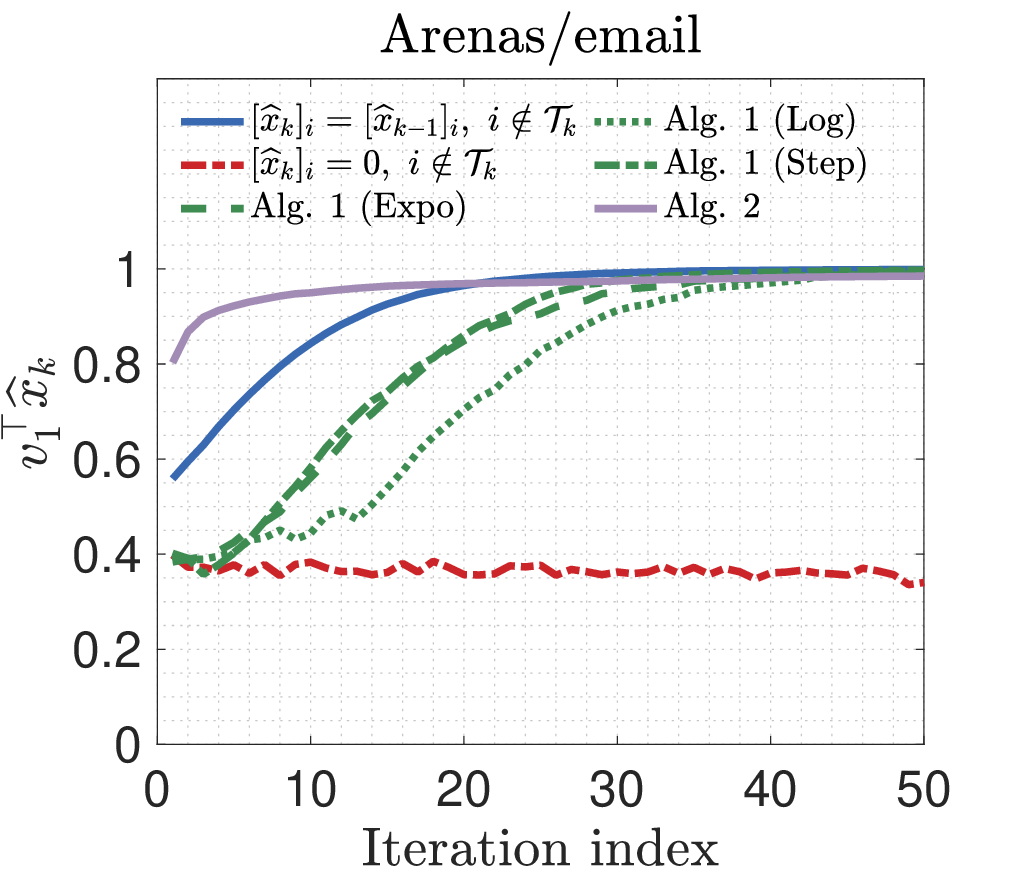}    
				\includegraphics[width=0.32\linewidth]{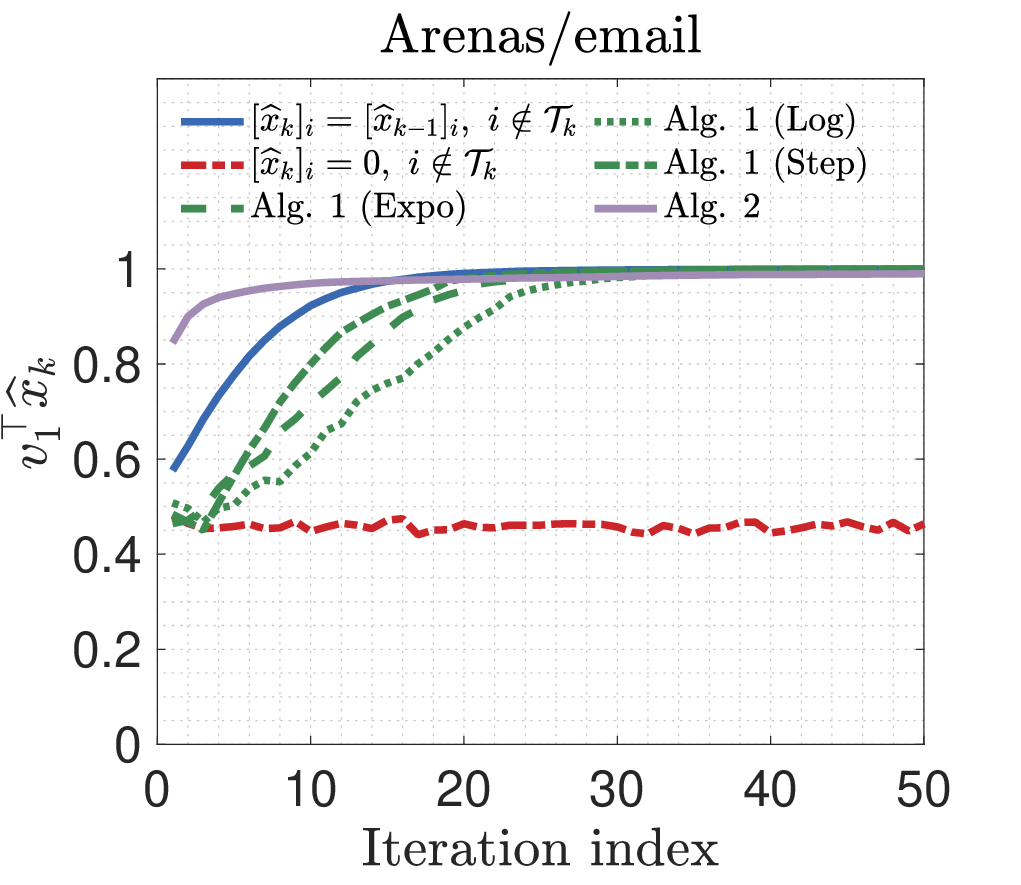} 
				\includegraphics[width=0.32\linewidth]{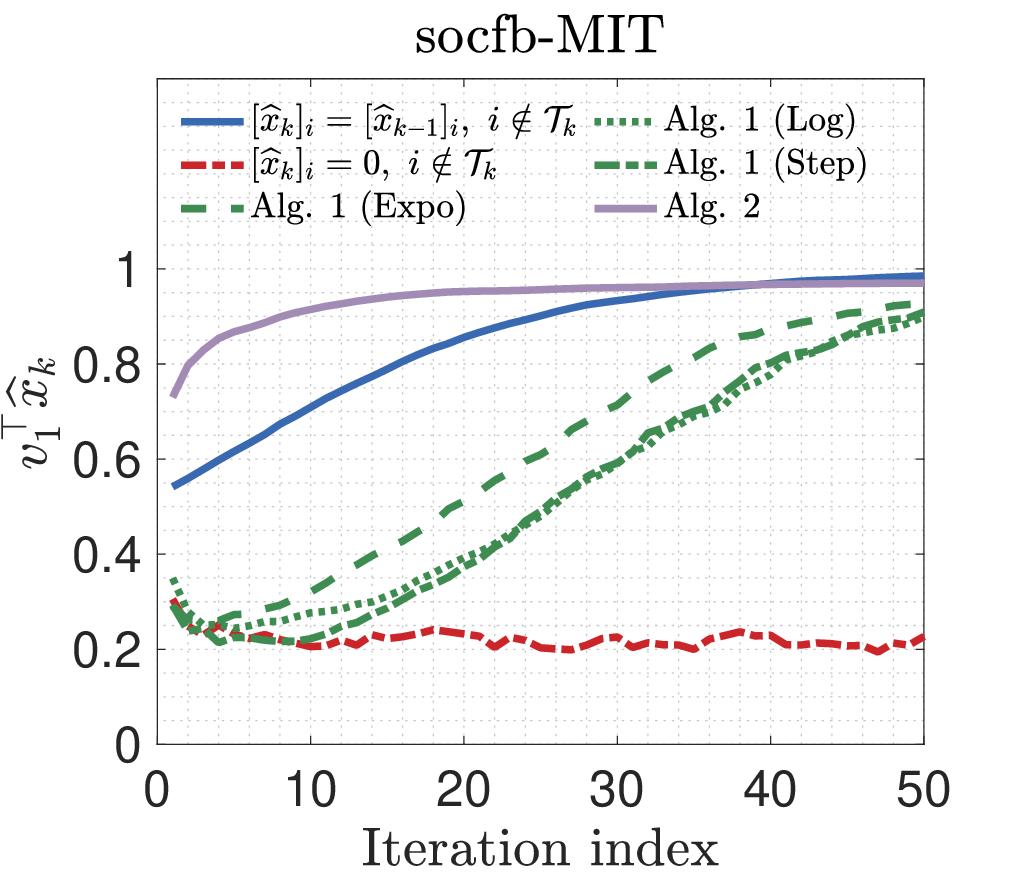}   
				\includegraphics[width=0.32\linewidth]{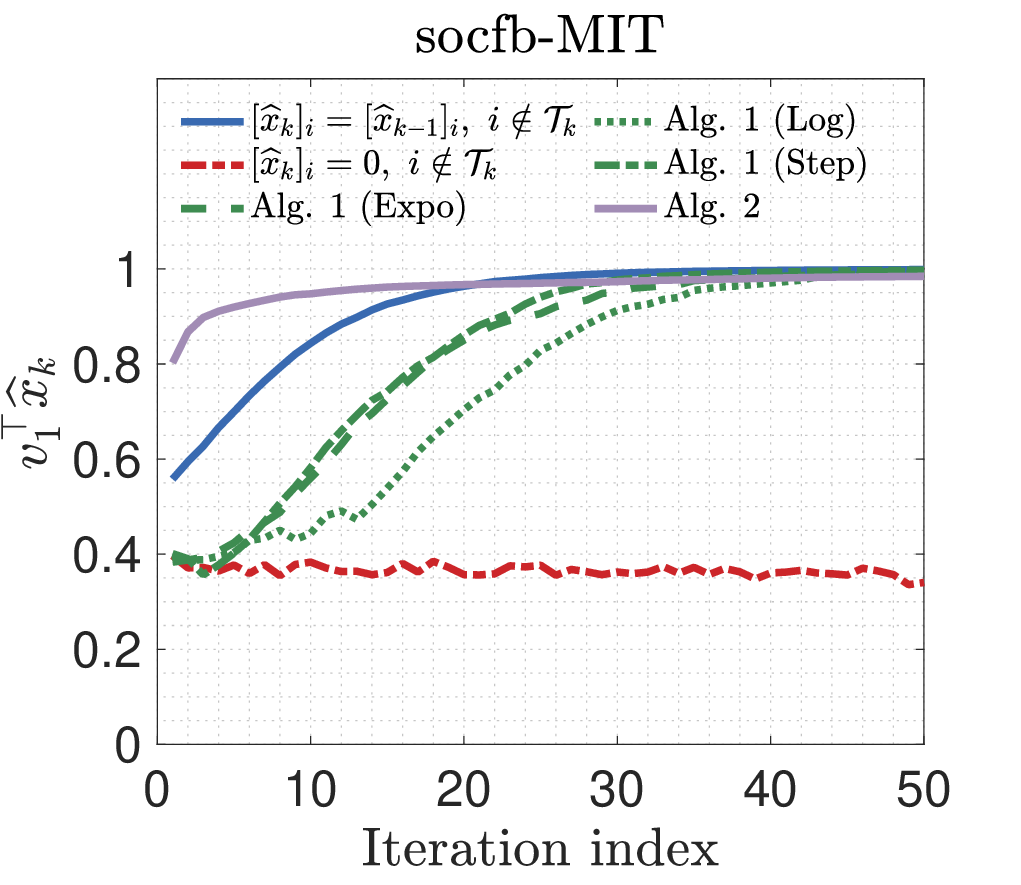}    
				\includegraphics[width=0.32\linewidth]{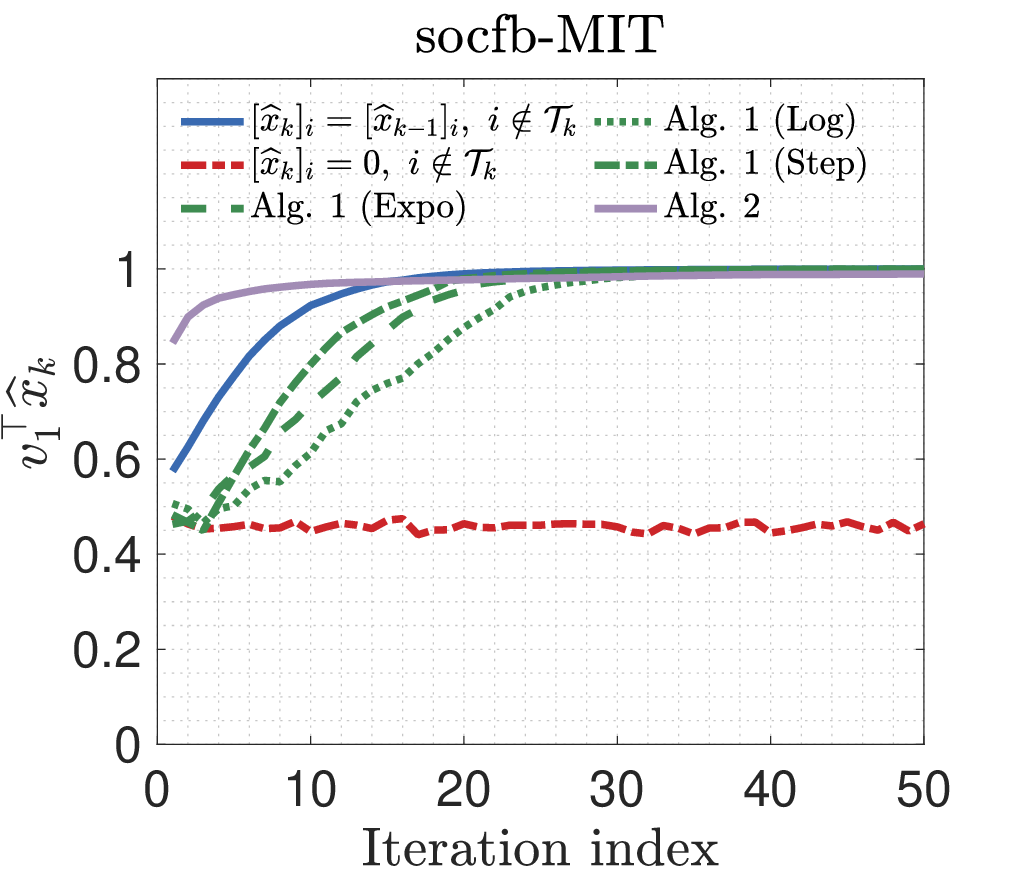} 
				\includegraphics[width=0.32\linewidth]{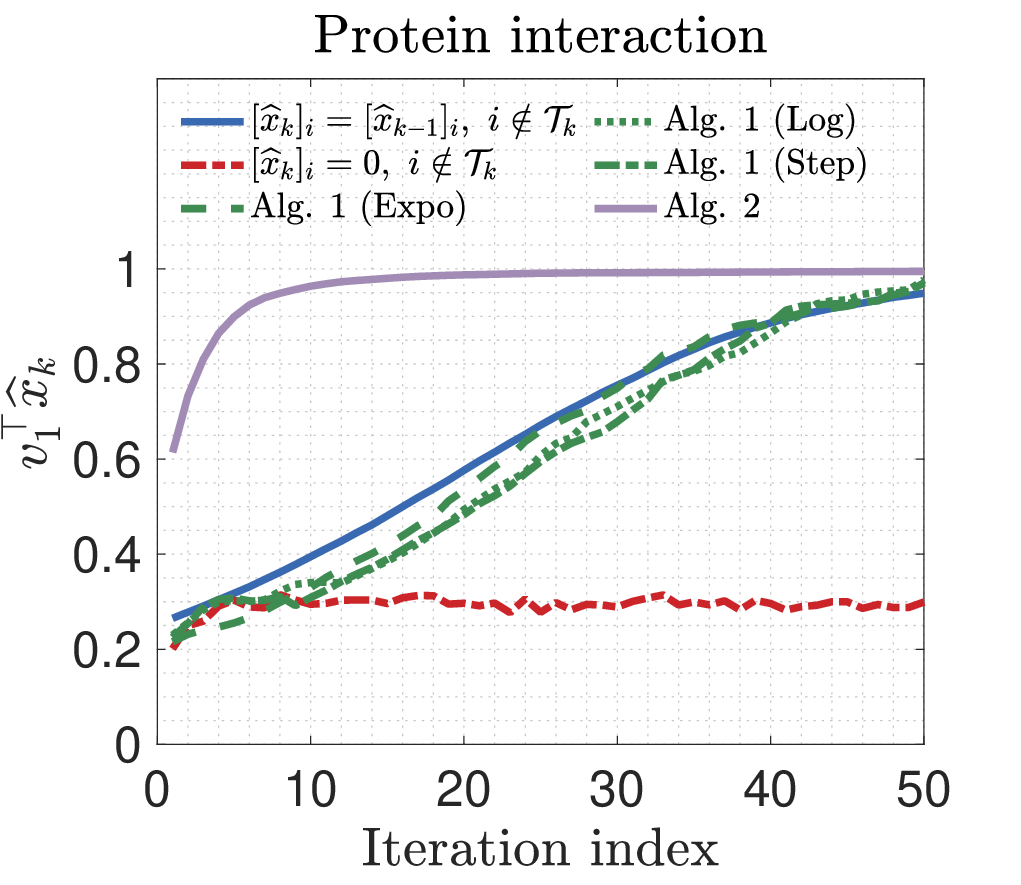}   
				\includegraphics[width=0.32\linewidth]{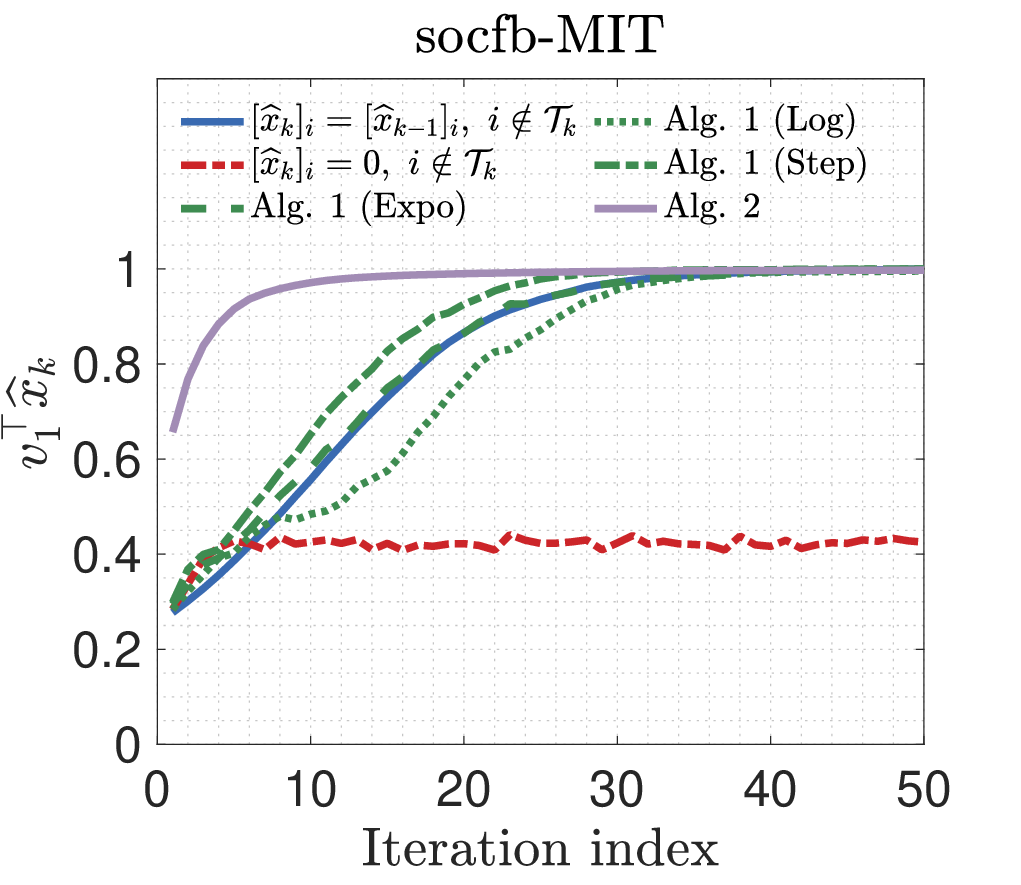}    
				\includegraphics[width=0.32\linewidth]{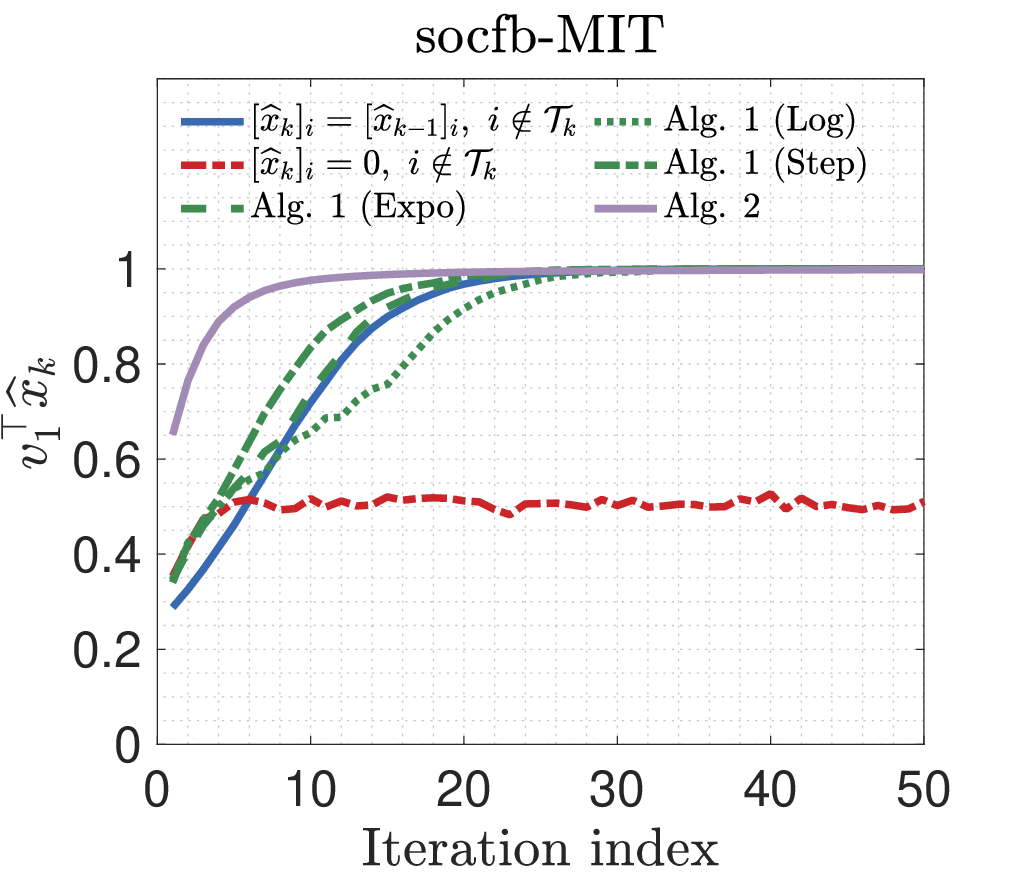} 
				\caption{{\it Convergence plots for the sparse adjacency matrices. Left: $\mathbb{E}[T]/N=0.1$. Center: $\mathbb{E}[T]/N=0.2$. Right: $\mathbb{E}[T]/N=0.3$.}}\label{fig9}
			\end{figure*}
			
			Figure \ref{fig9} plots the modulus of $\bm{v}_1^\top \bm{\widehat{x}}_{k}$ as a function of 
			the iteration index $k=1,\ldots,50$, for the three networks described above. In 
			agreement with the results for the dense model problems, Algorithm \ref{alg2} performed 
			the best; especially for lower values of $\mathbb{E}[T]$. On the other hand, asynchronous 
			power iteration generally performed better than all variants of Algorithm \ref{alg1} 
			due to the smaller spectral gap of the adjacency matrices $\mathbf{A}$ compared to the previous 
			cases.
			
			\subsection{Further discussion}
			
			The two partial power iteration algorithms presented in this paper can be a good  
			fit for hybrid cloud architectures implemented via the controller-worker model, 
			a parallel computing paradigm for distributed memory architectures  where the 
			controller is responsible for facilitating parallel computations by distributing 
			and receiving data to/from the workers \cite{pacheco2011introduction}. Each worker 
			operates independently from the rest of the workers and communicates directly only 
			with the controller. 
			In large-scale installations, each worker can represent one or more servers, connected 
			in a flat or hierarchical topology, and the servers incorporated in each worker might 
			exercise different processing capabilities and have varying latency due to their reliance 
			on commodity hardware required for cost-effectiveness and oversubscription, \emph{i.e.}, workers 
			are generally executing several third-party tasks that surpass the total processing capabilities. 
			As a result, while some workers might have already finished their 
			computations, some other might have not yet initiated their computations 
			\cite{severinson2022straggler}. 
			
			The phenomenon where some workers are non-responsive or take significantly longer to complete 
			their tasks compared to others -thus leading to delays in the overall completion time- is known 
			as \emph{straggling}, and these workers are commonly referred to as \emph{stragglers} 
			\cite{dean2013tail}. Stragglers degrade the parallel efficiency of distributed systems and can 
			increase the required wall-clock time considerably. Furthermore, the presence of straggling is dynamic 
			and unpredictable \cite{javadi2017dial,maji2015ice,kalantzis2025straggler}, and its effect on 
			throughput can vary considerably \cite{wang2018peeking,severinson2022straggler}. Controller-worker 
			models can find various forms of implementations also in federated learning for 
			the purpose of building AI models~\cite{Wen2023,roth2022nvidia}. Straggling appears naturally due 
			to delay in processing or communication, given that federated learning is being carried out globally 
			in the practice of building large-scale machine learning and AI models~\cite{roth2022nvidia}.

			\section{Conclusion} \label{sec:conclusions}
			
			In this paper we considered the problem of computing the dominant eigenvector of a symmetric matrix $\mathbf{A}$ subject to the constraint that matrix-vector products with $\mathbf{A}$ are only partially observed, motivated by the phenomenon 
			of straggling appearing in distributed computing infrastructures when the matrix-vector products with the matrix $\mathbf{A}$ are computed in parallel on remote processing elements. We presented and analyzed two algorithms, where the first algorithm exploits a probabilistic model to switch between two distinct models to fill in missing partial matrix-vector entries, and the second algorithm averages a cascade of partial matrix-vector products. Our numerical experiments demonstrate that the proposed algorithms can outperform inexact power iteration based on asynchronous modeling. 
			
			The algorithms proposed in this paper are limited to symmetric matrices and can compute only a single (the dominant) eigenvector.
			As part of our future work we plan to extend the algorithms and analysis presented in this paper towards the simultaneous computation of more than one eigenpairs via leveraging partial matrix-multivector models of subspace iteration. Additionally, we plan on implementing the algorithms presented in this paper on distributed computing cloud and cluster environments in order to benchmark the performance of partial power iteration against classical power iteration and its asynchronous variants. Finally, we plan to explore the deployment of the ideas presented in this paper on controller-worker models found in federated learning for the purpose of building large-scale machine learning and AI models~\cite{Wen2023,roth2022nvidia}. 

			\bibliographystyle{plain} 

			\appendix

			\section{Additional Experiments}
			\label{appendix_a} 
			
			\subsection{Sample mean and standard deviation of selected experiments}
			
			\begin{figure*}[ht]
				\centering
				\includegraphics[width=0.30\linewidth]{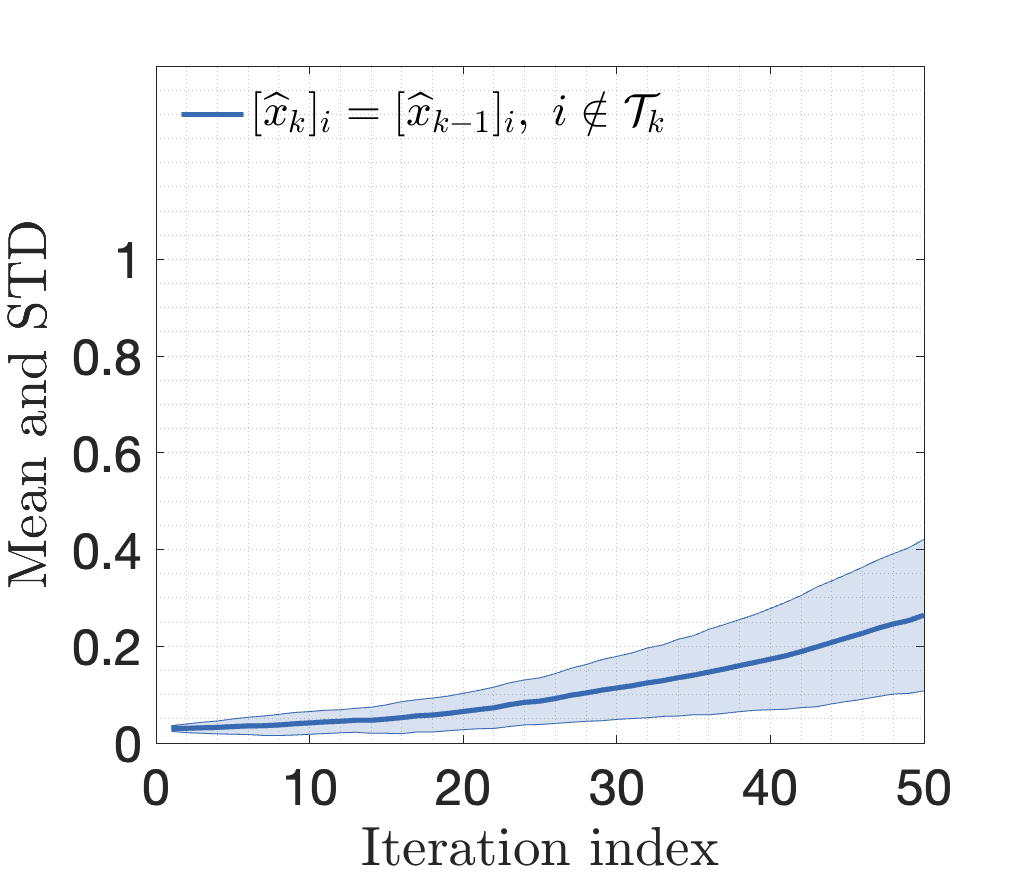}
				\includegraphics[width=0.30\linewidth]{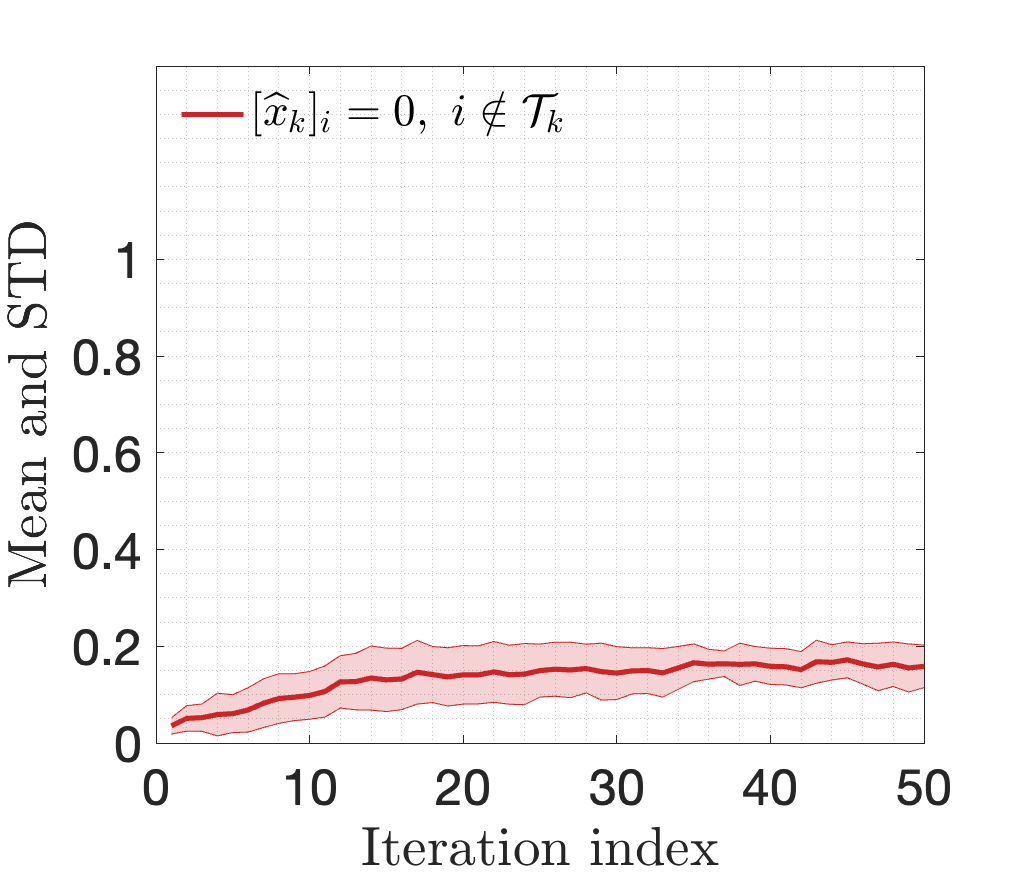}
				\includegraphics[width=0.30\linewidth]{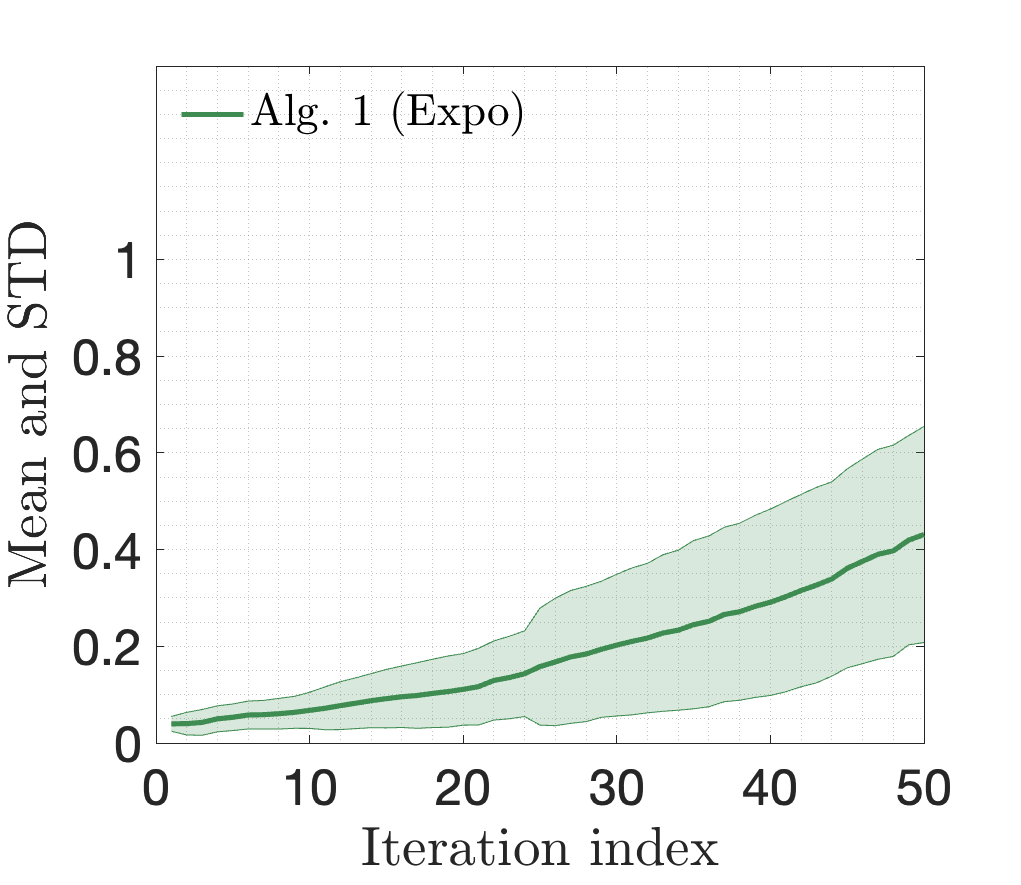}
				\includegraphics[width=0.30\linewidth]{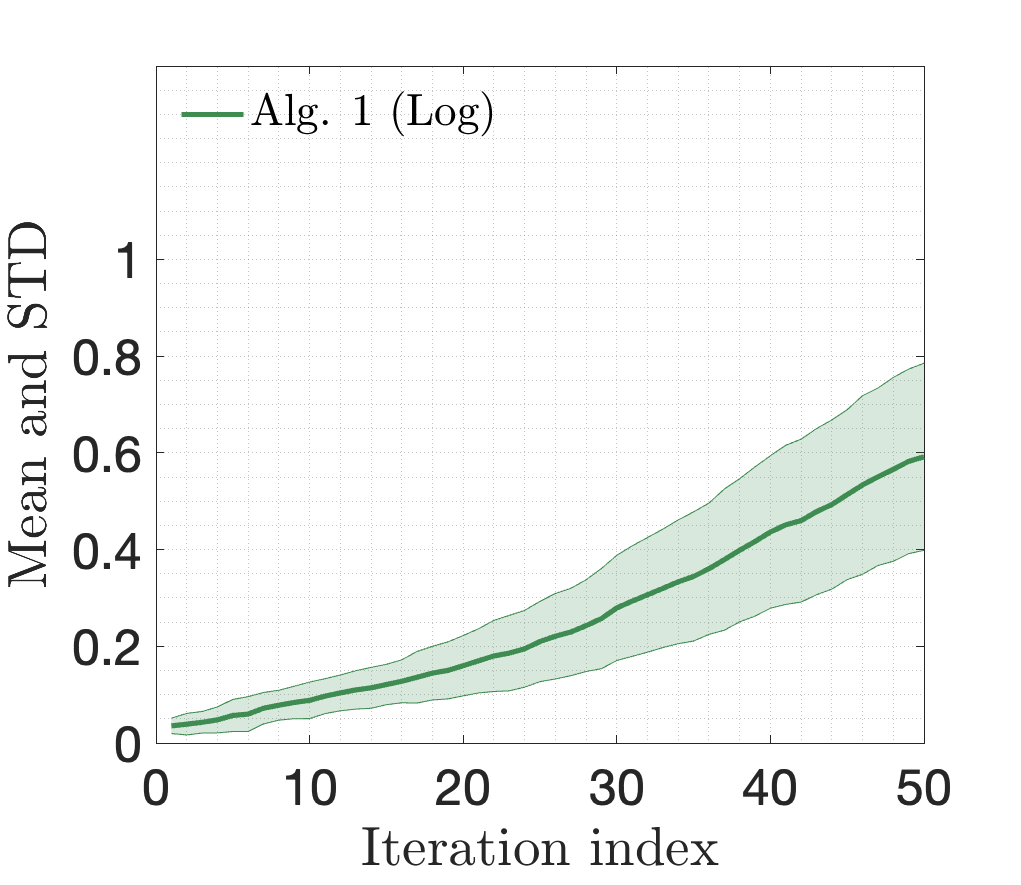}
				\includegraphics[width=0.30\linewidth]{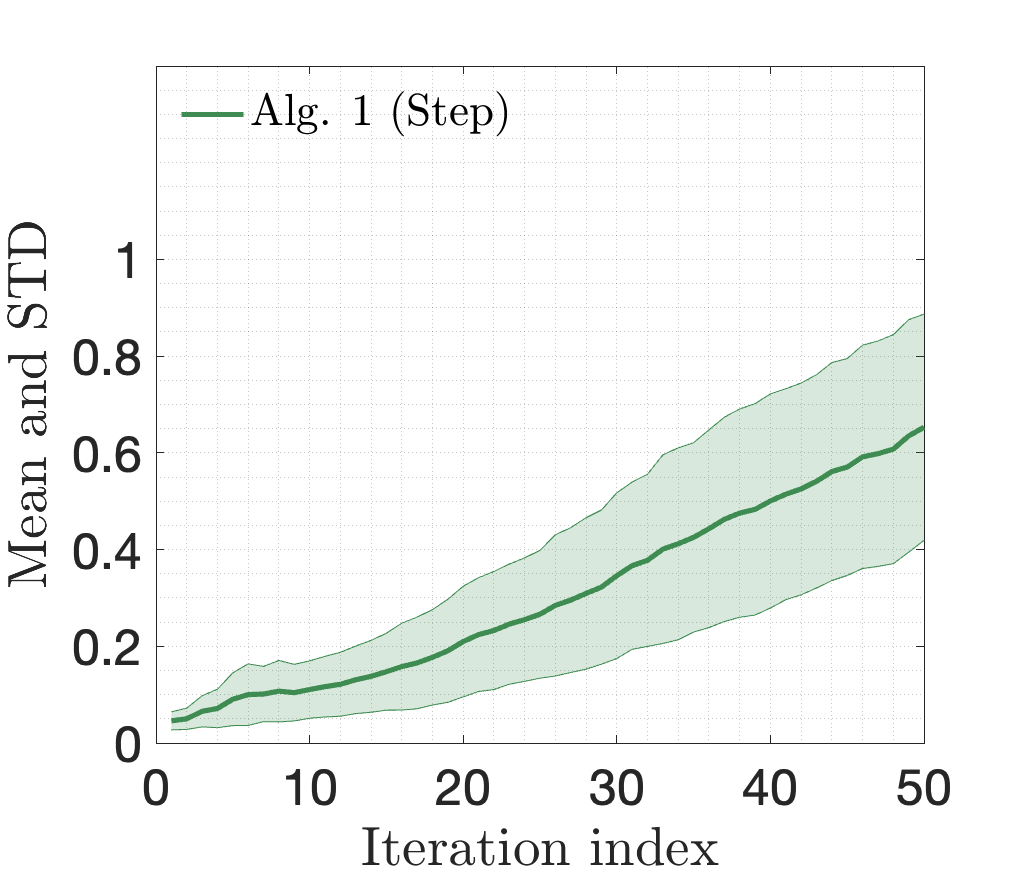}
				\includegraphics[width=0.30\linewidth]{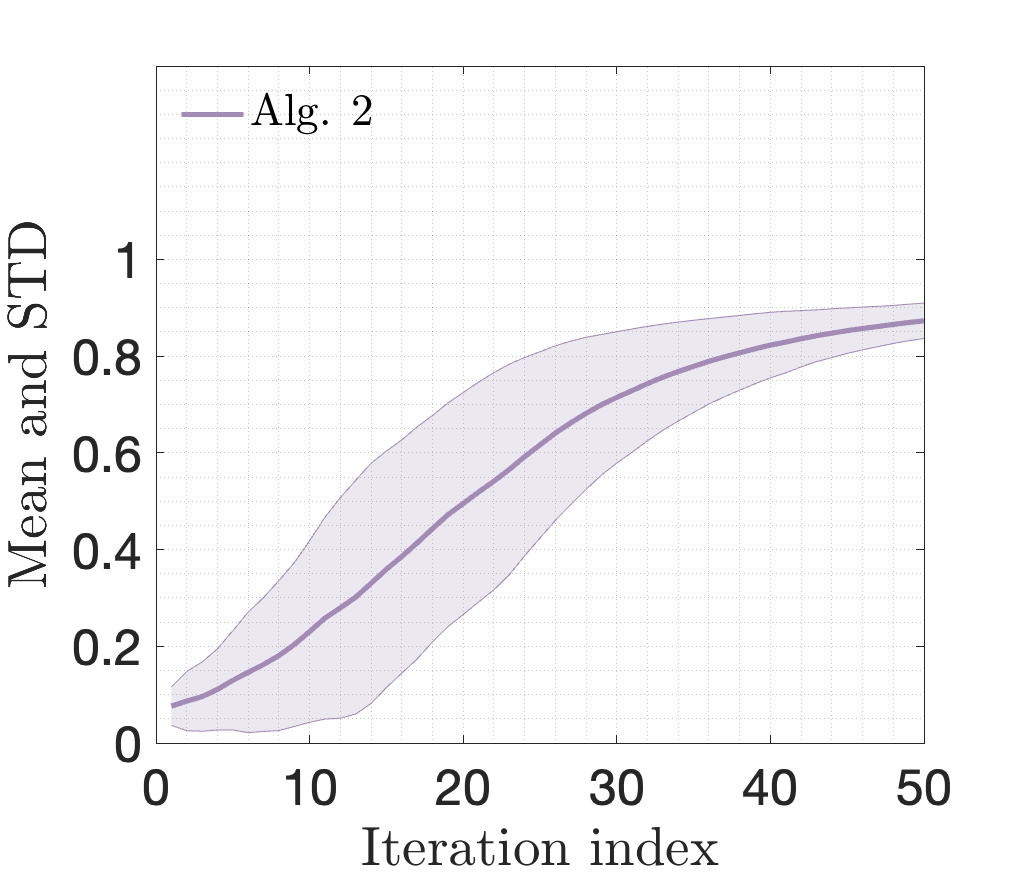}
				\caption{{\it Mean and standard deviation over twenty runs. The spectrum of $A$ is designed as 
						$\lambda_1=1,\ \lambda_j \in (0.0,0.5),\ j=2,\ldots,N$, and $\mathbb{E}[T]/N=0.1$.}}\label{figg1}
			\end{figure*}
			
			\begin{figure*}[ht]
				\centering
				\includegraphics[width=0.30\linewidth]{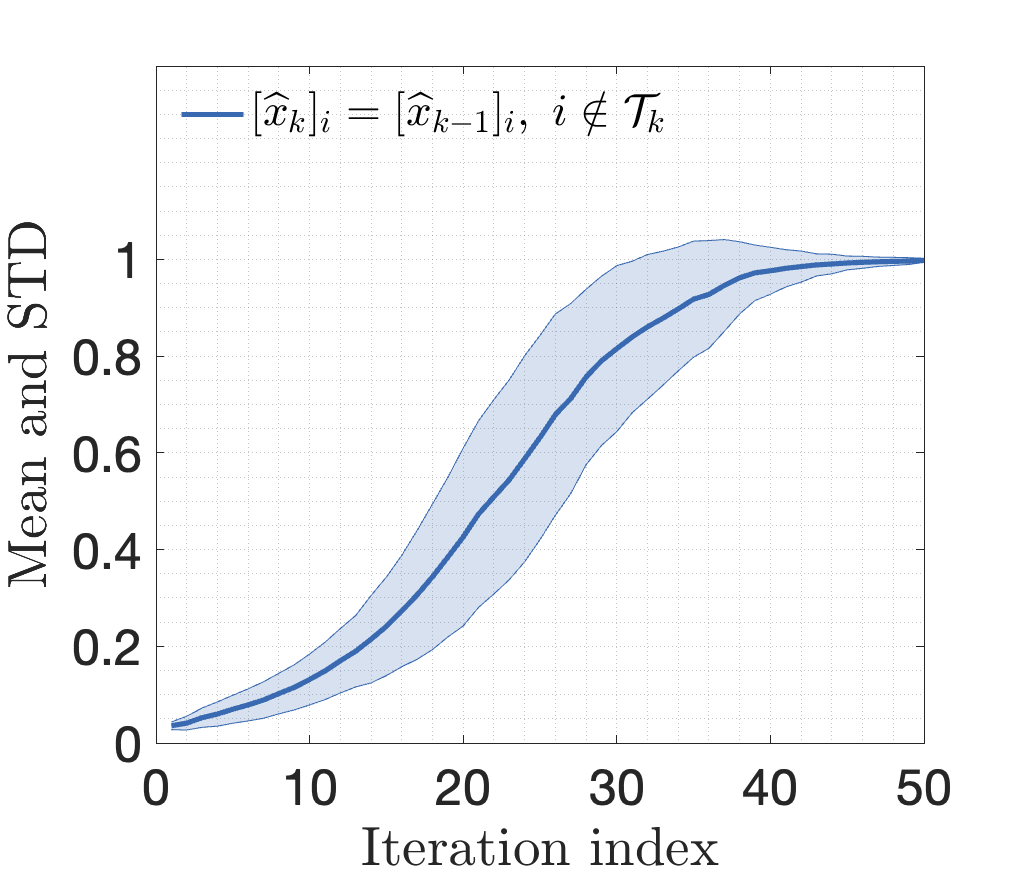}
				\includegraphics[width=0.30\linewidth]{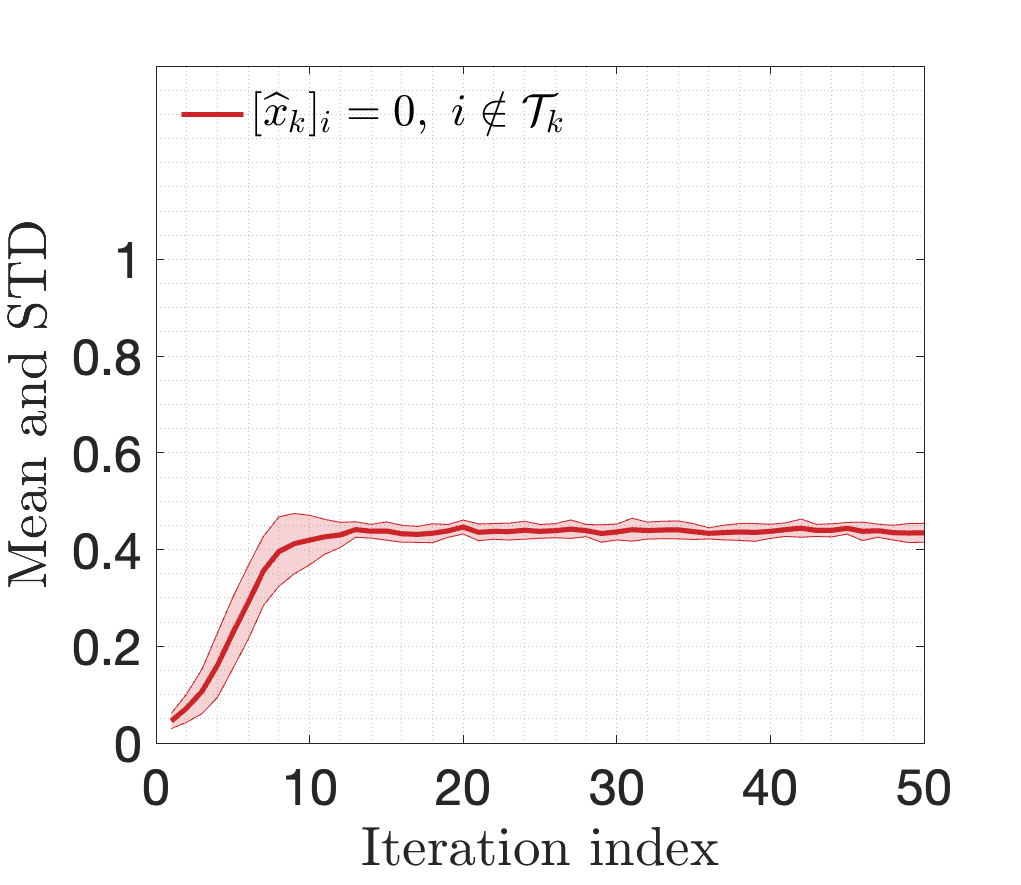}
				\includegraphics[width=0.30\linewidth]{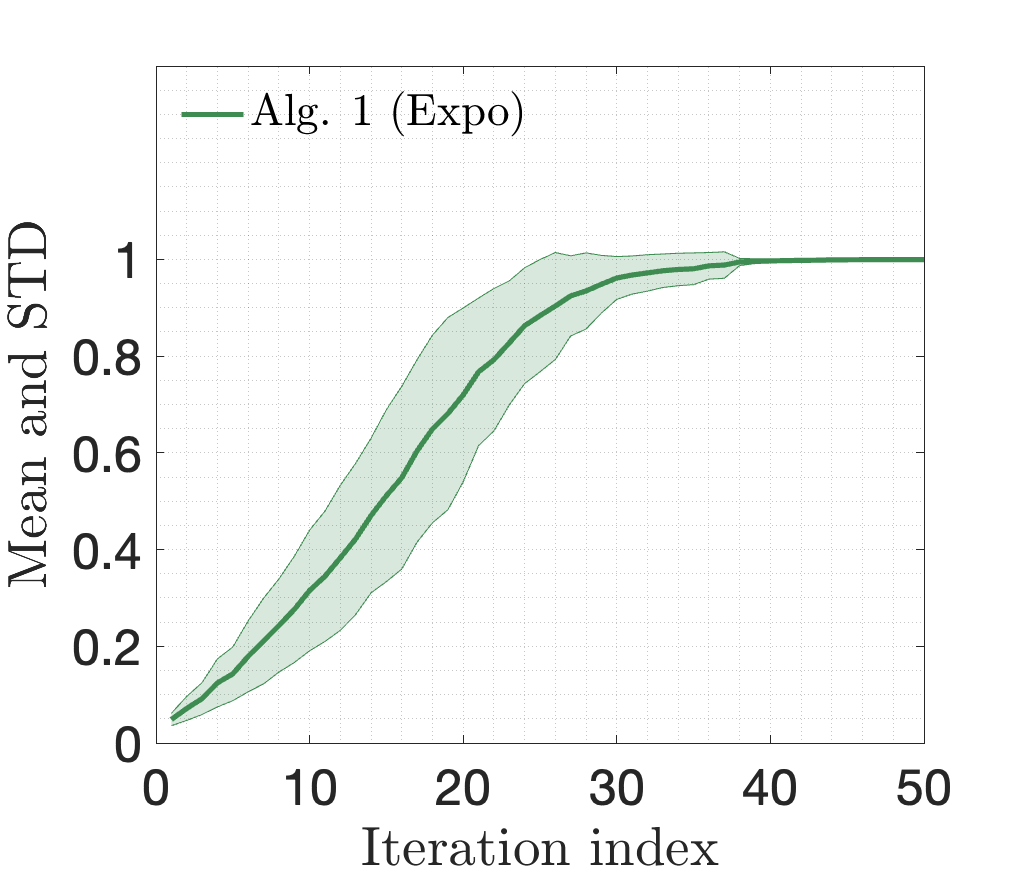}
				\includegraphics[width=0.30\linewidth]{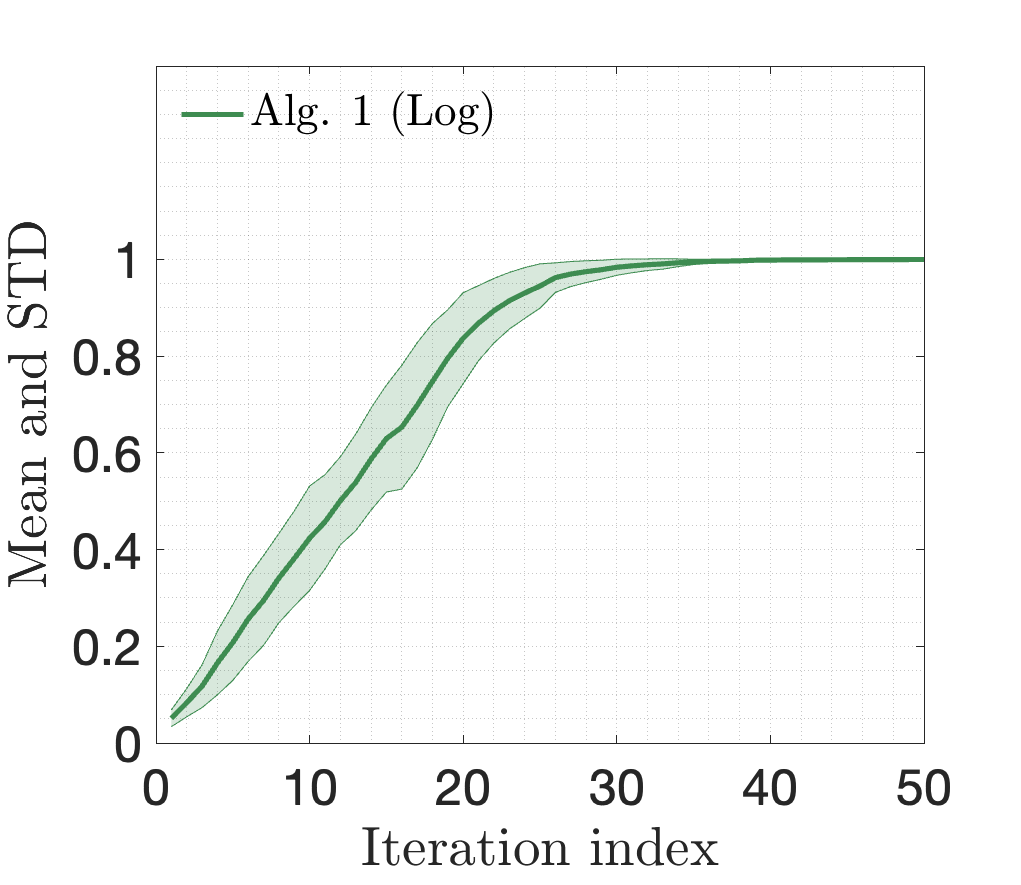}
				\includegraphics[width=0.30\linewidth]{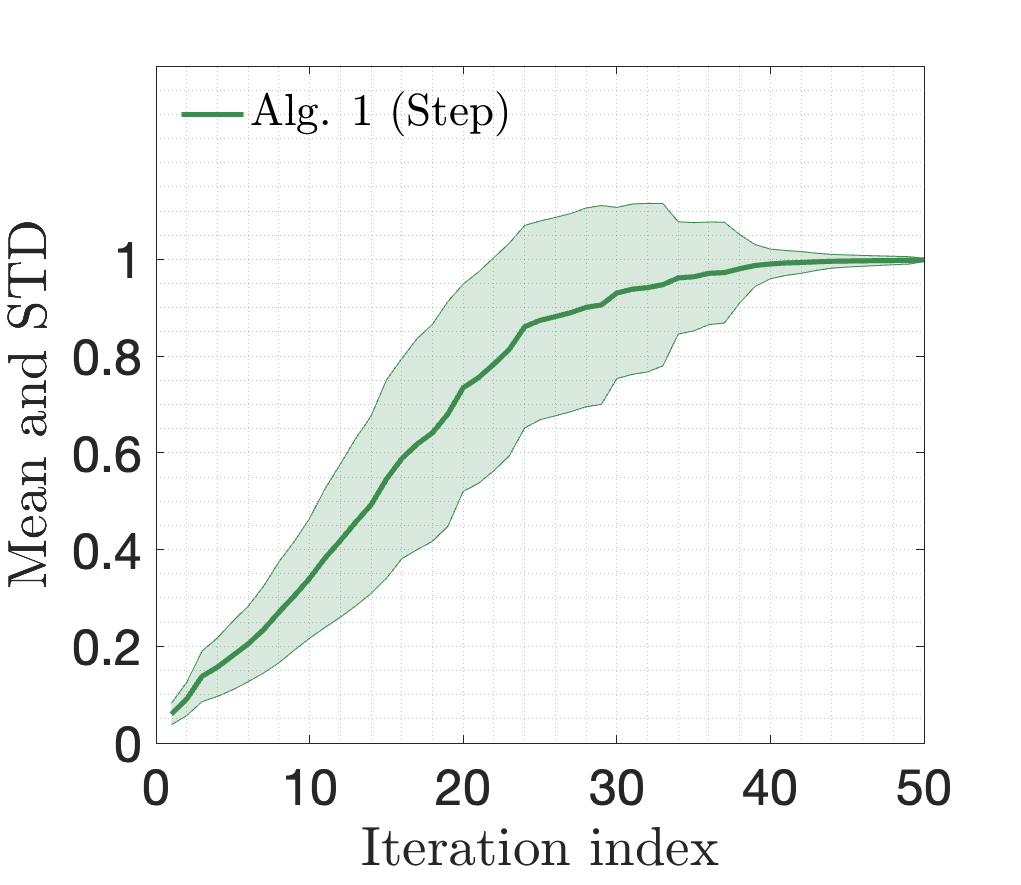}
				\includegraphics[width=0.30\linewidth]{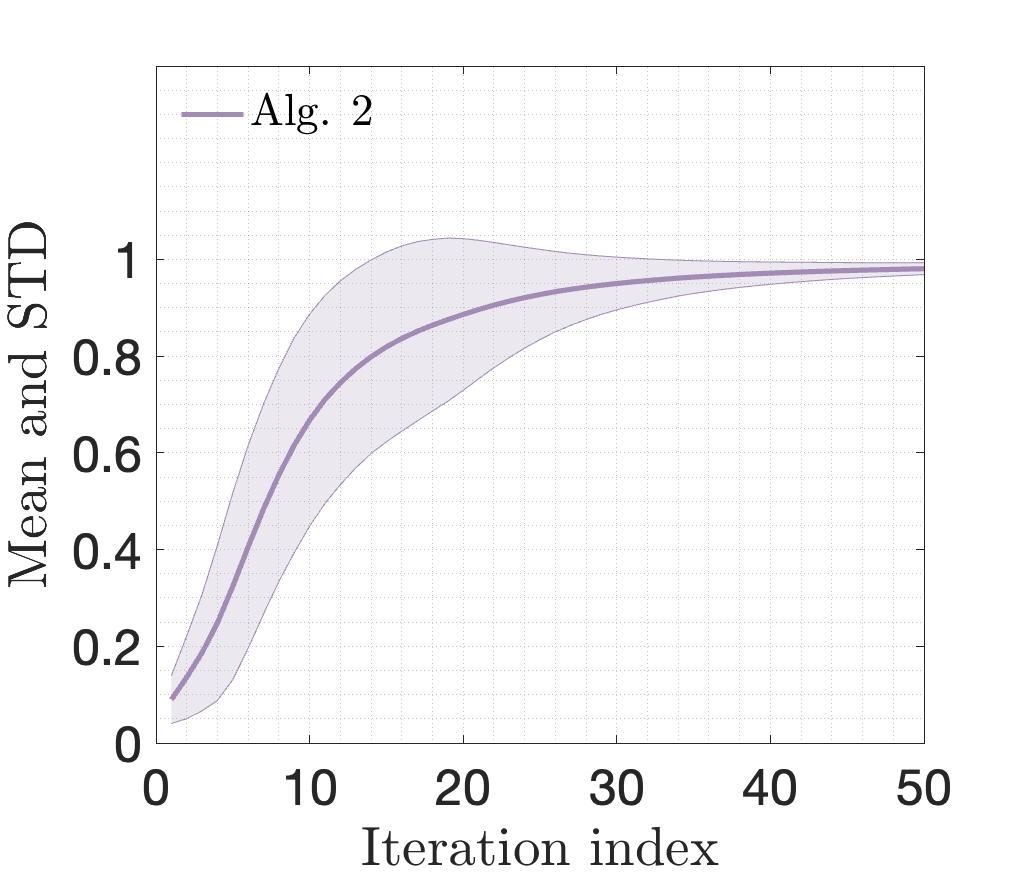}
				\caption{{\it Mean and standard deviation over twenty runs. The spectrum of $A$ is designed as 
						$\lambda_1=1,\ \lambda_j \in (0.0,0.5),\ j=2,\ldots,N$, and $\mathbb{E}[T]/N=0.3$.}}\label{figg2}
			\end{figure*}

			\begin{figure*}[ht]
				\centering
				\includegraphics[width=0.30\linewidth]{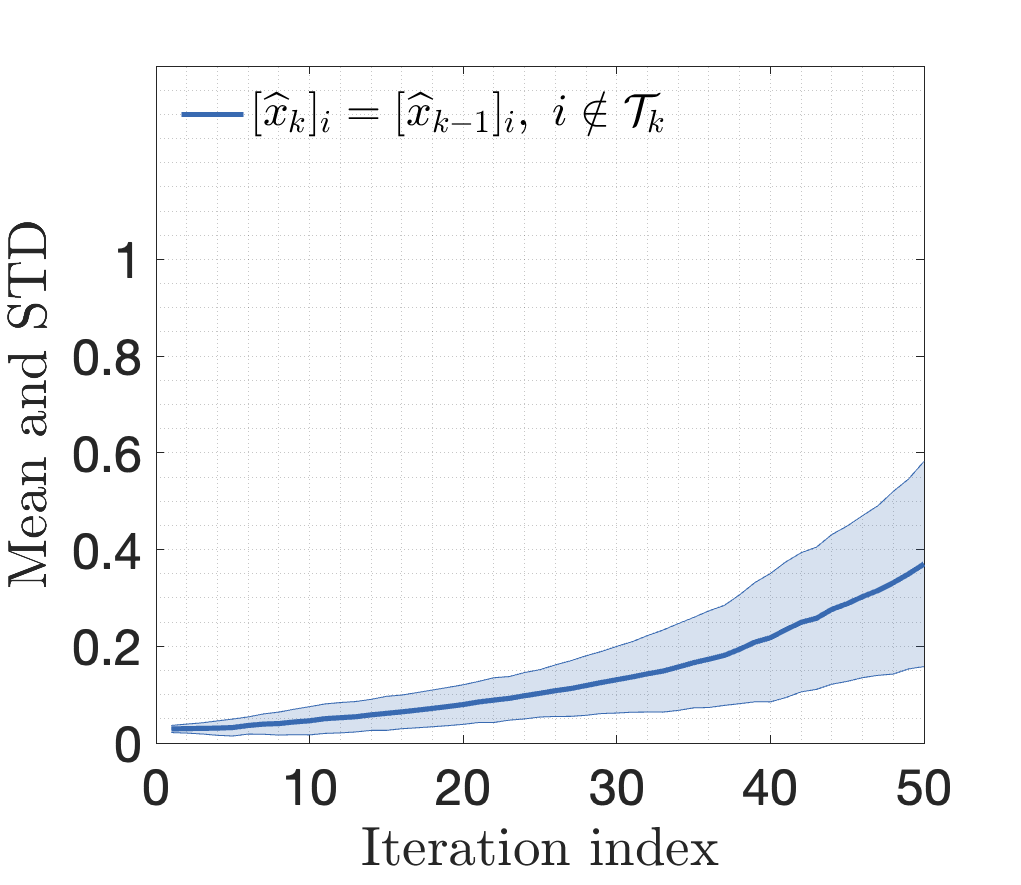}
				\includegraphics[width=0.30\linewidth]{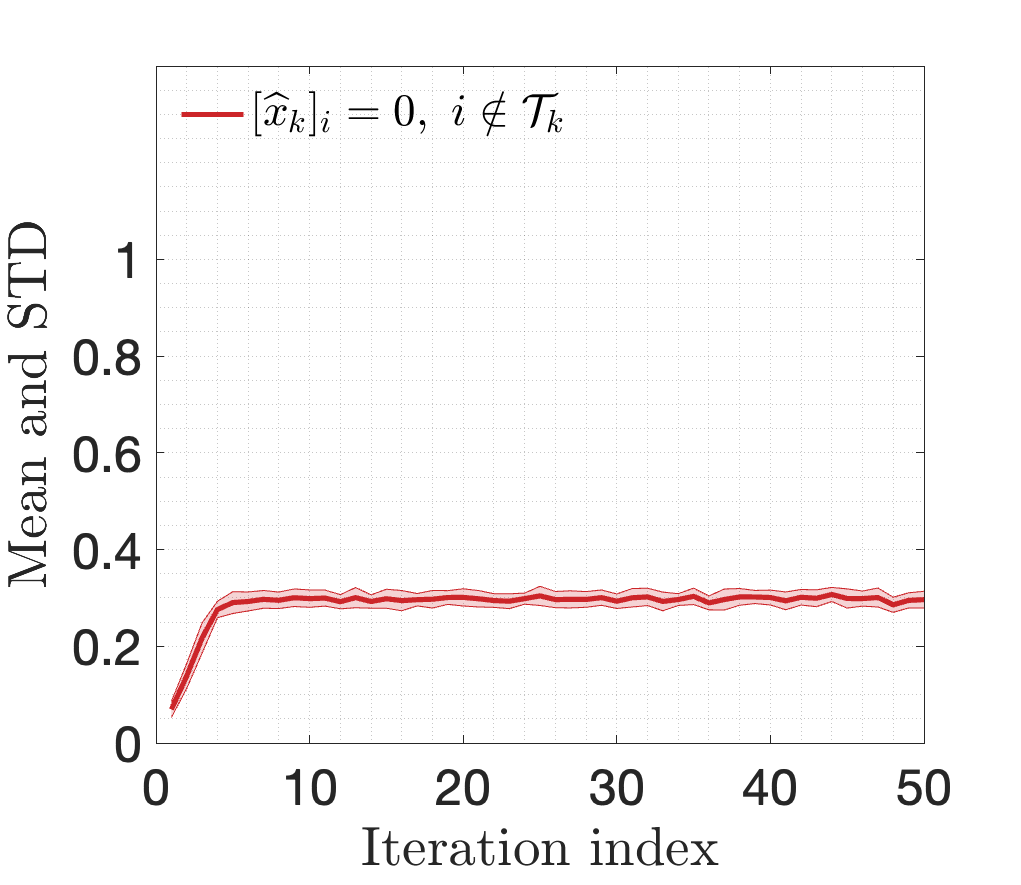}
				\includegraphics[width=0.30\linewidth]{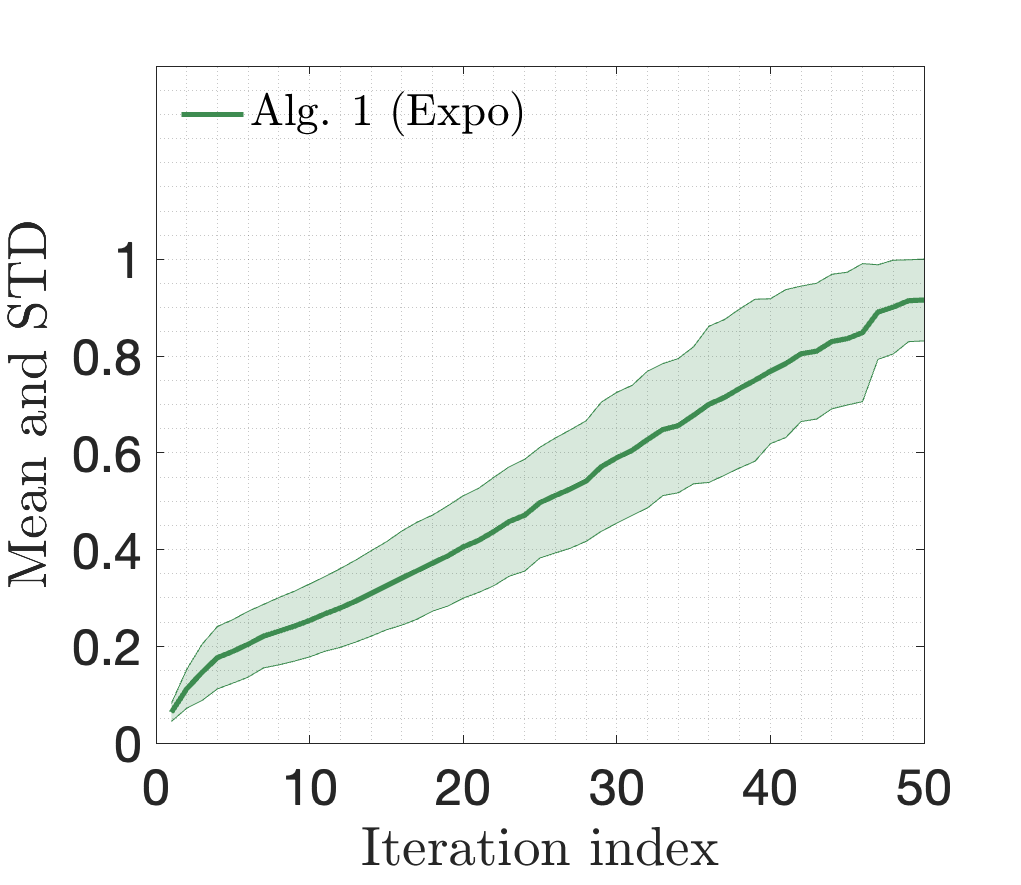}
				\includegraphics[width=0.30\linewidth]{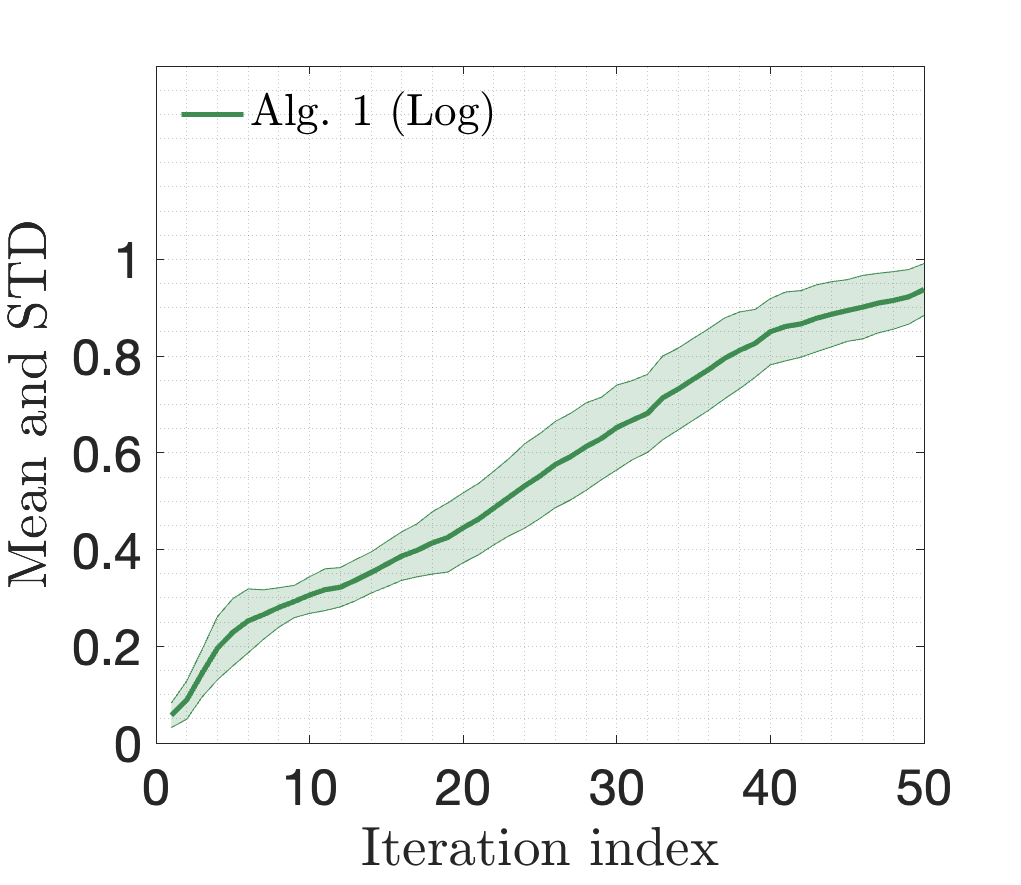}
				\includegraphics[width=0.30\linewidth]{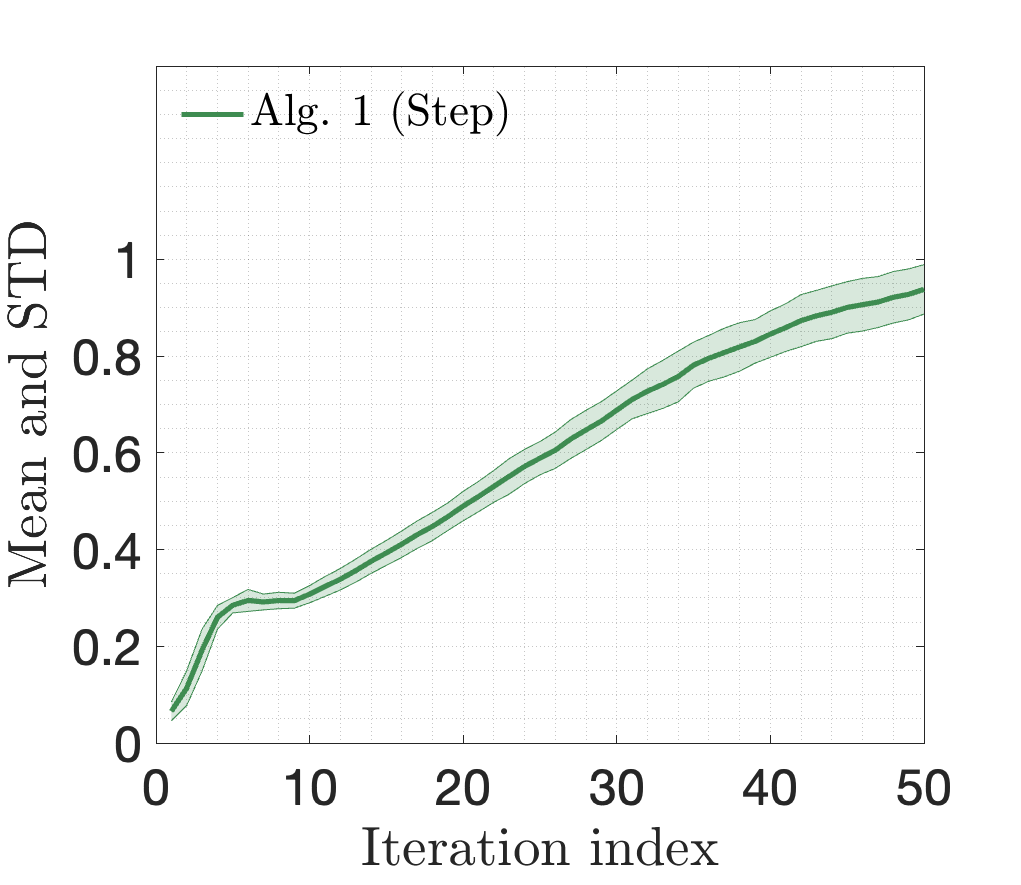}
				\includegraphics[width=0.30\linewidth]{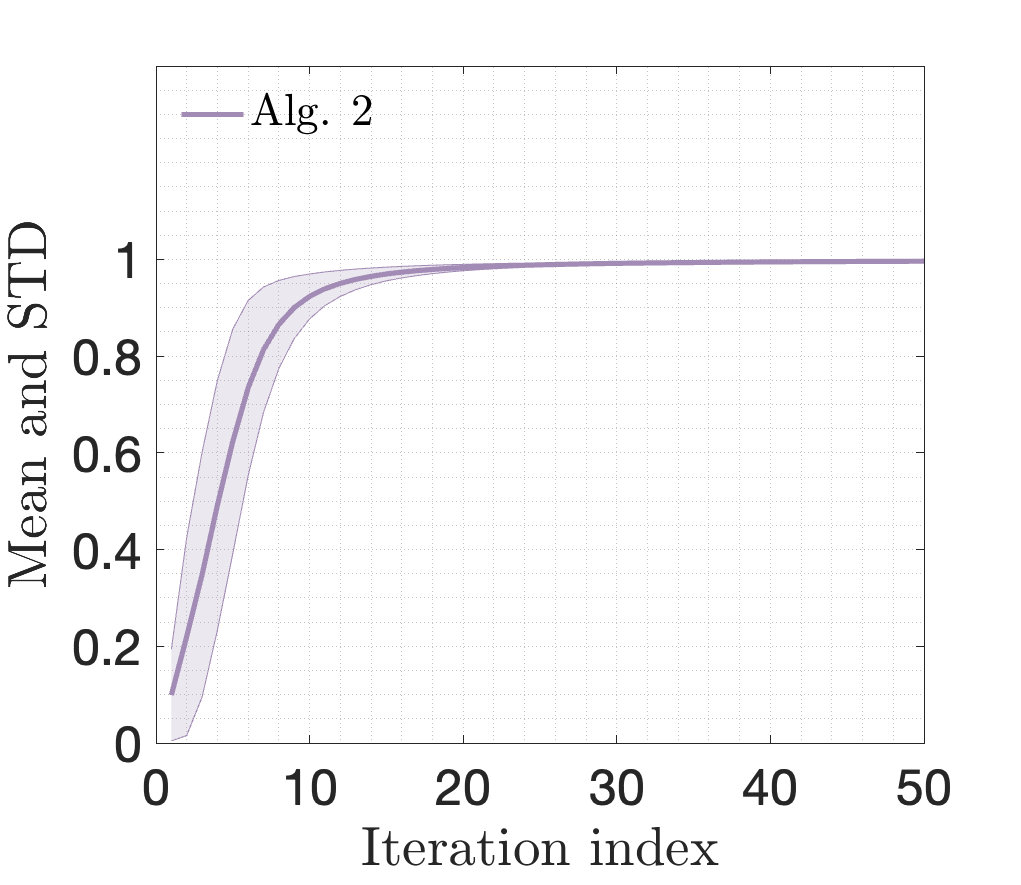}
				\caption{{\it Mean and standard deviation over twenty runs. The spectrum of $A$ is designed as 
						$\lambda_1=1,\ \lambda_j \in (-0.5,0.5),\ j=2,\ldots,N$, and $\mathbb{E}[T]/N=0.1$.}}\label{figg3}
			\end{figure*}
			
			\begin{figure*}[ht]
				\centering
				\includegraphics[width=0.30\linewidth]{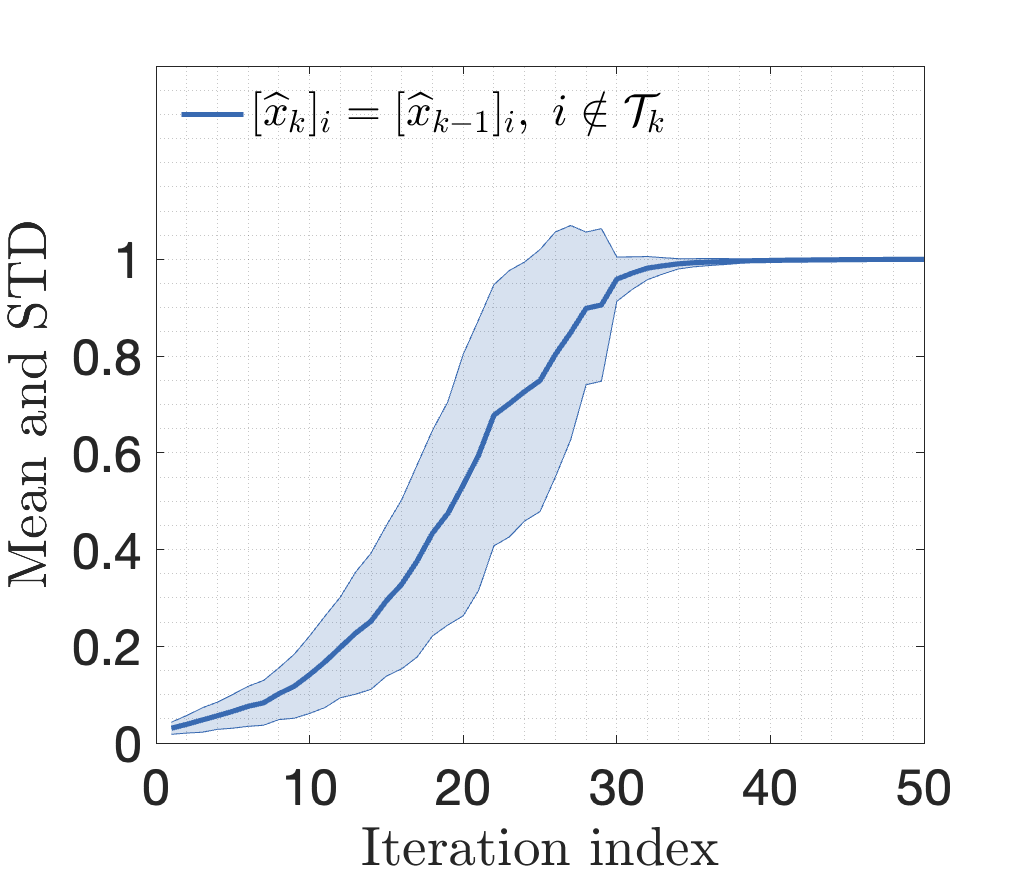}
				\includegraphics[width=0.30\linewidth]{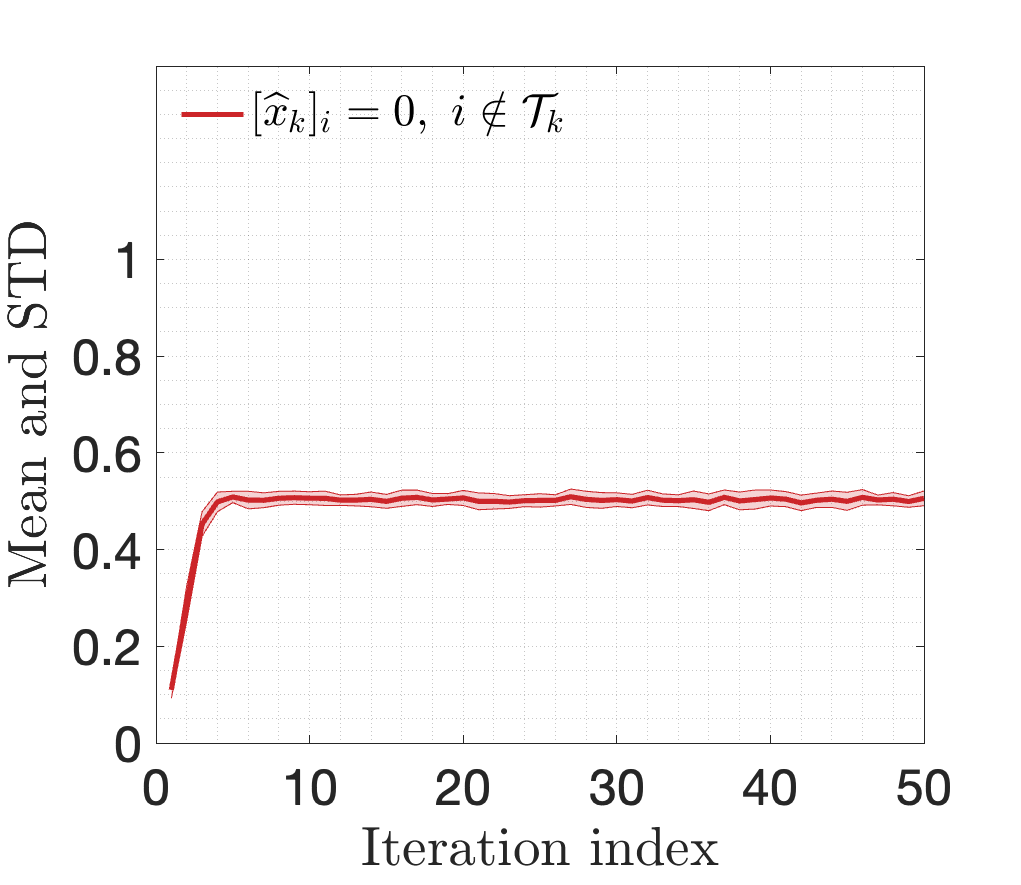}
				\includegraphics[width=0.30\linewidth]{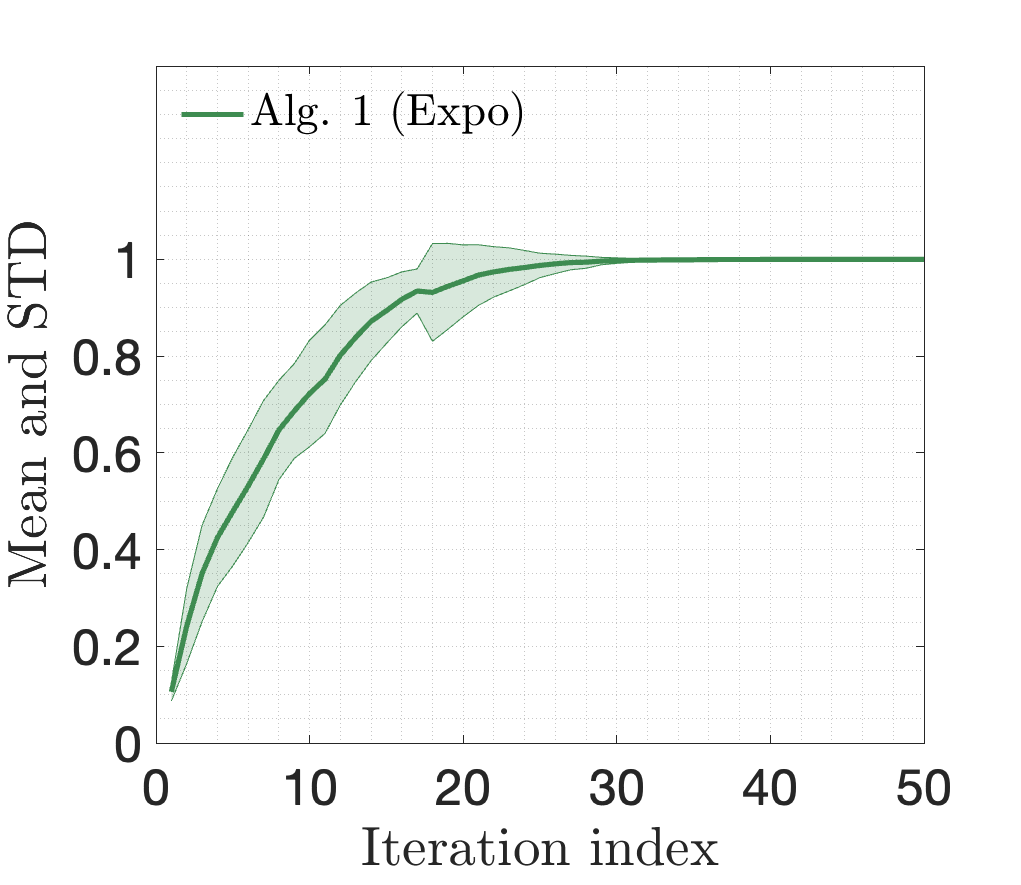}
				\includegraphics[width=0.30\linewidth]{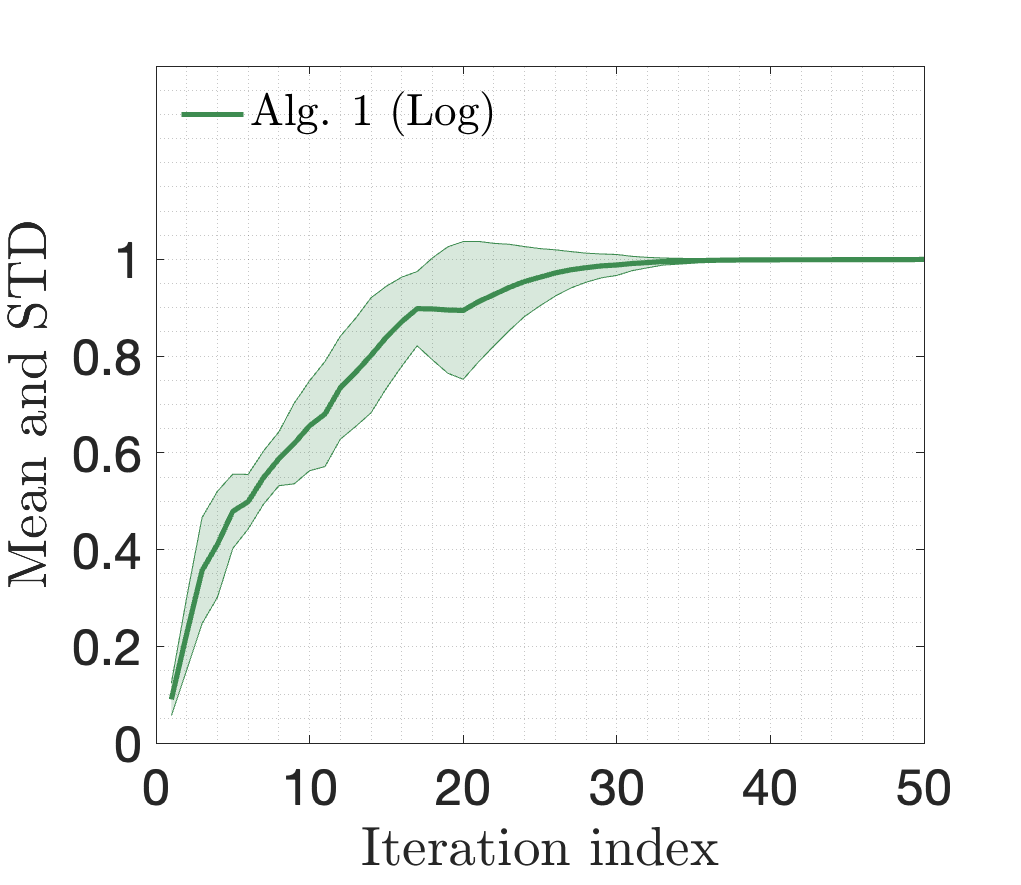}
				\includegraphics[width=0.30\linewidth]{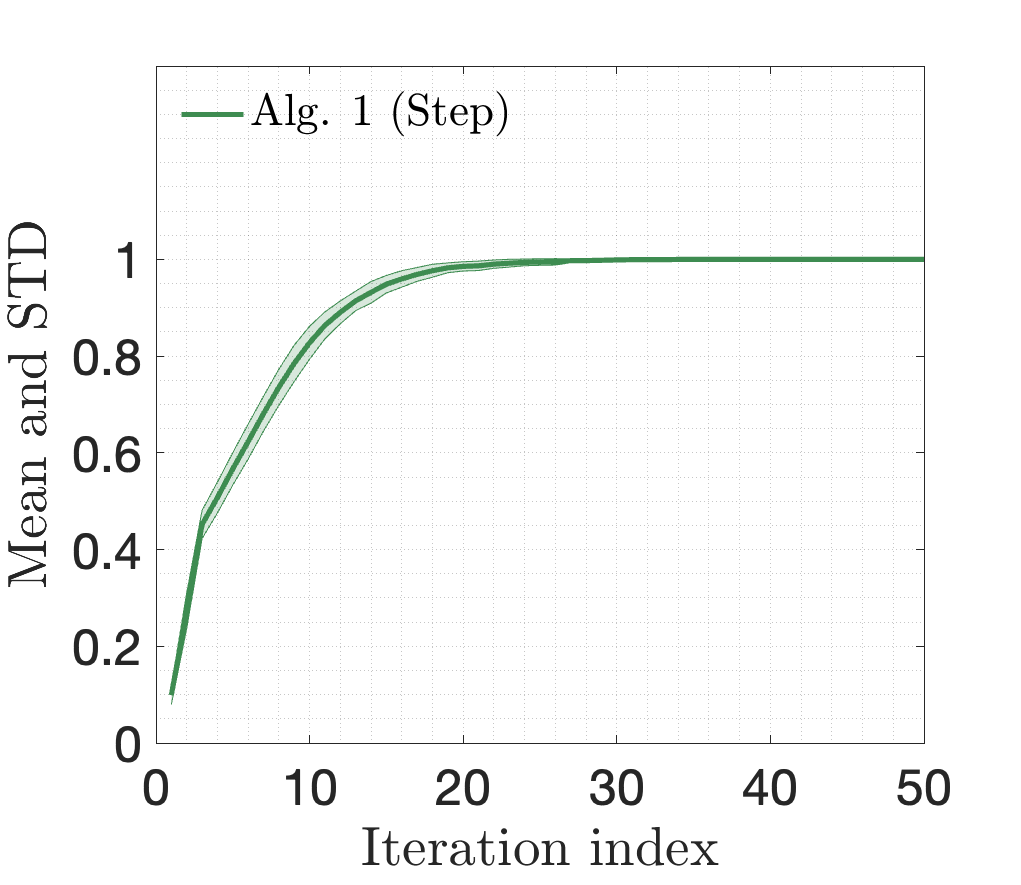}
				\includegraphics[width=0.30\linewidth]{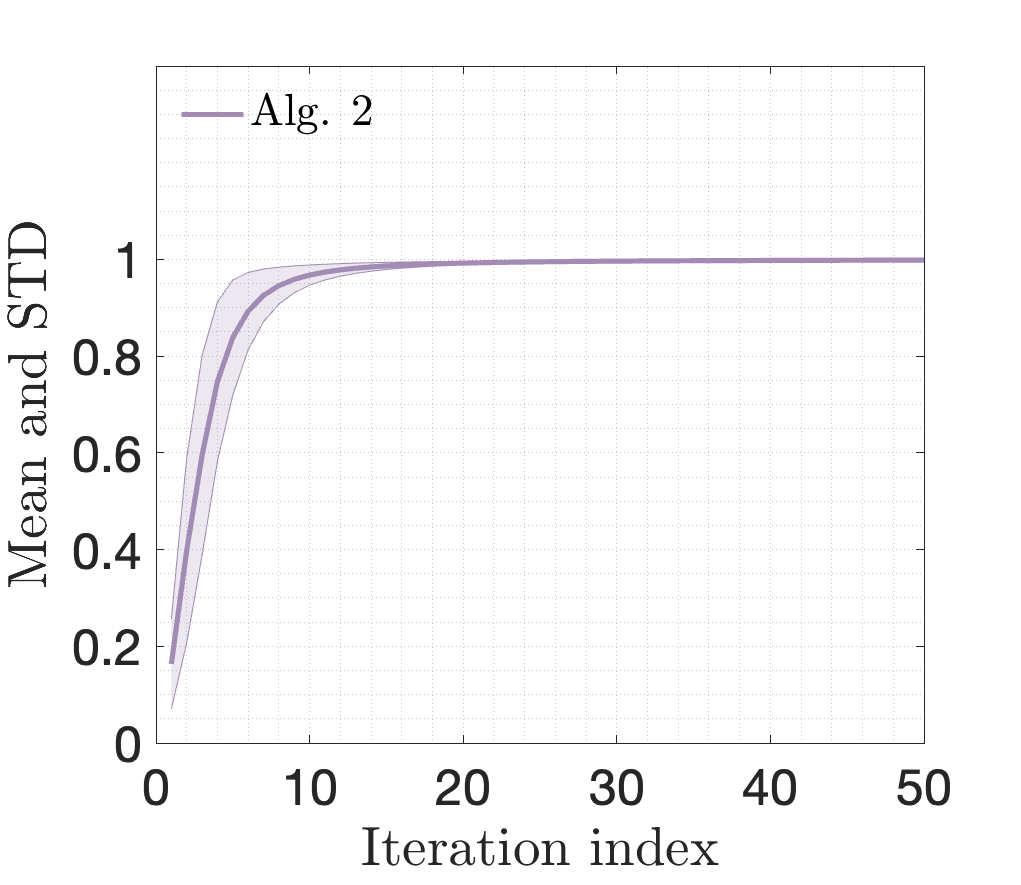}
				\caption{{\it Mean and standard deviation over twenty runs. The spectrum of $A$ is designed as 
						$\lambda_1=1,\ \lambda_j \in (-0.5,0.5),\ j=2,\ldots,N$, and $\mathbb{E}[T]/N=0.1$.}}\label{figg4}
			\end{figure*}
			
			Figures \ref{figg1}-\ref{figg2} plots the mean and standard deviation of twenty replications of the same experiment for the case where the spectrum of $\mathbf{A}$ is designed as $\lambda_1=1,\ \lambda_j \in (0.0,0.5),\ j=2,\ldots,N$. For the sake of compactness, we only consider the cases $\mathbb{E}[T]/N=0.1$ and $\mathbb{E}[T]/N=0.3$. 
			Similarly, Figures \ref{figg3}-\ref{figg4} plots the same quantities for the case where the spectrum of $\mathbf{A}$ is designed as $\lambda_1=1,\ \lambda_j \in (-0.5,0.5),\ j=2,\ldots,N$. For the sake of keeping this supplement short, we plot results only for a subset of the algorithms.
			
			\subsection{Monte Carlo-based partial power iteration}
			
			\begin{figure*}[ht]
				\centering
				\includegraphics[width=0.30\linewidth]{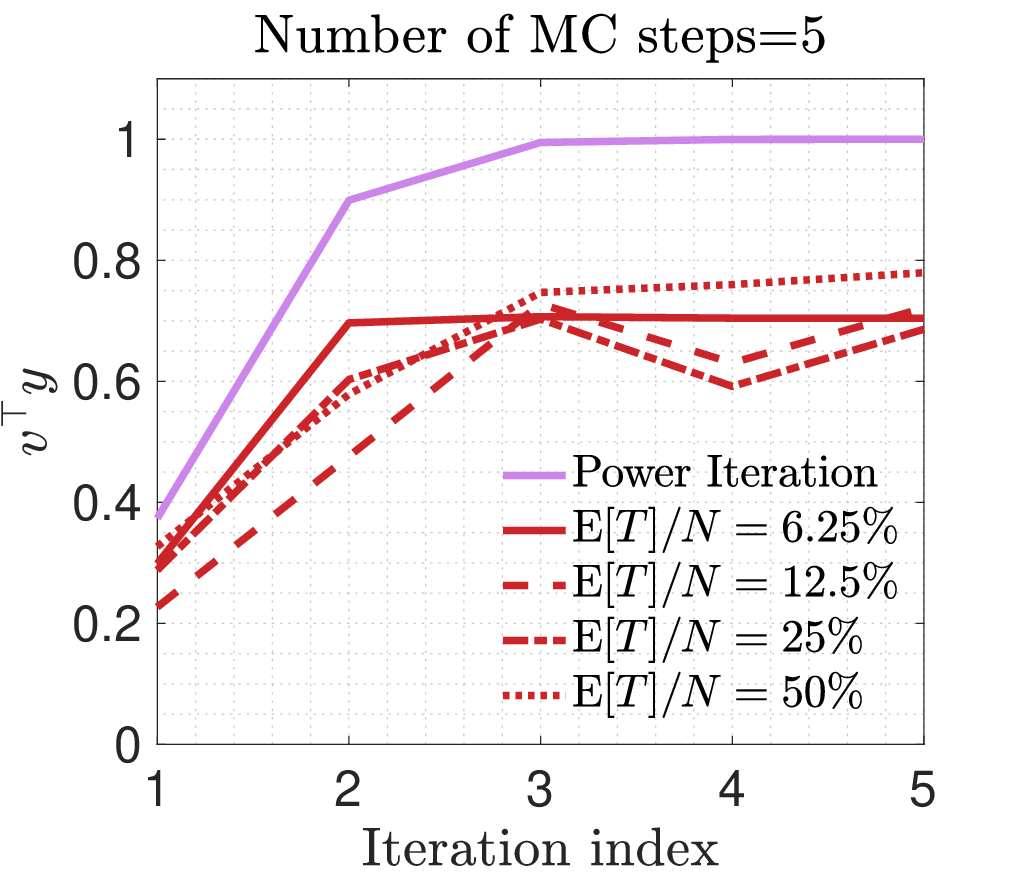}   
				\includegraphics[width=0.30\linewidth]{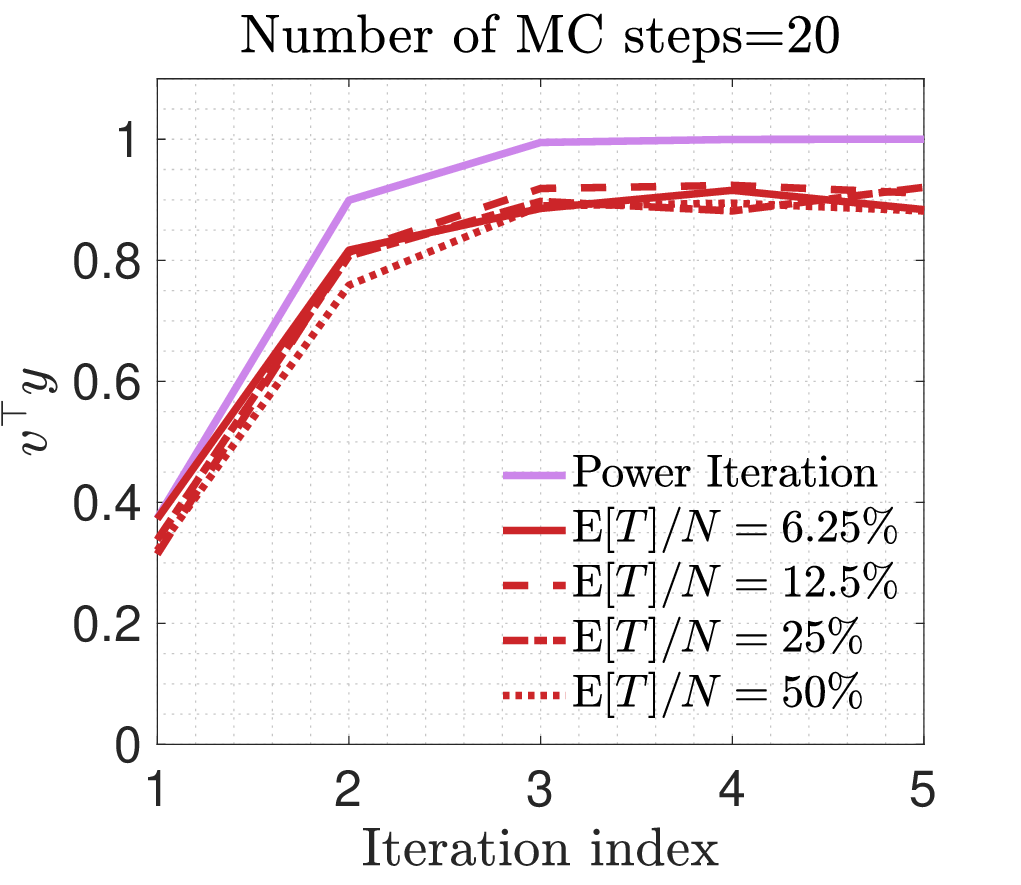}    
				\includegraphics[width=0.30\linewidth]{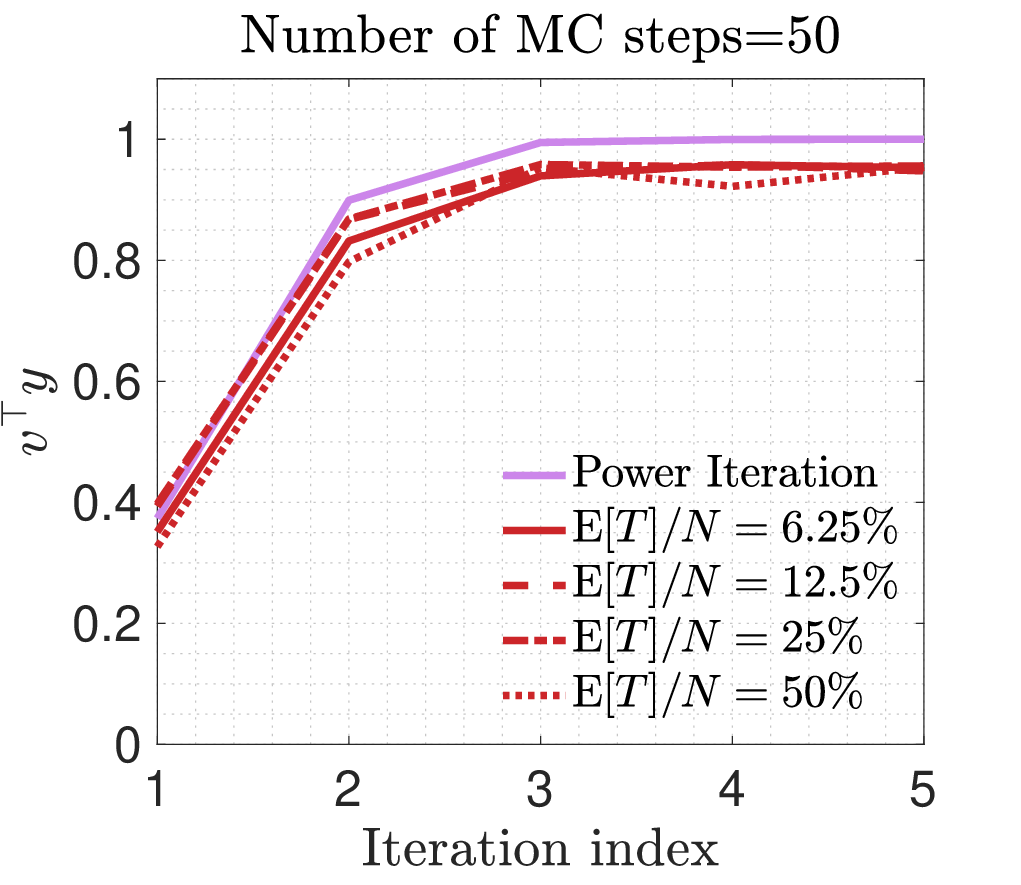} 
				\caption{{\it Monte-Carlo process for the inexact power iteration equipped 
						with the matrix-vector product dictated by (\ref{updzero1}).}}\label{fig13}
			\end{figure*}
			
			Earlier in this paper we mentioned that applying inexact power iteration 
			equipped with the matrix-vector product dictated by the matrix-vector model
			\begin{equation*}
				[\widehat{\bm{x}}_{k}]_i = 
				\begin{cases}
					\frac{N}{\mathbb{E}[T]}\sum_{j=1}^{j=N} \bm{A}_{ij}[\widehat{\bm{x}}_{k-1}]_j  &\ \mathrm{if}\ i\in {\mathcal T}_k\\
					0 &\ \mathrm{if}\ i\notin {\mathcal T}_k
				\end{cases}, 
			\end{equation*} 
			leads to a stagnation in the approximation accuracy of $v_1$. This phenomenon is explained by the fact that each $\widehat{\bm{x}}_k$ retains only $N-T_k$ non-zero entries. On the other hand, from Proposition \ref{pro1} we know that $\mathbb{E}[\widehat{\bm{x}}_k]$ is equal to $\mathbf{A}^k\widehat{\bm{x}}_0$. 
			Therefore, while the approximation $\widehat{\bm{x}}_k$ produced by inexact power iteration equipped 
			with the matrix-vector model in (\ref{updzero1}) can not converge to $\bm{v}_1$ regardless of the value of $k$, 
			the empirical mean over an infinite number of different $\widehat{\bm{x}}_k$ obtained with the same initial approximation $\widehat{\bm{x}}_0$ is guaranteed to converge to $\mathbf{A}^k\widehat{\bm{x}}_0$. 
			The expectation $\mathbb{E}[\widehat{\bm{x}}_k]$ 
			can be approximated via Monte-Carlo where we consider 
			the empirical mean $\frac{1}{L}\sum_{l=1}^L\widehat{\bm{x}}_k^{(l)},\ L\in \mathbb{N}$, with $\widehat{\bm{x}}_k^{(l)}$ denoting the approximation of inexact power iteration with 
			(\ref{updzero1}) after $k$ steps for the $l$-th trial. Figure \ref{fig13} plots the 
			inner product between $\frac{1}{L}\sum_{l=1}^L\widehat{\bm{x}}_k^{(l)}$ and $v_1$ as the 
			value of $L$ is varied as $L=5,\ L=20$, and $L=50$. For our illustration we use a smaller 
			matrix of size $N=128$ with eigenvalues set as $\lambda_1=1,\ \lambda_j \in (0.0,0.25),\ j=2,\ldots,N$. For the sake of completeness, we also vary 
			the ratio $\mathbb{E}[T]/N$. Following the above Monte-Carlo procedure, inexact power iteration equipped 
			with the matrix-vector model in (\ref{updzero1}) can now provide a much better approximation 
			of $\bm{v}_1$ if multiple independent runs with the same initial approximation are averaged. 
			As expected, increasing the value of $L$ allows for a better approximation of $\bm{v}_1$.

			\section{Covariance Calculations}
			\label{appendix_b} 
			
			\subsection{General Calculations}
			Here we provide more details for Section~\ref{sec:varr}. We can write
			\begin{align}
				\mathbf{S}& =\mathbf{Q}_{\mathcal{T}}\otimes \mathbf{Q}_{\mathcal{T}} \nonumber
				\\ &= (\mathbf{D}_{\mathcal{T}}\mathbf{A}+\mathbf{D}_{\mathcal{T}^c}\mathbf{I})\otimes (\mathbf{D}_{\mathcal{T}}\mathbf{A}+\mathbf{D}_{\mathcal{T}^c}\mathbf{I}) \nonumber
				\\ &= \underbrace{\mathbf{D}_{\mathcal{T}}\mathbf{A}\otimes \mathbf{D}_{\mathcal{T}}\mathbf{A}}_{I}+\underbrace{\mathbf{D}_{\mathcal{T}^c}\mathbf{I}\otimes \mathbf{D}_{\mathcal{T}}\mathbf{A}}_{II}+  \underbrace{\mathbf{D}_{\mathcal{T}}\mathbf{A}\otimes \mathbf{D}_{\mathcal{T}^c}\mathbf{I}}_{III}+ \underbrace{\mathbf{D}_{\mathcal{T}^c}\mathbf{I} \otimes \mathbf{D}_{\mathcal{T}^c}\mathbf{I}}_{IV}\,
				\tag{$I-IV$}\label{eq:I-IV}
			\end{align}
			with $\mathcal{T}^c$ denoting the complement of $\mathcal{T}$.
			We recall that $\mathbf{A}_i(=\mathbf{A}_{i:})$ represents the $i$-th row of $\mathbf{A}$, so that 
			$\mathbf{A}^\top=(\mathbf{A}_1^\top, \mathbf{A}^\top_2, \cdots,\mathbf{A}_N^\top)
			$. 
			Furthermore,
			\begin{align*}
				\mathbf{A}\otimes \mathbf{A} = \left(\begin{array}{c}
					\mathbf{B}_1 \\ \mathbf{B}_2\\ \vdots\\ \mathbf{B}_N\end{array}\right) 
				\quad\text{where}\quad \mathbf{B}_i = \left(\begin{array}{c}\mathbf{A}_i\otimes \mathbf{A}_1\\\mathbf{A}_i\otimes \mathbf{A}_2\\ \vdots \\ \mathbf{A}_i\otimes \mathbf{A}_N\end{array}\right).
			\end{align*}
			Through straightforward calculations, we have,  
			\begin{lem}\label{lem:ex AA}
				For any $i\neq j$, 
				\begin{align*}
					\ex[{\bf 1}_{i\in \mathcal{T}}\mathbf{A}_i\otimes {\bf 1}_{j\in \mathcal{T}}\mathbf{A}_j]&=  \pr[i,j\in \mathcal{T}]\, \mathbf{A}_i\otimes \mathbf{A}_j= \frac{\ex[T]}{N}\frac{\ex[T]-1}{N-1} \mathbf{A}_i\otimes \mathbf{A}_j,
					\\
					\ex[{\bf 1}_{i\notin \mathcal{T}}\, \bm{e}_i\otimes {\bf 1}_{j\in \mathcal{T}}\mathbf{A}_j]&=\pr[j\in \mathcal{T}, i\notin \mathcal{T}]\,\bm{e}_i\otimes \mathbf{A}_j
					= \frac{\ex[T]}{N}\left(1-\frac{\ex[T]-1}{N-1}\right) \bm{e}_i\otimes \mathbf{A}_j, 
					\\
					\ex[{\bf 1}_{i\in \mathcal{T}}\mathbf{A}_i\otimes {\bf 1}_{i\in \mathcal{T}}\mathbf{A}_i]&=  \pr[i\in \mathcal{T}] \,\mathbf{A}_i\otimes \mathbf{A}_i= \frac{\ex[T]}{N} \mathbf{A}_i\otimes \mathbf{A}_i.
				\end{align*}
				\begin{proof}[ ]
				\end{proof}
			\end{lem}

			We recall the matrices  $I'$ and $\mathbf{\Sigma}$ introduced before Section~\ref{subsec: CovUpdate1} in \eqref{eqn:Sigma Ip} and introduce three more $N^2\times N^2$ matrices:
			\begin{align*}
				I'=\left(\begin{array}{c} \mathbf{A}_1\otimes \mathbf{A}_1\\0\\ \vdots \\ 0\\ 0\\\mathbf{A}_2\otimes \mathbf{A}_2\\ \vdots \\ 0 \\ \vdots \\ 0\\0\\ \vdots \\ \mathbf{A}_N\otimes \mathbf{A}_N\end{array}\right),
				\quad
				II'\left(\begin{array}{c} 0\\\bm{e}_1\otimes \mathbf{A}_2\\ \vdots \\ \bm{e}_1\otimes \mathbf{A}_N\\[1mm] \bm{e}_2\otimes \mathbf{A}_1\\0\\ \vdots \\ \bm{e}_2\otimes \mathbf{A}_N \\ \vdots \\ \bm{e}_N\otimes \mathbf{A}_1\\\bm{e}_N\otimes \mathbf{A}_2\\ \vdots \\ 0\end{array}\right),
				\quad
				III'=\left(\begin{array}{c} 0\\\mathbf{A}_1\otimes \bm{e}_2\\ \vdots \\ \mathbf{A}_1\otimes \bm{e}_N\\[1mm] \mathbf{A}_2\otimes \bm{e}_1\\0\\ \vdots \\ \mathbf{A}_2\otimes \bm{e}_N \\ \vdots \\ \mathbf{A}_N\otimes \bm{e}_1\\\mathbf{A}_N\otimes \bm{e}_2\\ \vdots \\ 0\end{array}\right),\quad
				\mathbf{E}=\left(\begin{array}{c} \mathbf{e}_1\otimes \mathbf{e}_1\\0\\ \vdots \\ 0\\ 0\\\mathbf{e}_2\otimes \mathbf{e}_2\\ \vdots \\ 0 \\ \vdots \\ 0\\0\\ \vdots \\ \mathbf{e}_N\otimes \mathbf{e}_N\end{array}\right)\,.
			\end{align*}
			In the matrices $I'$ and $\mathbf{E}$ the only non-zero rows are those indexed by $k+(k-1)N$, $k=1,\dots, N$, where the entries are $\mathbf{A}_k\otimes \mathbf{A}_k$ and respectively $\mathbf{e}_k\otimes \mathbf{e}_k$, while in the matrices $II'$ and $III'$ these rows are the only zero ones. 
			Note that $\mathbf{\Sigma}=I'+II'+III'$. 
			
			We can write the expectations of the consecutive terms of  the sum~\eqref{eq:I-IV} in the following way.
			The expectation of the first term $(I)$ can be written as
			\begin{align*}
				\ex[\underbrace{\mathbf{D}_{\mathcal{T}}\mathbf{A}\otimes \mathbf{D}_{\mathcal{T}}\mathbf{A}}_{I}]= \frac{\ex[T]}{N}\frac{\ex[T]-1}{N-1} \mathbf{A}\otimes \mathbf{A} + \frac{\ex[T]}{N}\left(1-\frac{\ex[T]-1}{N-1}\right)
				I'\,.
			\end{align*}
			The expectation of the second term $(II)$ can be written as
			\begin{align*}
				\ex[\underbrace{\mathbf{D}_{\mathcal{T}^c}\mathbf{I}\otimes \mathbf{D}_{\mathcal{T}}\mathbf{A}}_{II}]= \frac{\ex[T]}{N}\left(1-\frac{\ex[T]-1}{N-1}\right)II',
			\end{align*}
			since $\ex[
			\mathbf{D}_{\mathcal{T}^c} \bm{e}_i\otimes \mathbf{D}_{\mathcal{T}}\mathbf{A}_i]=0$. The expectation of the  third term $(III)$ is equal to 
			\begin{align*}
				\ex[\underbrace{\mathbf{D}_{\mathcal{T}}\mathbf{A}\otimes \mathbf{D}_{\mathcal{T}^c}\mathbf{I}}_{III}] = \frac{\ex[T]}{N}\left(1-\frac{\ex[T]-1}{N-1}\right)III',
			\end{align*}
			since $\ex[\mathbf{D}_{\mathcal{T}^c} \mathbf{A}_i\otimes \mathbf{D}_{\mathcal{T}}\bm{e}_i]=0$. Note that $(III)$ is a permutational equivalent to $(II)$ and that the coefficients in front of $I'$, $II'$ and $III'$ are equal. 
			
			\noindent
			For the last term $(IV)$, similarly as with $(I)$,  with $\mathcal{T}^c$ in place of $\mathcal{T}$ and $\mathbf{I}$ in place of $\mathbf{A}$, we have 
			\begin{align*}
				\ex[\underbrace{\mathbf{D}_{\mathcal{T}^c}\mathbf{I} \otimes \mathbf{D}_{\mathcal{T}^c}\mathbf{I}}_{IV}]&=\frac{N-\ex[T]}{N}\frac{N-\ex[T]-1}{N-1}\mathbf{I}\otimes\mathbf{I}+\frac{N-\ex[T]}{N}\left(1-\frac{N-\ex[T]-1}{N-1}\right)\mathbf{E}   \\
				&=\left(1-\frac{\ex[T]}{N}\right)\left(\left(1-\frac{\ex[T]}{N-1}\right)\mathbf{I}\otimes\mathbf{I}
				+\frac{\ex[T]}{N-1}\mathbf{E}\right)\,.
			\end{align*} 
			Therefore,
			\begin{align*}
				\ex[\mathbf{S}]= &\frac{\ex[T]}{N}\frac{\ex[T]-1}{N-1} \mathbf{A}\otimes \mathbf{A} + \frac{\ex[T]}{N}\left(1-\frac{\ex[T]-1}{N-1}\right)\mathbf{\Sigma}
				+\left(1-\frac{\ex[T]}{N}\right) \left(\left(1-\frac{\ex[T]}{N-1}\right)\mathbf{I}\otimes\mathbf{I}
				+\frac{\ex[T]}{N-1}\mathbf{E}\right)
			\end{align*}
			Meanwhile, 
			\begin{align*}
				\ex[\mathbf{Q}_{\mathcal{T}}]\otimes \ex[\mathbf{Q}_{\mathcal{T}}]
				&= \left[\frac{\ex[T]}{N}\mathbf{A} + \left(1-\frac{\ex[T]}{N}\right)\mathbf{I}\right]\otimes \left[\frac{\ex[T]}{N}\mathbf{A} + \left(1-\frac{\ex[T]}{N}\right)\mathbf{I}\right]
				\\ &=  \left(\frac{\ex[T]}{N}\right)^2 \mathbf{A}\otimes \mathbf{A} + \frac{\ex[T]}{N}\left(1-\frac{\ex[T]}{N}\right) (\mathbf{I}\otimes \mathbf{A}+\mathbf{A}\otimes \mathbf{I}) + \left(1-\frac{\ex[T]}{N}\right)^2 \mathbf{I}\otimes\mathbf{I}.
			\end{align*}
			Thus, $\cov(\widehat{\bm{x}}_1)$ has the form
			\begin{align*}
				\cov(\widehat{\bm{x}}_1) 
				=&
				\frac{\ex[T]}{N}\frac{\ex[T]-N}{N(N-1)} \mathbf{A}\otimes \mathbf{A} + \frac{\ex[T]}{N}\left(1-\frac{\ex[T]-1}{N-1}\right)\mathbf{\Sigma}
				-\frac{\ex[T]}{N}\left(1-\frac{\ex[T]}{N}\right) (\mathbf{I}\otimes \mathbf{A}+\mathbf{A}\otimes \mathbf{I})
				\\
				&+\left(1-\frac{\ex[T]}{N}\right)\left(1-\frac{\ex[T]}{N-1}-1+\frac{\ex[T]}{N}\right)\mathbf{I}\otimes \mathbf{I} +\frac{\ex[T]}{N}\left(1-\frac{\ex[T]-1}{N-1}\right)\mathbf{E}
				\\
				=&
				\frac{\ex[T]}{N}\left(1-\frac{\ex[T]}{N}\right)\left(\frac{-1}{N-1}\mathbf{A}\otimes \mathbf{A}- (\mathbf{I}\otimes \mathbf{A}+\mathbf{A}\otimes \mathbf{I})+
				\left(1-\frac{1}{N-1}\right)\mathbf{I}\otimes \mathbf{I}
				+ \left(1+\frac{1}{N-1}\right)(\mathbf{\Sigma}+\mathbf{E})\right)
				\\
				=&
				\frac{\ex[T]}{N}\left(1-\frac{\ex[T]}{N}\right)
				\left(\frac{1}{N-1}\left(\mathbf{\Sigma}-\mathbf{A}\otimes \mathbf{A}+\mathbf{\Sigma}+\mathbf{E}-\mathbf{I}\otimes\mathbf{I}\right)+ \mathbf{\Sigma}+\mathbf{E} +
				\mathbf{I}\otimes\mathbf{I}-(\mathbf{I}\otimes  \mathbf{A}+\mathbf{A}\otimes \mathbf{I})
				\right)
				\,.
			\end{align*}
			If we fix $\ex[T]/N$, 
			the proportion of rows updated at each iteration, for large $N$ we get,
				\begin{align*}
					\cov(\widehat{\bm{x}}_1) =  \frac{\ex[T]}{N}\left(1-\frac{\ex[T]}{N}\right)\left( I'-II''-III'+\mathbf{E}\right) 
					+O\left(\frac{1}{N}\right)\,.
				\end{align*}
			The covariance will be “largest'' when $\ex[T]/N$ is around $\frac12$.

		\end{document}